\documentclass[11pt]{article}
\usepackage{amsmath,amssymb,amsfonts,amsxtra}
\usepackage{amsthm}
\usepackage{graphicx,psfrag}
\usepackage{booktabs}
\usepackage{subfigure}
\usepackage{multirow}
\usepackage{url}
\usepackage[usenames]{color}
\usepackage[colorlinks,linktocpage,linkcolor=blue]{hyperref}
\usepackage[margin=1.1in]{geometry}
\allowdisplaybreaks

\definecolor{darkred}{rgb}{.7,0,0}

\definecolor{green}{rgb}{0,0.7,0}

\newtheoremstyle{thmm}{1.5ex plus 1ex minus .2ex}{1.5ex plus 1ex minus.2ex}{\rmfamily}{}{\bfseries}{}{1em}{} \theoremstyle{thmm}
\newtheorem{theorem}{Theorem}[section]
\newtheorem{lemma}{Lemma}[section]

\newtheorem{remark}{Remark}[section]
\renewenvironment{proof}[1][Proof]{\noindent\textit{#1. }
}{\hfill$\square$}

\def\R{\mathbb{R}}
\def\C{{\mathbb C}}
\def\i{{\mathrm i}}
\def\d{{\mathrm d}}
\def\Omega{{\varOmega}}

\newcommand{\sech}{{\rm sech}}

\definecolor{mygray}{rgb}{0.9,0.9,0.9}
\definecolor{b}{rgb}{0.00,0.00,0.8}

\definecolor{darkred}{rgb}{0.55,0,0.0}

\title{\Large\bf High-order mass- and energy-conserving SAV--Gauss collocation finite element methods for the nonlinear Schr\"odinger equation}
\author{\normalsize 
Xiaobing Feng\thanks{Department of Mathematics, 
The University of Tennessee, Knoxville, TN 37996, U.S.A. 
Email address: xfeng@math.utk.edu). The work of this author
was partially supported by the NSF-grant DMS-1620168.}
\and 
Buyang Li\thanks{Department of Applied Mathematics, The Hong Kong Polytechnic University, Hung Hom, Hong Kong. \linebreak Email address: buyang.li@polyu.edu.hk, maisie.ma@connect.polyu.hk. 
The work of B. Li was partially supported by an internal grant of the univeristy (project code: ZZKQ) and the work of S. Ma was partially supported by a Hong Kong RGC grant (Project No. 15300817). 
}
\and 
Shu Ma\footnotemark[2]
}

\date{}

\begin{document}

\maketitle

\vspace{-10pt}

\begin{abstract}
A family of arbitrarily high-order fully discrete space-time finite element methods are proposed for the nonlinear Schr\"odinger equation based on the scalar auxiliary variable formulation, which consists of a Gauss collocation temporal discretization and the finite element spatial discretization. The proposed methods are proved to be well-posed and conserving both mass and energy at the discrete level. 
An error bound of the form $O(h^p+\tau^{k+1})$ in the $L^\infty(0,T;H^1)$-norm is established, where $h$ and $\tau$ denote the spatial and temporal mesh sizes,  respectively, and $(p,k)$ is the degree of the space-time finite elements. 
Numerical experiments are provided to validate the theoretical results on the convergence rates and conservation properties. The effectiveness of the proposed methods in preserving the shape of a soliton wave is also demonstrated by numerical results. 
\\ 

\end{abstract}



\section{Introduction}\label{sec-1}

This paper is concerned with the development and analysis of high-order fully discrete numerical methods for the following 
initial-boundary value problem of the nonlinear Schr\"odinger (NLS) equation:
\begin{subequations}\label{pde}
\begin{alignat}{2}
\i \partial_t u - \Delta u - f(|u|^2)u &=0 &&\qquad\mbox{in}\,\,\,\Omega\times(0,T],
\label{pde1}\\
u&=0 &&\qquad\mbox{on}\,\,\,\partial\Omega\times(0,T],\\
u &=u_0&&\qquad \mbox{in}\,\,\,\Omega\times\{0\},
\end{alignat}
\end{subequations} 
where $\Omega\subset\R^d$ is a polygonal or polyhedral domain with boundary $\partial\Omega$, and $u:\Omega\rightarrow\C$ is a complex-valued function, with $\i=\sqrt{-1}$, and $f:\R_+\to \R$ is the derivative of some function $F:\R_+\to \R$. The best known examples are 
\begin{align}\label{nonlinearity-f}
f(s)=\pm s^{\frac{q-1}{2}}
\quad\mbox{and}\quad 
F(s)=\pm\frac{2}{q+1}s^{\frac{q+1}{2}}, 
\quad\mbox{with}\quad q>1,
\end{align}
where the ``$-$'' and ``$+$'' cases are often referred to as defocusing and focusing models, respectively. In the focusing case, the solution will blow up in $L^\infty(\Omega)$ within finite time when the initial energy is negative; see \cite{Bourgain1999Global,Tao2006Nonlinear}. The NLS equation \eqref{pde} arises from many applications in physics and engineering, and is one of the fundamental equations in mathematical physics \cite{Bourgain1999Global,Tao2006Nonlinear,Zabusky1965Interaction, Pelinovsky1996Nonlinear,
Sch1996Traveling}. 

It is well known that the solutions of \eqref{pde} conserve the mass and energy in the sense that
for all $t\geq 0$
\begin{align}
 \frac{\d}{\d t}\int_\Omega |u|^2\d x &=0, \qquad \mbox{(mass conservation)} \label{mass} \\[5pt]
\frac{\d}{\d t}\int_\Omega \Big(\frac12|\nabla u|^2 -  \frac12F(|u|^2) \Big)\d x &=0. \qquad
 \mbox{(energy conservation)}  \label{energy}
\end{align}
The development of numerical methods that can retain these conservation properties in numerical solutions is important for long-time numerical simulation, and therefore has been one of the research focuses in numerical approximation to the NLS equation.

There exists a large amount of literature on numerical solutions and numerical analysis of the NLS
equation.  Delfour et al. \cite{Delfour1981Finite} proposed a second-order modified Crank--Nicolson scheme for the NLS equation 
to preserve both mass and energy conservations. The construction of this method was motivated by a similar scheme for the Klein--Gordon equation in \cite{Strauss1978Numerical}.  
 Sanz-Serna \cite{Sanz-Serna1984Methods} extended the modified Crank--Nicolson time-stepping 
scheme to the NLS equation with more general nonlinear term, and established an optimal order error 
estimate for the fully discrete finite element discretization. For the cubic NLS equation, more delicate 
results on the existence, uniqueness and convergence of numerical solutions of the modified Crank--Nicolson scheme were proved by Akrivis et al. in \cite{Akrivis1991On}. 
The modified Crank--Nicolson scheme has been widely used in computation and was combined 
with spatial finite difference methods to solve the NLS and Gross--Pitaevskii equations in \cite{Akrivis1993,Antoine2013Computational,Bao-Cai-2013,bao2013numerical},
  and with various spatial discretization methods to approximate the NLS in \cite{Wang201New, Hong2006Globally,Liu2019Accuracy, Xu2005Local, Gao2011Fourth}. 
Besides the modified Crank--Nicolson scheme, a second order explicit leapfrog scheme  
for the NLS equation was proposed by Sanz-Serna and Manoranjan in \cite{Sanz-Serna1983A}. This  
scheme was proved to preserve mass conservation at the discrete level. Optimal rates of convergence 
of the fully discrete leapfrog finite element method was established in \cite{Sanz-Serna1984Methods}. 
To avoid solving nonlinear systems, a linearly implicit leapfrog scheme was proposed by Fei et al. in \cite{Fei1995Numerical} for the NLS equation, and a linearly implicit relaxation scheme was 
proposed by Besse in \cite{Besse2004Relaxation}. Both schemes preserve mass and energy conservations 
at the discrete level. More recently, Feng et al. \cite{Feng_Liu_Ma2019} constructed a class of second-order mass- and energy-conserving schemes for the NLS equation, including the modified Crank--Nicolson method, the implicit leapfrog method and a class of modified backward differentiation formulae as special cases.

To the best of our knowledge, all the existing mass- and energy-conserving methods have only 
second-order accuracy in time. No higher-order time-stepping schemes, 
which conserve both mass 
and energy, have been reported in the literature. Moreover, the existing error estimates for 
nonlinearly implicit schemes for the NLS equation generally require certain grid-ratio conditions. 
The standard grid-ratio conditions in the literature are $\tau=o(h^{\frac{d}{4}})$ for the cubic NLS equation and $\tau=o(h^{\frac{d}{2}})$ for general nonlinearity, where $h$ and $\tau$ denote the 
spatial and temporal mesh sizes. Karakashian and Makridakis \cite{1998-Karakashian-Makridakis,1999-Karakashian-Makridakis} proposed some continuous and 
 discontinuous space-time Galerkin finite element methods for the cubic 
NLS equation and proved optimal-order convergence under a weaker grid-ratio condition $\tau^{k-1}|\ln h|\rightarrow 0$ in two dimensions, where $k\ge 2$ is the degree of finite elements in time. 
For the defocusing cubic NLS equation (or the focusing cubic NLS equation with sufficiently small initial data), using the energy conservation of the numerical scheme, error estimates were established
 without grid-ratio condition in \cite{Gong-Wang-Wang-2017,Wang-Guo-Xu-2013}. For general 
 nonlinearity (possibly focusing), Wang \cite{Wang201New} established an error 
 estimate for 
 a linearized semi-implicit scheme without grid-ratio condition; Henning and Peterseim \cite{Henning-Peterseim-2017} established an error estimate for the nonlinearly implicit Crank--Nicolson finite element method without grid-ratio condition. Both \cite{Wang201New} and \cite{Henning-Peterseim-2017} used an error splitting technique in which they proved boundedness of the numerical solutions by establishing an $L^\infty$-norm error estimate between the fully discrete 
 and the semidiscrete-in-time numerical solutions. 
 The error splitting technique allows to avoid grid-ratio conditions in using the inverse inequality.

The objective of this paper is to develop a family of arbitrarily higher-order mass- and energy-conserving fully discrete space-time finite element methods based 
on the scalar auxiliary variable (SAV) formulation of the NLS equation, and to establish 
the existence, uniqueness and optimal order convergence of numerical solutions 
without grid-ratio condition.   
Two key ideas are utilized in our construction of the method. First, the SAV reformulation of the NLS equation is used. This approach was introduced in \cite{2018-Shen-Xu-Yang,2019-Shen-Xu-Yang} as an enhanced version of the invariant energy quadratization (IEQ) approach  \cite{Yang-2016,Yang-Ju-2017-1,Yang-Ju-2017-2,Yang-ZWS}, for developing energy-decay methods for dissipative (gradient flow) systems. Here we adapt the SAV approach to the dispersive 
NLS equation, and the SAV reformulation is essential to enable our methods to maintain the energy conservation property at the discrete level. 
Second, the Gauss collocation method is used 
for time discretization in the SAV formulation of the NLS equation. The method can be viewed as an efficient implementation of the space-time finite element methods for 
the SAV formulation with Gauss quadrature in time. The Gauss collocation method was combined with IEQ and SAV to 
preserve energy decay in solving phase field equations in \cite{Akrivis-Li-Li-2019,Gong-Zhao-2019,Gong-Zhao-Wang-2020}. 
We adopt this method here to preserve mass conservation without affecting the energy 
conservation structure of the SAV formulation. 

The SAV formulation introduces new difficulties to error analysis for the NLS equation due to 
the presence of $\partial_t u$ in the equation of $r$, see equation \eqref{IEQ-r}, which leads to a consistency error of sub-optimal order in time and introduces new difficulty in obtaining the stability estimate. As far as we know, rigorous analysis for convergence of numerical methods based on SAV formulations has not been done for any wave equation so far. These difficulties are overcome by combining three techniques. First, inspired by the error analysis of Karakashian and Makridakis \cite{1999-Karakashian-Makridakis}, our proof makes use of properties of the 
Legendre polynomials on each interval $I_n$, rewriting the Gauss collocation method into a 
space-time Galerkin finite element method, which makes it easier to choose suitable test 
functions in the error estimation. Second, we introduce a temporal Ritz projection and use a super-approximation result of the temporal local $L^2$ projection to eliminate the sub-optimal temporal consistency error caused by $\partial_t u$ in the equation of $r$; see Remark \ref{Remark:consistency} (hence, the proof of optimal order in Theorem \ref{THM:consistency} is one of our main contributions).
Third, we estimate the time derivative of the error 
in $H^{-1}(\Omega)$ with a duality argument following an $H^1$-norm error estimate. As a result, we obtain an optimal-order $H^1$-norm error estimate in the end. We prove the existence, uniqueness 
and optimal-order convergence of numerical solutions based on Schaefer's  fixed point 
theorem in an $L^\infty$-neighborhood of the exact solution. This avoids grid-ratio conditions 
for the NLS equation with general nonlinearity. 

The rest of this paper is organized as follows. In Section \ref{sec-2}, we present the SAV 
reformulation of the NLS equation and introduce our SAV space-time Gauss collocation finite 
element method. 
In Section \ref{sec-3}, we first present an integral reformulation of the proposed numerical 
method and  then establish its mass and energy conservation properties. We also derive a 
consistency error estimate for the proposed method, which is vitally used to prove an error 
estimate in the subsequent section. 
In Section \ref{sec-4}, we first establish the well-posedness of the numerical method and then prove an error bound of the form $O( h^p + \tau^{k+1} )$ in the energy norm, where $\tau$ and $h$ denote the temporal and spatial mesh sizes, respectively, with $(p,k)$ denoting the degree of polynomials in 
the space-time finite element method. 
Finally, in Section \ref{sec-5}, we present a few numerical experiments to validate the  theoretical
results, and to demonstrate the effectiveness of the proposed method in preserving the shape 
of a soliton wave. 

Throughout this paper,  unless stated otherwise, $C$ will be used to denote a generic positive 
constant which is 
independent of $\tau$, $h$, $n$ and $N$, but may depend on $T$ and the regularity of solution.

\section{Formulation of the SAV--Gauss collocation finite element method}\label{sec-2}

In this section, we construct a Gauss collocation finite element method based on the SAV reformulation of the NLS equation.

\subsection{Function spaces}
Let $H^k(\Omega)$, $k\ge 0$, be the conventional complex-valued Sobolev space of functions on $\Omega$, and denote 
$$L^2(\Omega)=H^0(\Omega)
\quad\mbox{and}\quad
H^1_0(\Omega)=\{v\in H^1(\Omega): v=0 \mbox{ on}\,\,\partial\Omega\}.
$$ 
We denote by $(\cdot,\cdot)$ and $\|\cdot\|$ the inner product and norm of the complex-valued Hilbert space $L^2(\Omega)$, respectively, defined by  
$$
(u,v):=\int_\Omega u\,\overline{v}\,\d x 
\quad\mbox{and}\quad
\|u\|:= \sqrt{(u,u)}. 
$$
For $m,s\ge 0$ and $1\le p\le \infty$, the notation $W^{m,p}(0,T;H^s(\Omega))$ stands for the space-time Sobolev space of functions which are $W^{m,p}$ in time and $H^s$ in space; see \cite[Chapter 5.9]{Evans-PDE}. We abbreviate the norms of $H^s(\Omega)$ and $W^{m,p}(0,T;H^s(\Omega))$ as $\|\cdot\|_{H^k}$ and $\|\cdot\|_{W^{m,p}(I_n;H^s)}$, respectively, omitting the dependence on $\Omega$ in the subscripts. 

\subsection{The SAV reformulation of \eqref{pde}}\label{sec-2.1}

The SAV formulation of the NLS equation (cf. \cite{2019-Shen-Xu-Yang}) introduces a scalar auxiliary variable   
\begin{align}
r = \sqrt{ \mbox{$\int_\Omega \frac12  $} F(|u|^2) \d x +c_0}
\quad\mbox{with}\quad
g(u)=\frac{f(|u|^2)}{ \sqrt{ \mbox{$\int_\Omega $} \frac12 F(|u|^2) \d x +c_0}},
\end{align}
with a positive $c_0$ (which guarantees that the function $r$ has a positive lower bound), and reformulate \eqref{pde} as 
\begin{subequations}\label{IEQ} 
\begin{alignat}{2} 
\i\partial_tu - \Delta u - r g(u)u &= 0 &&\qquad \mbox{in}\,\,\,\Omega\times(0,T], \label{IEQ-u}\\ 
\frac{\d r}{\d t} &= {\rm Re} \big(\mbox{$\frac{1}{2}$}g(u)u , \partial_tu\big) 
&&\qquad \mbox{in}\,\,\,\Omega\times(0,T], \label{IEQ-r}\\ 
u &=0 &&\qquad \mbox{on}\,\,\,\partial\Omega\times(0,T],\\ 
u =u_0, \quad r &=r_0 &&\qquad\mbox{in}\,\,\,\Omega\times \{0\}, 
\end{alignat} 
\end{subequations} 
where $r_0=\sqrt{ \mbox{$\int_\Omega \frac12$} F(|u_0|^2) \d x+c_0}$. The mass and energy conservation in the SAV formulation are
\begin{align}\label{SAV-mass}
\frac{\d}{\d t}\int_\Omega |u|^2\d x =0, \qquad \mbox{and}\qquad 
\frac{\d}{\d t}\bigg(\frac12\int_\Omega |\nabla u|^2 \d x -  r^2 + c_0\bigg) =0.
\end{align}

\subsection{Space-time finite element spaces}
Let $\mathcal{T}_h$ be a shape-regular and quasi-uniform triangulation of $\Omega$ with mesh size $h\in (0,1)$ and $\{t_n\}_{n=0}^N$ be a uniform partition of $[0,T]$ with the time step size $\tau\in (0,1)$, where $N$ is a positive integer and hence $\tau=\frac{T}{N}$. 
For an integer $p\ge 1$ we denote by $\mathbb{Q}^p$ the space of complex-valued polynomials of degree $\le p$ in space, and we denote by $S_h$ the complex-valued Lagrange finite element space subject to the triangulation of $\Omega$, defined by  
\begin{align*}
S_h&=\bigl\{v\in C(\overline\Omega) :\, 
v|_{K}\in \mathbb{Q}^p\,\,\,\mbox{for all}\,\,\, K\in\mathcal{T}_h,\,\,\, v=0\,\,\,\mbox{on}\,\,\,\partial\Omega 
\bigr\} ,
\end{align*}
where $C(\overline\Omega)$ denotes the space of complex-valued uniformly continuous functions on $\Omega$.
Then $S_h$ is a complex Hilbert spaces with the inner product $(\cdot,\cdot)$ and norm $\|\cdot\|$. 

For an integer $k\ge 1$, let $\mathbb{P}^k$ denote the space of real-valued polynomials of degree $\le k$ in $t$. 
For a Banach space $X$, such as $X=L^2(\Omega)$ or $X=S_h$, we define the following tensor-product space:
\begin{align} 
\mathbb{P}^k\otimes  X:=
\mbox{span}\Bigl\{ p(t)\phi(x):\, p\in \mathbb{P}^k, \, \phi\in X \Big\}
 	=\Big\{ \mbox{$\sum_{j=0}^k$} t^j \phi_j: \, \phi_j\in  X \Big\} . 
\end{align}
Moreover, let $P_h:L^2(\Omega)\rightarrow S_h$ denote the $L^2$ projection operator defined by 
$$
\bigl(w-P_hw,v_h \bigr)=0\quad\forall\, v_h\in S_h,\,\,\forall\, w\in L^2(\Omega).
$$
The following stability properties are well-known (cf. \cite{Brenner-Scott}):
\begin{subequations} 
\begin{alignat}{2}
\|P_hw\| &\le \|w\|  &&\qquad\forall\, w\in L^2(\Omega) ,\\
\|P_hw\|_{H^1} &\le C\|w\|_{H^1} &&\qquad\forall\, w\in H^1_0(\Omega) , 
\label{Ph-H1-stability}
\end{alignat}
\end{subequations}
where $C$ depends only on the shape-regularity and quasi-uniformity of the mesh. 

We also introduce the global space-time finite element spaces
\begin{align} \label{globa-space_1}
X_{\tau,h}&=\{v_h\in C([0,T];S_h): v_h|_{I_n}\in \mathbb{P}^k\otimes S_h\,\,\,\mbox{for}\,\,\,n=1,\dots,N \},\\
Y_{\tau,h}&=\{q_h\in C([0,T]): q_h|_{I_n}\in \mathbb{P}^k \,\,\,\mbox{for}\,\,\,n=1,\dots,N \}.
\label{global_space_2}
\end{align}

\subsection{SAV--Gauss collocation finite element method}
Let $c_j$ and $w_j$, $j=1,\dots,k$, be the nodes and weights of the $k$-point Gauss quadrature
rule in the interval $[-1,1]$ (see \cite[Table 3.1]{Shen2011Spectral}), and let $t_{nj}=t_{n-1}+(1+c_j)\tau/2$, $j=1,\dots,k$ denote the Gauss points in the interval $I_n=[t_{n-1},t_n]$.
We define the following Gauss collocation finite element method for \eqref{IEQ}. 

\medskip
\noindent
{\bf Main Algorithm} 
 
{\em Step 1:} Set $u_{h}^{0}:=I_{h}u_{0}$ and $r_{h}^{0}:=r_{0}$, where $I_h$ is the Lagrange interpolation operator onto the finite element space. Determine $(u_h,r_h)\in X_{\tau,h}\times Y_{\tau,h}$ by the following two steps.

\smallskip
{\em Step 2:} 
For $n=1,2,\cdots, N$, define $\bigl\{(u_h(t_{nj}) , r_h(t_{nj}) ) \bigr\}_{j=1}^k\subset S_h \times \R $ by solving recursively (in $n$) the following nonlinear (algebraic) system: 
\begin{subequations}\label{GL}
\begin{align}\label{GL-uh}
\i\bigl(\partial_t u_h(t_{nj}), v_h \bigr)
&+\bigl(\nabla u_h(t_{nj}) ,\nabla v_h \bigr) \\  
& - \bigl(r_h(t_{nj}) g(u_h (t_{nj})) u_h(t_{nj}) ,v_h \bigr) = 0,  \qquad\forall\, v_h\in S_h, \nonumber \\
\partial_t r_h(t_{nj})  &=
\frac{1}{2}{\rm Re} \big(g(u_h(t_{nj}))u_h(t_{nj}),\partial_t u_h(t_{nj}) \big) ,
\label{GL-rh} \\
u_h(t_{n-1}) &=u_h^{n-1} \quad\mbox{and}\quad
r_h(t_{n-1}) =r_h^{n-1} \label{GL-initial}.
\end{align}
\end{subequations}

\smallskip
{\em Step 3:} Set $u_h^{n}:=u_h(t_n)$ and $r_h^{n}:=r_h(t_n)$.

\medskip  
\begin{remark}
{\upshape  
(a) We note that in \eqref{GL-uh} and \eqref{GL-rh}, $\partial_t u_h(t_{nj})=\partial_t u_h(t)|_{t=t_{nj}}$ and 
$\partial_t r_h(t_{nj})=\partial_t r_h(t)|_{t=t_{nj}}$. Main Algorithm really computes  $\bigl\{(u_h(t_{nj}) , r_h(t_{nj}) )\bigr\}_{j=1}^k$ for each $n\geq 1$, however, 
since any $k$th order polynomial on $I_n$ is uniquely determined by its initial value 
at $t_{n-1}$ and its values at the $k$ Gauss points $t_{nj}$, $j=1,\dots,k$, then the Gauss-point values generated by
Main Algorithm uniquely determine the pair $(u_h,r_h)\in X_{\tau,h}\times Y_{\tau,h}$.

(b) Each of \eqref{GL-uh} and \eqref{GL-rh} consists of nonlinear algebraic equations, 
note that the test function $v_h$ can be different for different $j$, and one ``side/initial condition" is prescribed for 
each of $u_h$ and $r_h$ for each $n$. The number of equations imposed is the same
as the degree of freedoms which equals the dimension of the space $\mathbb{P}^k\otimes S_h$
for each $n$.   

(c) Main Algorithm can be obtained by applying the Gauss quadrature rule (in time) to a 
(continuous) space-time finite element method for \eqref{IEQ}. See Section \ref{sec-3.1}.

(d) In practical computation, we solve the solution of the nonlinear scheme \eqref{GL} by Newton's method: For given $\bigl\{(u_h^{\ell-1}(t_{nj}) , r_h^{\ell-1}(t_{nj}) ) \bigr\}_{j=1}^k\subset S_h \times \R $, find 
$$\bigl\{(u_h^{\ell}(t_{nj}) , r_h^{\ell}(t_{nj}) ) \bigr\}_{j=1}^k\subset S_h \times \R $$ satisfying the linearized equations 
\begin{subequations}\label{GL-Newton}
\begin{align} 
\i\bigl(\partial_t u_h^{\ell}(t_{nj}), v_h \bigr)
&+\bigl(\nabla u_h^{\ell}(t_{nj}) ,\nabla v_h \bigr) \\  
=&  \bigl(r_h^{\ell}(t_{nj}) g(u_h^{\ell-1} (t_{nj})) u_h^{\ell-1}(t_{nj}) ,v_h \bigr) \nonumber\\
&
+\bigl(r_h^{\ell-1}(t_{nj}) g_1(u_h^{\ell-1}(t_{nj}))(u_h^{\ell}(t_{nj})-u_h^{\ell-1}(t_{nj})) ,v_h \bigr) 
\nonumber\\
&
+\bigl(r_h^{\ell-1}(t_{nj})  g_2(u_h^{\ell-1}(t_{nj}))(\bar u_h^{\ell}(t_{nj})-\bar u_h^{\ell-1}(t_{nj})),v_h \bigr),  \qquad\forall\, v_h\in S_h, \nonumber \\
\partial_t r_h^{\ell}(t_{nj})  = &
\frac{1}{2}{\rm Re} \big(g(u_h^{\ell-1}(t_{nj})) u_h^{\ell-1}(t_{nj}),\partial_t u_h^{\ell}(t_{nj}) \big) \\
&+
\frac{1}{2}{\rm Re} \big(g_1(u_h^{\ell-1}(t_{nj}))(u_h^{\ell}(t_{nj})-u_h^{\ell-1}(t_{nj})) ,\partial_t u_h^{\ell-1}(t_{nj}) \big) \nonumber\\
&+
\frac{1}{2}{\rm Re} \big(g_2(u_h^{\ell-1}(t_{nj}))(\bar u_h^{\ell}(t_{nj})-\bar u_h^{\ell-1}(t_{nj})) ,\partial_t u_h^{\ell-1}(t_{nj}) \big) \nonumber\\
u_h^{\ell}(t_{n-1}) =& u_h^{n-1} \quad\mbox{and}\quad
r_h^{\ell}(t_{n-1}) =r_h^{n-1},
\end{align}
\end{subequations}
where 
$$
g_1(u):= \partial_u [g(u)u]
\quad\mbox{and}\quad
g_2(u):= \partial_{\bar u} [g(u)u],
$$
and $\partial_{\bar u}$ denotes the differentiation with respect to $\bar u$ in the expression of 
$$
g(u)u
=\frac{f(u\bar u) u }{ \sqrt{ \mbox{$\int_\Omega $} \frac12 F(u\bar u) \d x +c_0}} . 
$$
The iteration in $\ell$ is set to stop when the desired tolerance error is achieved. 
}

\end{remark}


\section{Conservation, stability and consistency analysis}\label{sec-3}

\subsection{A reformulation of scheme \eqref{GL-uh}--\eqref{GL-rh}}\label{sec-3.1}
In this subsection, we present several integral identities, including a reformulation of Main Algorithm, and inequalities related to the proposed numerical method. These identities and inequalities will be used in the subsequent analysis of existence, uniqueness and convergence of numerical solutions. 

Consider the interval $I_n=[t_{n-1},t_n]$, then we define $P_\tau^n: L^2(I_n;L^2(\Omega))\rightarrow \mathbb{P}^{k-1}\otimes L^2(\Omega)$ to be the $L^2$ projection defined by 
\begin{align}
\label{l2-pro}
\int_{I_n} (u - P_\tau^n u, v)\,\d t=0\quad\forall\,v \in \mathbb{P}^{k-1} \otimes L^2(\Omega) .     
\end{align}
Thus $u - P_\tau^n u$ is orthogonal to all temporal polynomials of degree $\le k-1$, which means that if $u\in \mathbb{P}^{k} \otimes L^2(\Omega)$ then 
\begin{align}\label{l2-pro-rp}
u - P_\tau^n u = \phi_{n-1} L_k,    
\end{align}
where $\phi_{n-1} \in L^2(\Omega)$ and 
\begin{align}\label{def-Lk-t}
L_k (t) := \widehat{L}_k \bigg(\frac{2t - t_{n-1} - t_n}{\tau}\bigg)
\end{align}
is the shifted Legendre polynomial (orthogonal to polynomials of lower degree on $I_n$). The temporal $L^2$ projection  operator $P_\tau^n$ has the following approximation property (cf. \cite{Canuto2007Spectral}): 
\begin{align}\label{approx-l2-pro}
&\max_{t\in I_n} \|v - P_\tau^nv\|_{X}    
\le C\tau^m \max_{t\in I_n} \|\partial_t^m v\|_{X} ,
\quad 0\le m\le k, 
\end{align}  
for all $v\in C^k([0,T];X)$, where $X=\R$ or $X=H^s(\Omega)$ for some $s\in\R$. 

Since the $k$-point Gauss quadrature holds exactly for polynomials of degree $2k-1$  (cf. \cite[p.~222]{Golub1969Calculation}), and the Gauss points $t_{nj}$, $j=1,\dots,k$, are the roots of the Legendre polynomial $L_k(t)$ (cf. \cite[p.~33]{Kopriva2009Implementing}), it follows that the following two identities hold: 
\begin{alignat}{2}
\int_{I_n} v(t)\d t
&= \frac{\tau}{2} \sum_{j=1}^{k} v(t_{nj}) w_{j} 
&&\qquad \forall\, v\in \mathbb{P}^{2k-1} \otimes S_h ,
\label{Gauss-integral} \\
v(t_{nj})
&= P_\tau^n v(t_{nj}) &&\qquad \forall\, v\in \mathbb{P}^{k} \otimes S_h .
\label{Ptau-tnj}
\end{alignat}

By choosing $v_h= \frac{\tau}{2} v_h(t_{nj}) w_j$ in \eqref{GL-uh} and summing up the results for $j=1,\dots,k$, and using \eqref{Gauss-integral}--\eqref{Ptau-tnj} in the first two terms, we obtain the following integral identity:  
\begin{align}\label{uh-test-vh}
&\int_{I_n} \i \bigl(\partial_tu_h,v_h \bigr) \d t 
+ \int_{I_n} (\nabla P_\tau^n u_h ,\nabla v_h) \d t \\
&\qquad 
- \frac{\tau}{2} \sum_{j=1}^k w_j (r_h(t_{nj})g(u_h(t_{nj})) u_h(t_{nj}) , v_h(t_{nj})) = 0  
\qquad \forall\, v_h\in  \mathbb{P}^{k}\otimes S_h . \nonumber
\end{align}
Similarly, multiplying \eqref{GL-rh} by $\frac{\tau}{2} q_h(t_{nj})w_j$ and summing up the results for $j=1,\dots,k$, and using \eqref{Gauss-integral} in the first term, we have 
\begin{align}\label{rh-test-qh}
&\int_{I_n} \partial_tr_hq_h \d t 
=\frac{\tau}{2} \sum_{j=1}^k\frac{w_j}{2}{\rm Re} \big(g(u_h(t_{nj})) u_h(t_{nj}),\partial_t u_h(t_{nj})\, q_h(t_{nj})\big)
\quad\forall\, q_h\in  \mathbb{P}^{k} . 
\end{align}
\eqref{uh-test-vh}--\eqref{rh-test-qh} provides a reformulation of Main Algorithm. 
	It will be crucially used to show mass and energy conservations,
as well as existence, uniqueness and convergence of numerical solutions. 

From \eqref{l2-pro-rp} we get
\begin{align*}
\|\phi_{n-1} \|
&= \frac{1}{|L_k(t_{n-1})|}
\|u_h(t_{n-1})-P_\tau^nu_h(t_{n-1})\| \\
&\le
C\|u_h(t_{n-1})\|
+C\bigg(\frac{1}{\tau}\int_{I_n}\|P_\tau^nu_h(t)\|^2 \d t\bigg)^{\frac12} ,
\end{align*}
where we have used the inverse inequality in time. Thus, by using \eqref{l2-pro-rp} again, we obtain the following inequality: 
\begin{align}\label{L2-uh-Phuh}
\int_{I_n} \|u_h\|^2 \d t
&\le 
C\int_{I_n} \|P_\tau^nu_h\|^2 \d t
+C\tau \|u_h(t_{n-1})\|^2 
\qquad\forall\, u_h\in \mathbb{P}^k\otimes S_h. 
\end{align}
By using the two identities \eqref{Gauss-integral}--\eqref{Ptau-tnj}, one can also prove the following inequality:
\begin{align}\label{sum-inequality}
\frac{\tau}{2} \sum_{j=1}^k w_j
\|v_h(t_{nj})\|^2
=\int_{I_n} \|P_\tau^nv_h(t)\|^2\d t
\le \int_{I_n} \|v_h(t)\|^2\d t 
\quad\forall\,
v_h\in \mathbb{P}^k\otimes S_h.
\end{align}
The inequalities \eqref{L2-uh-Phuh}--\eqref{sum-inequality} will be frequently used in the subsequent error analysis. 

\subsection{Mass and energy conservation properties}

In this subsection, we prove the following conservation properties of the numerical solution, 
which comprise of the first main theorem of this paper.  

\begin{theorem}\label{THM:conservation}
Let $(u_h,r_h) \in X_{\tau,h}\times Y_{\tau,h}$ be a solution of Main Algorithm, 
then the following mass and energy conservations hold: 
\begin{align*}
\begin{aligned}
\frac12\|u_h(t_n)\|^2 &= \frac12\|u_h(t_{0})\|^2 &&\mbox{for } n\geq 1, \\
\frac12  \|\nabla u_h(t_n)\|^2 - |r_h(t_n)|^2 + c_0 &= 
\frac12  \|\nabla u_h(t_{0})\|^2 - |r_h(t_{0})|^2  + c_0 &&\mbox{for } n\geq 1.
\end{aligned}
\end{align*}
\end{theorem}

\begin{proof}
Setting $v_h=u_h\in \mathbb{P}^{k}\otimes S_h$ in \eqref{uh-test-vh} and taking the imaginary part yield 
\begin{align}\label{mass-conserv-1}
 {\rm Im} \int_{I_n} \i \bigl(\partial_t u_h, u_h \bigr) \d t 
&= - {\rm Im} \int_{I_n} (\nabla P_\tau^n u_h ,\nabla u_h) \d t   \\
&\quad + {\rm Im} \bigg[\frac{\tau}{2} \sum_{j=1}^k w_j 
(r_h(t_{nj})g(u_h(t_{nj})), |u_h(t_{nj})|^2)\bigg] =0, \nonumber
\end{align}
where we have used the definition of the projection operator $P_\tau^n$, which implies 
$$
{\rm Im} \int_{I_n} (\nabla P_\tau^n u_h ,\nabla u_h) \d t
= {\rm Im} \int_{I_n} (\nabla P_\tau^n u_h ,\nabla P_\tau^n u_h) \d t
=0 .
$$
Then the mass conservation follows from \eqref{mass-conserv-1} and the identity 
$$
{\rm Im}\int_{I_n} \i \bigl( \partial_tu_h,u_h \bigr) \d t
= \frac12\|u_h(t_n)\|^2 - \frac12\|u_h(t_{n-1})\|^2 .
$$

Alternatively, setting $v_h=\partial_tu_h$ and $q_h=2 r_h$ in \eqref{uh-test-vh} and \eqref{rh-test-qh}, respectively, and taking the real parts yield
\begin{align} 
&{\rm Re}\int_{I_n} (\nabla P_\tau^n u_h, \nabla\partial_tu_h) \d t 
= \frac{\tau}{2} \,{\rm Re}\sum_{j=1}^kw_j (r_h(t_{nj})g(u_h(t_{nj})) u_h(t_{nj}) , \partial_tu_h(t_{nj}))   \\
&|r_h(t_n)|^2 - |r_h(t_{n-1})|^2
= \frac{\tau}{2} \,{\rm Re}\sum_{j=1}^k w_j  \big(r_h(t_{nj})g(u_h(t_{nj})) u_h(t_{nj}),\partial_tu_h(t_{nj}) \big) . \label{energy-r-1}
\end{align}
Since 
\begin{align*}
{\rm Re}\int_{I_n} (\nabla P_\tau^n u_h, \nabla\partial_tu_h) \d t 
&={\rm Re}\int_{I_n} (P_\tau^n  \nabla u_h, \nabla\partial_tu_h) \d t 
= {\rm Re}\int_{I_n} (\nabla u_h, \nabla\partial_tu_h) \d t \\
& =
\frac12 \|\nabla u_h(t_n)\|^2 - \frac12 \|\nabla u_h(t_{n-1})\|^2 ,
\end{align*}
it follows that 
\begin{align}\label{energy-u-1}
&\frac12 \|\nabla u_h(t_n)\|^2 - \frac12 \|\nabla u_h(t_{n-1})\|^2 \\
&\qquad = \frac{\tau}{2} \,{\rm Re}\sum_{j=1}^k w_j \bigl( r_h(t_{nj})g(u_h(t_{nj})) u_h(t_{nj}) , \partial_t u_h(t_{nj})\bigr) . 
\nonumber
\end{align}
Subtracting \eqref{energy-r-1} from \eqref{energy-u-1} yields 
\begin{align}\label{con-energy}
\frac12  \|\nabla u_h(t_n)\|^2 -  |r_h(t_n)|^2 = 
\frac12  \|\nabla u_h(t_{n-1})\|^2 - |r_h(t_{n-1})|^2 
\quad \mbox{for}\,\,\, n\ge 1. 
\end{align}
Thus, the energy conservation holds. The proof is complete.
\end{proof}

\subsection{An upper bound of mass at internal stages}\label{section:upper-bound}

In this subsection, we prove that the average mass of numerical solutions at internal stages has an upper bound unconditionally (independent of the regularity of solutions). This property furthermore strengthens the stability of numerical solutions when the exact solution is not smooth (for example, close to blow up). 

\begin{theorem}\label{THM:mass-bound}
Let $(u_h,r_h) \in X_{\tau,h}\times Y_{\tau,h}$ be a solution of Main Algorithm,
then the following inequalities hold: 
\begin{subequations}
\begin{align}\label{interior_est_1}
\max_{1\le n\le N} \frac{1}{\tau} \int_{I_n} \|P_\tau^nu_h\|^2 \d t 
&\le \|u_h(0)\|^2 , \\
\max_{1 \le n\le N}\max_{1\le j\le k}\|u_h(t_{nj})\| &\le C\|u_h(0)\| ,  \label{interior_est_2}
\end{align}
\end{subequations}
where $C$ is a constant independent of $\tau$, $h$ and the regularity of the solution. 
\end{theorem}
\begin{proof}
By the definition of the temporal $L^2$ projection $P_\tau^n$,  we get
\begin{align}\label{L2L2Pu-M}
\int_{I_n} &\|P_\tau^n u_h(t) \|^2 \d t 
={\rm Re} \int_{I_n} (u_h(t), P_\tau^nu_h(t)) \d t \\
&={\rm Re}\int_{I_n} (u_h(t_{n-1}), P_\tau^n u_h(t_{n-1})) \d t \nonumber  \\
&\quad +{\rm Re}\int_{I_n} 
\int_{t_{n-1}}^t [(\partial_s u_h(s), P_\tau^n u_h(s))
+(u_h(s), \partial_s P_\tau^nu_h(s))] \d s\d t \nonumber \\
&={\rm Re}\, (u_h(t_{n-1}), P_\tau^n u_h(t_{n-1}))\tau  
+ {\rm Re}\int_{I_n} (\partial_t u_h(t), (t_n-t)P_\tau^n u_h(t)) \d t  \nonumber \\
&\quad + {\rm Re}\int_{I_n} (u_h(t), (t_n-t)\partial_t P_\tau^n u_h(t)) \d t \nonumber\\
&=: J_1+J_2+J_3, \nonumber
\end{align}
where we have interchanged the order of integration in deriving the second to last equality. 
Using H\"older's and Young's inequalities, we have 
\begin{align*}
J_1 &={\rm Re}\, (u_h(t_{n-1}), P_\tau^n u_h(t_{n-1}))\tau 
\le\|u_h(t_{n-1})\|\|P_\tau^n u_h(t_{n-1})\|\tau \\
&\le \frac\tau2\|u_h(t_{n-1})\|^2 + \frac\tau2\|P_\tau^n u_h(t_{n-1})\|^2 .
\end{align*}
Setting $v_h=(t_n-t)P_\tau^n u_h$ in \eqref{uh-test-vh} and taking the 
imaginary part yield  
\begin{align*}
J_2&=
{\rm Re}\int_{I_n} 
(\partial_t u_h(t), (t_n-t)P_\tau^n u_h(t))\d t 
={\rm Im}\int_{I_n} \i\bigl( \partial_t u_h(t), (t_n-t)P_\tau^n u_h(t) \bigr)\d t \\
&={\rm Im}\, \,\tau\sum_{j=1}^k w_j \bigl(r_h(t_{nj})g(u_h(t_{nj})) , |u_h(t_{nj})|^2 \bigr) (t_n-t_{nj}) \\
&\qquad -{\rm Im}\int_{I_n} \|\nabla P_\tau^nu_h\|^2 (t_n-t)  \d t =0 ,
\end{align*}
where we have used \eqref{Ptau-tnj} in deriving the first term on the right-hand side.
Since $(t_n-t)\partial_t P_\tau^n u_h(t)$ is a polynomial of degree $k-1$ in $t$, it follows that 
$$
\int_{I_n} (u_h(t), (t_n-t)\partial_t P_\tau^n u_h(t)) \d t 
=\int_{I_n} (P_\tau^n u_h(t), (t_n-t)\partial_t P_\tau^n u_h(t)) \d t .
$$
Hence, 
\begin{align*}
J_3
&={\rm Re} \int_{I_n} (P_\tau^n u_h(t), (t_n-t)\partial_t P_\tau^n u_h(t)) \d t 
=\int_{I_n}  \frac12 \frac{\d}{\d t} \|P_\tau^n u_h(t)\|^2(t_n-t) \d t \nonumber\\
&=-\frac{\tau}{2} \|P_\tau^n u_h(t_{n-1})\|^2 
+ \int_{I_n}  \frac12 \|P_\tau^n u_h(t)\|^2 \d t . 
\end{align*}
Substituting the estimates of $J_1$, $J_2$ and $J_3$ into \eqref{L2L2Pu-M}, we obtain
\begin{align}\label{L2L2-Ptau-uh}
&\int_{I_n} \|P_\tau^nu_h\|^2 \d t 
\le 
\tau \|u_h(t_{n-1})\|^2 
=\tau \|u_h(0)\|^2 ,
\end{align}
where the last equality follows from mass conservation proved in Theorem \ref{THM:conservation}. 
This proves \eqref{interior_est_1} holds.

Substituting \eqref{L2L2-Ptau-uh} into \eqref{L2-uh-Phuh} and using the mass conservation again, 
we obtain $\int_{I_n} \|u_h\|^2 \d t
\le C\tau \|u_h(0)\|^2$. Then, by using the inverse inequality, we obtain 
\begin{align*}
\max_{t\in I_n} \|u_h(t)\| \le C \|u_h(0)\|,
\end{align*}
which proves \eqref{interior_est_2}. The proof is complete.
\end{proof}

\subsection{Temporal and spatial Ritz projections}

Let $I_\tau^n u$ and $I_\tau^n r$ be the temporal Lagrange interpolation polynomials of $u$ and $r$, respectively, interpolated at the $k+1$ points $t_{n-1}$ and $t_{nj}$, $j=1,\dots,k$. Both $I_\tau^n u$ and $I_\tau^n r$ are temporal polynomials of degree $\le k$. The one-dimensional temporal Lagrange interpolation operator $I_\tau^n$ has the following approximation property (cf. \cite{Canuto2007Spectral}): 
\begin{align}\label{approx-In}
&\max_{t\in I_n}
\bigl(\|v - I_\tau^nv\|_{X} + \tau\|\partial_t (v - I_\tau^nv)\|_{X}\bigr)    
\le C\tau^{m+1} \max_{t\in I_n} \|\partial_t^{m+1}v\|_{X} 
\end{align}  
for all $v\in C^{m+1}([0,T];X),\, 0\le m\le k$, and $X=\R$ or $X=H^s(\Omega)$ for some $s\in\R$. 

To analyze the error of numerical solutions due to temporal discretization, we define a temporal Ritz projection operator $R_\tau^n: W^{1,\infty}(I_n;L^2(\Omega))\rightarrow \mathbb{P}^k\otimes L^2(\Omega)$ by the following two conditions: 
\begin{align}
\int_{I_n}(\partial_t(u-R_\tau^n u),v)\d t &= 0
\quad\forall\, v\in \mathbb{P}^{k-1} \otimes L^2(\Omega) , \label{Ritz-time} \\[3pt]
\label{Ritz-time-u0}
u(t_{n-1})-R_\tau^n u(t_{n-1})&=0 .
\end{align}
Clearly, $\partial_t R_\tau^n u $ is the temporal $L^2$ projection of $\partial_t u$ onto $\mathbb{P}^{k-1}\otimes L^2(\Omega)$. 
By using this property and the shifted Legendre polynomials defined in \eqref{def-Lk-t}, we can express the temporal Ritz projection as
\begin{align}
R_\tau^n u(t) = u(t_{n-1}) + \sum_{j=0}^{k-1} \frac{\int_{I_n} L_j(s) \partial_su(s) \d s}{\int_{I_n} |L_j(s)|^2\d s} \int_{t_{n-1}}^t L_j(s) \d s . 
\end{align}
This expression implies that if $X\subset L^2(\Omega)$ is a Banach space  and $u\in W^{1,\infty}(I_n;X)$, then $R_\tau^n u$ is automatically in $\mathbb{P}^k\otimes X$. 
Meanwhile, this temporal Ritz projection has the following approximation property. 
\begin{lemma}\label{Lemma:time-Ritz} 
Let $X=\R$ or $H^s(\Omega)$ for some $s\ge 0$. 
For $u\in W^{m+1,\infty}(I_n;X)$, with $0 \le m \le k$, the following approximation property holds:
\begin{align*}
\|u-R_\tau^n u\|_{L^{\infty}(I_n;X)}
+\tau \|\partial_t(u-R_\tau^n u)\|_{L^{\infty}(I_n;X)} 
\le
C\tau^{m+1} \|u\|_{W^{m+1,\infty}(I_n;X)} . 
\end{align*}
\end{lemma}
\begin{proof}
We prove the result for the case $X=H^s(\Omega)$ with $s\ge 0$. To this end, we denote by $H^s(\Omega)'$ the dual space of $H^s(\Omega)$. Then, by the Riesz representation theorem, there exists a continuous linear bijection $J:H^s(\Omega)'\rightarrow H^s(\Omega)$ such that 
\begin{align}\label{u-v-J}
(u,Jv)_{H^s} = \langle u, v \rangle   \quad\forall\,u\in H^s(\Omega)\,\,\,\mbox{and}\,\,\, v\in H^s(\Omega)' ,
\end{align}
where $(\cdot,\cdot)_{H^s}$ is the inner product of $H^s(\Omega)$, and $\langle\cdot,\cdot\rangle$ denotes the pairing between $H^s(\Omega)$ and its dual space $H^s(\Omega)' $, satisfying 
$$
\langle w, v \rangle = (w,v) \quad\forall\, w\in H^s(\Omega),\,\,\,v\in L^2(\Omega)\hookrightarrow H^s(\Omega)' . 
$$
Then \eqref{Ritz-time} implies that  
\begin{align}
\int_{I_n}(\partial_t(u-R_\tau^n u), Jv)_{H^s} \d t &= 0
\quad\forall\, v\in \mathbb{P}^{k-1} \otimes L^2(\Omega) . \label{Ritz-time-1}  
\end{align}
Since $L^2(\Omega) $ is dense in $H^s(\Omega)'$, it follows that \eqref{Ritz-time-1} actually holds for all $v\in \mathbb{P}^{k-1} \otimes H^s(\Omega)'$. Since for any $w\in \mathbb{P}^{k-1} \otimes H^s(\Omega)$ there exists $v\in \mathbb{P}^{k-1} \otimes H^s(\Omega)'$ satisfying $Jv=w$, it follows that \eqref{Ritz-time-1} can be equivalently written as 
\begin{align}
\int_{I_n}(\partial_t(u-R_\tau^n u),w)_{H^s} \d t &= 0
\quad\forall\, w \in \mathbb{P}^{k-1} \otimes H^s(\Omega) . \label{Ritz-time-2}  
\end{align} 

By using the inverse inequality in time and the property \eqref{Ritz-time-2}, we have
\begin{align}
\|\partial_t (u - R_\tau^n u)\|^2_{L^{\infty}(I_n; H^s)}
\le& C \tau^{-1}\int_{I_n} \bigl(\partial_t (u - R_\tau^n u),  \partial_t (u - R_\tau^n u)
\bigr)_{H^s}\d t\\
=&  C \tau^{-1}\int_{I_n} \bigl(\partial_t (u - R_\tau^n u),  \partial_t (u - g) \bigr)_{H^s}\d t
\notag\\
\le&C \|\partial_t (u - R_\tau^n u)\|_{L^{\infty}(I_n; H^s)} 
\|\partial_t (u - g)\|_{L^{\infty}(I_n; H^s)}
\notag,
\end{align}
where $g$ can be an arbitrary function in $\mathbb{P}^k\otimes H^s(\Omega)$, which implies $ \partial_t (g - R_\tau^n u) \in \mathbb{P}^{k-1} \otimes H^s(\Omega) $. The inequality above implies 
\begin{align}
\| \partial_t (u - R_\tau^n u)\|_{L^{\infty}(I_n; H^s)}
\le& C \inf_{g \in \mathbb{P}^k\otimes H^s(\Omega)} 
\|\partial_t(u - g)\|_{L^{\infty}(I_n; H^s)}\\
\le& C\|\partial_t(u - I^n_\tau u)\|_{L^{\infty}(I_n; H^s)} . 
\notag
\end{align}
This and the approximation property \eqref{approx-In} together imply the following inequality: 
\begin{align}\label{time-Ritz-H1}
\| \partial_t (u - R_\tau^n u)\|_{L^{\infty}(I_n; H^s)}
\le C\tau^m \|u\|_{W^{m+1,\infty}(I_n;H^s)}.
\end{align}

To estimate $\|u - R_\tau^n u\|_{L^{\infty}(I_n; H^s)}$, we use a duality argument in time.
Let $w \in W^{1, \infty}(I_n; H^s(\Omega))$ be the solution of the following IVP:
\begin{align} 
\label{Ritz-time-dual}
\partial_t w(t) = u(t) - (R_\tau^n u)(t), \quad \mbox{and} \quad  
w(t_n)=0 .
\end{align}
Then, testing \eqref{Ritz-time-dual} by $J^{-1}(u - R_\tau^n u)$, with the operator $J$ defined in \eqref{u-v-J}, we obtain 
\begin{align}
\int_{I_n} \|u - R_\tau^n u\|_{H^s}^2 \d t
=& \int_{I_n} \bigl(\partial_t w,  u - R_\tau^n u
\bigr)_{H^s}\d t
\notag\\
=&- \int_{I_n} \bigl( w, \partial_t(u - R_\tau^n u)
\bigr)_{H^s}\d t 
\notag\\
=&- \int_{I_n} \bigl(w - P^n_\tau w,  \partial_t (u - R_\tau^n u)
\bigr)_{H^s}\d t 
\qquad\mbox{(here \eqref{Ritz-time-2} is used)} 
\notag\\
\le&  C \tau \|w - P^n_\tau w\|_{L^{\infty}(I_n; H^s)}  
\|\partial_t (u - R_\tau^n u)\|_{L^{\infty}(I_n; H^s)}
\notag\\
\le& C\tau^{m+2} \|\partial_t w\|_{L^{\infty}(I_n; H^s)}
\|u\|_{W^{m+1,\infty}(I_n;H^s)} ,
\notag
\end{align}
where we have used \eqref{time-Ritz-H1} in the last inequality. 
By applying the temporal inverse inequality and \eqref{Ritz-time-dual}, we have  
\begin{align}\label{eq3.30}
\|u - R_\tau^n u\|^2_{L^{\infty}(I_n; H^s)}
\le& C \tau^{-1}\|u - R_\tau^n u\|^2_{L^2(I_n; H^s)}\\
\le& C\tau^{m+1} \|\partial_t w\|_{L^{\infty}(I_n; H^s)}
\|u\|_{W^{m+1,\infty}(I_n;H^s)}
\notag\\
\le& C\tau^{m+1} \|u-R_\tau^n u\|_{L^{\infty}(I_n; H^s)}
\|u\|_{W^{m+1,\infty}(I_n;H^s)} , \nonumber
\end{align}
which leads to
\begin{align}\label{time-Ritz-L2}
\|u - R_\tau^n u\|_{L^{\infty}(I_n; H^s)}
\le C\tau^{m+1}\|u\|_{W^{m+1,\infty}(I_n;H^s)}.
\end{align}
This completes the proof of Lemma \ref{Lemma:time-Ritz}. 
\end{proof}

In addition to the above optimal-order approximation result, 
we also have the following super-convergence result. 

\begin{lemma}[A super-approximation property]\label{Lemma:time-Ritz-superappx} 
	Let $X=\R$ or $H^s(\Omega)$ for some $s\ge 0$. 
	If $w\in W^{k,\infty}(I_n;W^{s,\infty}(\Omega))$ and $v\in \mathbb{P}^{k-1}\otimes X$, then 
	\begin{align*}
	\|wv-P_\tau^n(wv)\|_{L^2(I_n;X)}
	\le
	C\tau \|v\|_{L^2(I_n;X)} . 
	\end{align*}
\end{lemma}

\begin{proof}
	We only give a proof for the case $X=H^s(\Omega)$ because the other cases are similar. 
	By applying \eqref{approx-l2-pro} with $m=k$, we have 
	\begin{align*}
	\|wv-P_\tau^n(wv)\|_{L^2(I_n;H^s)}
	&\le
	C\tau^{\frac12} \|wv-P_\tau^n(wv)\|_{L^\infty(I_n;H^s)} \\ 
	&\le
	C\tau^{k+\frac12} \|\partial_t^k(wv)\|_{L^\infty(I_n;H^s)} \\
	&\le
	C\sum_{m=0}^{k-1} \tau^{k+\frac12} \| \partial_t^{k-m} w \partial_t^{m}v\|_{L^\infty(I_n;H^s)}  
	\quad\mbox{(since $\partial_t^kv=0$)}\\
	&\le
	C\sum_{m=0}^{k-1} \tau^{k+\frac12} \| \partial_t^{k-m} w \|_{L^\infty(I_n;W^{s,\infty})}  
	\|\partial_t^{m}v\|_{L^\infty(I_n;H^{s})}  \\
	&\le
	C\sum_{m=0}^{k-1} \tau^{k+\frac12-m} \| v \|_{L^\infty(I_n;H^s)}  \\
	&\le
	C\tau^{\frac32} \| v \|_{L^\infty(I_n;H^s)}  \\
	&\le
	C\tau  \| v \|_{L^2(I_n;H^s)}. 
	\end{align*}
	here we have used the inverse inequality in time twice above. The proof is complete. 
\end{proof}

Finally, we also recall the (spatial) Ritz projection operator $R_h:H^1_0(\Omega)\rightarrow S_h$ defined by 
$$
\bigl(\nabla (w-R_hw),\nabla v_h \bigr)=0 \qquad\forall v_h\in S_h,\,\,\,\forall\, w\in H^1_0(\Omega),
$$ 
and the discrete Laplacian operator $\Delta_h : S_h \to S_h$ defined by
\begin{align} 
(\Delta_h \phi_h, \chi_h) := -(\nabla \phi_h, \nabla \chi_h) \quad \forall\, \phi_h, \chi_h \in S_h .
\end{align}
It is known \cite{Brenner-Scott} that there hold the following identities: 
\begin{subequations}\label{Ph-Delta-Rh}
	\begin{alignat}{2}
	&P_h\Delta v =\Delta_hR_hv  &&\qquad \forall\,v\in H^1_0(\Omega) , \\
	&R_\tau^nR_hv=R_hR_\tau^nv &&\qquad \forall\,v\in W^{1,\infty}(I_n;H^1_0(\Omega)), \\
	&R_\tau^n\Delta_hv_h=\Delta_hR_\tau^nv_h &&\qquad \forall\,v\in W^{1,\infty}(I_n;S_h).
	\end{alignat}
\end{subequations}  
Moreover, there holds the following approximation property (cf. \cite{Brenner-Scott}): 
\begin{align}
&\|v - R_h v\|_{H^1} \le Ch^{p} \| v \|_{H^{p+1} } \quad\forall\, v\in H^1_0(\Omega)\cap H^{p+1}(\Omega) . 
\label{L2-H1-Ritz} 
\end{align}

\subsection{Consistency of scheme \eqref{GL-uh}--\eqref{GL-rh}}

We define a pair of intermediate solutions (for comparison with the numerical solutions)
$$
u_h^*=R_\tau^n R_h u
\quad\mbox{and}\quad
r_h^*=R_\tau^n r .
$$ 
Then, testing \eqref{IEQ-u} and \eqref{IEQ-r} by $P_\tau^nv_h$ and $P_\tau^nq_h$, respectively, we obtain the following equations for $u_h^*$ and $r_h^*$:
\begin{align}\label{GL-u-rf1}
&\int_{I_n}
\i\bigl(\partial_t u_h^*, P_\tau^nv_h \bigr) \d t
+ 
\int_{I_n}
\bigl(\nabla u_h^* , \nabla P_\tau^nv_h \bigr) \d t
 \\
&\qquad - 
\frac{\tau}{2}\sum_{j=1}^k w_j \bigl(r_h^*(t_{nj})g(u_h^*(t_{nj})) u_h^*(t_{nj}),P_\tau^n v_h(t_{nj}) \bigr) 
= \int_{I_n} (d_u^n,P_\tau^nv_h) \d t , \nonumber \\ 
& \int_{I_n} \partial_tr_h^* P_\tau^nq_h \d t
=\frac{\tau}{4}\sum_{j=1}^k w_j{\rm Re} \bigl(P_\tau^nq_h(t_{nj}) g(u_h^*(t_{nj})) u_h^*(t_{nj}) , \partial_t u_h^*(t_{nj}) \bigr)  \label{GL-r-rf1}  \\
&\hspace{79pt}
+  \int_{I_n} d_r^n P_\tau^nq_h \d t , 
 \nonumber 
\end{align}
where $d_u^n$ and $d_r^n$ are consistency errors of the numerical method. By using the identities \eqref{Gauss-integral} and \eqref{Ritz-time}, we can express these consistency errors by 
\begin{align}\label{def:dnj-u} 
d_u^n =& \i \partial_tR_\tau^n(R_hu-u) + \Delta_h R_h ( u - R_\tau^n u) 
+ r g(u)u-I_\tau^n[r_h^*g(u_h^*)u_h^*]  ,  \\[3pt]
\label{def:dnj-r} 
d_r^n =
& \frac{1}{2} {\rm Re} \big[
\big(g(u) u, \partial_t  u \big) 
- I_\tau^n \big(g(u_h^*) u_h^*, \partial_t u_h^*\big) \big] . 
\end{align} 
After using \eqref{l2-pro} and \eqref{Ptau-tnj}, equations \eqref{GL-u-rf1}--\eqref{GL-r-rf1} can be rewritten as
\begin{align}\label{GL-u-rf}
&\int_{I_n}
\i\bigl(\partial_t u_h^*, v_h \bigr) \d t
+ 
\int_{I_n}
\bigl(\nabla P_\tau^n u_h^* , \nabla v_h \bigr) \d t
 \\
&\qquad - 
\frac{\tau}{2}\sum_{j=1}^k w_j \bigl(r_h^*(t_{nj})g(u_h^*(t_{nj})) u_h^*(t_{nj}), v_h(t_{nj}) \bigr) 
= \int_{I_n} (P_\tau^n d_u^n,v_h) \d t , \nonumber \\ 
& \label{GL-r-rf} 
\int_{I_n} \partial_tr_h^* q_h \d t
=\frac{\tau}{4}\sum_{j=1}^k w_j{\rm Re} \bigl(q_h(t_{nj}) g(u_h^*(t_{nj})) u_h^*(t_{nj}) , \partial_t u_h^*(t_{nj}) \bigr)  \\
&\hskip 1in +  \int_{I_n} P_\tau^n d_r^n q_h \d t.
 \nonumber
\end{align}
 
\begin{theorem}\label{THM:consistency}
Suppose that the solution of \eqref{pde} is sufficiently smooth, then $d_u^n \in C(I_n;H^1_0(\Omega))$ and there hold 
\begin{align}\label{consistency-dnj} 
&\sup_{t\in I_n} \|d_u^n\|_{H^1} \le C(h^{p}+\tau^{k+1}) 
\quad\mbox{and}\quad 
\sup_{t\in I_n}|P_\tau^n d_r^n|
\le
C(h^{p} +\tau^{k+1})  .
\end{align}
\end{theorem}

\begin{remark}\label{Remark:consistency}
{\upshape 
The key observation for the consistency errors is that, although \eqref{def:dnj-r} contains an $O(h^p+\tau^k)$ error from $\partial_tu-\partial_tu_h^*$, the temporal $L^2$ projection operator $P_\tau^n$ acting on $d_r^n$ furthermore reduces this error to $O(h^p+\tau^{k+1})$. This is proved by using the super-approximation result in Lemma \ref{Lemma:time-Ritz-superappx}.  
}
\end{remark}

\begin{proof}
Since the spatial Ritz projection $R_h$ maps $H^1_0(\Omega)$ into $S_h\subset H^1_0(\Omega)$, and the temporal Ritz projection $R_\tau^n$ maps $W^{1,\infty}(I_n;H^1_0(\Omega))$ into $\mathbb{P}^k\otimes H^1_0(\Omega)$, it follows that every term in \eqref{def:dnj-u} is in $C(I_n;H^1_0(\Omega))$. This implies $d_u^n \in C(I_n;H^1_0(\Omega))$. 

By using the triangle inequality, from \eqref{def:dnj-u} we get   
\begin{align} \label{dnj^u-est}
\max_{t\in I_n} \|d_u^n\|_{H^1} 
\le
&\max_{t\in I_n} 
\big(
\|\partial_tR_\tau^n(R_hu-u)\|_{H^1}  + \| \Delta_h R_h (u - R_\tau^n u) \|_{H^1} \big)  \\
&+ \max_{t\in I_n} 
\big( \|r g(u)u-I_\tau^n[rg(u)u]\|_{H^1}
+ \|rg(u)u-r_h^*g(u_h^*)u_h^*\|_{H^1} \big) 
\notag\\
=&:D^u_1 + D^u_2 + D^u_3 + D^u_4. \notag
\end{align}
Choosing $m=0$ in Lemma \ref{Lemma:time-Ritz}, we obtain the following stability result: 
\begin{align}\label{time-stability-W1infty}
\|R_\tau^n u\|_{W^{1,\infty}(I_n;H^s)} 
\le
C \|u\|_{W^{1,\infty}(I_n;H^s)}  . 
\end{align}
Using \eqref{time-stability-W1infty} and \eqref{L2-H1-Ritz}, we can estimate $D^u_1$ as follows: 
\begin{align*} 
D^u_1
= \max_{t\in I_n}  \|\partial_tR_\tau^n(R_hu-u)\|_{H^1}
&\le \|R_hu-u\|_{W^{1,\infty}(I_n;H^1)} \\
&\le Ch^p \|R_hu-u\|_{W^{1,\infty}(I_n;H^{p+1})} . \notag
\end{align*}
Similarly, using identity \eqref{Ph-Delta-Rh} and Lemma \ref{Lemma:time-Ritz}, we have
\begin{align*} 
D^u_2 =
\max_{t \in I_n} \| \Delta_h R_h (u - R_\tau^n u)  \|_{H^1}
&=
\max_{t \in I_n} \| P_h \Delta (u - R_\tau^n u)  \|_{H^1}  \\
&\le
\max_{t \in I_n} \| u - R_\tau^n u  \|_{H^3} \notag \\
&\le C\tau^{k+1} \| u \|_{W^{k+1,\infty}(I_n;H^3)}  , \notag
\quad\,\,\,
\end{align*}
and
\begin{align*} 
D^u_3 =
\max_{t \in I_n} \|r g(u)u-I_\tau^n[rg(u)u]\|_{H^1}
&\le 
C\tau^{k+1} . 
\end{align*}
By using the triangle inequality, we decompose $D^u_4$ into two parts, 
\begin{align*} 
D^u_4 
\le
&\max_{t \in I_n} 
\big( \|r g(u)u- r g(R_h u)R_hu\|_{H^1} 
 + \|r g(R_h u)R_hu-R_\tau^n r g(R_\tau^n R_hu )R_\tau^n R_hu\|_{H^1} \big) \notag \\
\le &
Ch^{p} + C\tau^{k+1} . \notag  
\end{align*}
Then, substituting the estimates of $D^u_j$, $j=1,2,3,4$, into \eqref{dnj^u-est}, we obtain the desired estimate for $\|d_u^n\|_{H^1}$. 

To estimate $|P_\tau^n d_r^n|$, we rewrite \eqref{def:dnj-r} as
\begin{align*} 
d_r^n
=&
\frac{1}{2} {\rm Re} \Big[ 
\bigl(g(u) u, \partial_t (u -  u_h^*)\bigr)  
+  \bigl(g(u) u-g(u_h^*) u_h^*, \partial_t u_h^*\bigr) \Big] \\
&
+ \frac{1}{2} {\rm Re} \Big[\big(g(u_h^*) u_h^* ,\partial_t u_h^*\big) 
- I_\tau^n \big( g(u_h^*) u_h^*, \partial_t u_h^*\big)  \Big]\notag
\end{align*}
and test this expression by $P_\tau^n v$ in the time interval $I_n$, with $v\in\mathbb{P}^{k}$. This yields 
\begin{align}\label{dnj^r-est1} \
&\int_{I_n} P_\tau^n d_r^n v \,\d t
=\int_{I_n}  d_r^n P_\tau^n  v \,\d t \\
&\quad \le \frac12{\rm Re} \int_{I_n} \bigl(g(u) u ,\partial_t (u-u_h^*)\bigr) P_\tau^n v \,\d t \notag\\
&\qquad+C \tau^{\frac12} \| (g(u) u - g(u_h^*) u_h^*,\partial_t u_h^*) \|_{L^\infty(I_n)} \|v\|_{L^2(I_n)} \notag \\
&\qquad +C\tau^{k+\frac32}
\| \partial_t^{k+1}\big(g(u_h^*) u_h^* ,\partial_t u_h^*\big)  \|_{L^\infty(I_n)}
\|v\|_{L^2(I_n)}\notag  \\
&\quad \le \frac12{\rm Re} \int_{I_n} \bigl(g(u) u ,\partial_t (u-u_h^*)\bigr) P_\tau^n v \,\d t 
+ C\tau^{\frac12} (h^{p} +\tau^{k+1}) \|v\|_{L^2(I_n)}. \notag 
\end{align} 
The first term on the right-hand side of \eqref{dnj^r-est1}  can be estimated as follows. 
\begin{align}\label{dnj^r-est2}
\frac12{\rm Re} \int_{I_n} \bigl(g(u) u ,\partial_t (u-u_h^*)\bigr) P_\tau^n v \,\d t 
=& \int_{I_n} (g(u) u ,\partial_t (u-R_\tau^nu))  P_\tau^n v \,\d t  \\
& +\int_{I_n} (g(u) u ,\partial_t R_\tau^n (u-R_hu))  P_\tau^n v \,\d t \notag \\
=&:D^r_1+D^r_2, \nonumber
\end{align}
where
\begin{align*}
D^r_1=
&
\int_{I_n} \bigl( g(u) u P_\tau^n v  ,\partial_t (u-R_\tau^nu)\bigr)  \,\d t \\
=
&
\int_{I_n} \bigl(g(u) uP_\tau^n v- P_\tau^n( g(u) uP_\tau^n v ) ,\partial_t (u-R_\tau^nu)\bigr)   \,\d t \\
\le 
&C\tau^{\frac12} \|g(u) uP_\tau^n v- P_\tau^n( g(u) uP_\tau^n v) \|_{L^2(I_n;L^2)} 
\|\partial_t (u-R_\tau^nu)\|_{L^\infty(I_n;L^2)} \\
\le
& C\tau^{\frac32} \|P_\tau^n v\|_{L^2(I_n;L^2)} 
\|\partial_t (u-R_\tau^nu)\|_{L^\infty(I_n;L^2)} \quad\mbox{(we have used Lemma \ref{Lemma:time-Ritz-superappx})}\\
\le
& C\tau^{k+\frac32} \| v\|_{L^2(I_n)} \|\partial_t^{k+1}u\|_{L^\infty(I_n;L^2)} 
\hspace{51.5pt}\mbox{(we have used Lemma \ref{Lemma:time-Ritz})}  ,\\[5pt] 
D^r_2\le
&
C\tau^{\frac12}
\| g(u) u \|_{L^\infty(I_n;L^2)}
\| u-R_hu \|_{W^{1,\infty}(I_n;L^2)} 
\|v\|_{L^2(I_n)} \\
\le 
&
C\tau^{\frac12} h^{p} \|v\|_{L^2(I_n)} \|u\|_{W^{1,\infty}(I_n;H^{p+1})} . 
\end{align*}
Substituting these estimates into \eqref{dnj^r-est1}, we obtain
\begin{align*}
&\bigg| \int_{I_n} P_\tau^n d_r^n  v \,\d t \bigg|
\le
C\tau^{\frac12} (h^{p} +\tau^{k+1}) \|v\|_{L^2(I_n)} .
\end{align*}
Since this inequality holds for arbitrary $v\in L^2(I_n)$, it follows that 
\begin{align*}
\|P_\tau^n d_r^n\|_{L^2(I_n)} 
\le C\tau^{\frac12} (h^{p} +\tau^{k+1})  . 
\end{align*}
Then, using the inverse inequality in time, we obtain  the desired estimate for $|P_\tau^n d_r^n|$. 
\end{proof}

\section{Well-posedness and convergence analysis}\label{sec-4} 

We define the error functions $e^u_h=u_h-u_h^*$ and $e^r_h=r_h-r_h^*$, with the following abbreviations: 
\begin{align*}
\begin{aligned}
&e^u_{nj}=e^u_h(t_{nj})
&&\mbox{and}&&
e^r_{nj}=e^r_h(t_{nj}) ,\\
&u_{nj}=u_h(t_{nj})
&&\mbox{and}&&
r_{nj}=r_h(t_{nj}) ,\\
&u_{nj}^*=u_h^*(t_{nj})
&&\mbox{and}&&
r_{nj}^*=r_h^*(t_{nj}) 
,\\
&v_{nj}=v_h(t_{nj}) 
&&\mbox{and}&&q_{nj}=q_h(t_{nj}).
\end{aligned}
\end{align*}
Subtracting \eqref{GL-u-rf}--\eqref{GL-r-rf} from \eqref{uh-test-vh}--\eqref{rh-test-qh}, we obtain the following error equations: 
\begin{subequations}\label{GL-uh-rh-err-0}
\begin{align}
\label{GL-uh-err-0}
\i\int_{I_n}
\big(\partial_t e^u_h, v_h\big)\, \d t 
=& -\int_{I_n} \big(\nabla P^n_\tau e^u_h, \nabla v_h\big)\, \d t 
+\frac{\tau}{2}\sum_{j=1}^k w_j 
\Big(e^r_{nj} g(u_{nj})u_{nj}, v_{nj}\Big)
\notag\\
&+\frac{\tau}{2}\sum_{j=1}^k w_j 
\Big(r_{nj}^* \big[g(u_{nj})u_{nj} - g(u_{nj}^*) u_{nj}^*\big], v_{nj}\Big) 
- \int_{I_n} (P_\tau^n d_u^n,v_h) \d t,\\
\label{GL-rh-err-0}
\int_{I_n} \partial_t e^r_h q_h  \d t 
=&\frac{\tau}{4}\sum_{j=1}^k w_j 
{\rm Re}\big( q_{nj} \big( g(u_{nj})u_{nj} -g(u_{nj}^*) u_{nj}^* \big), \partial_t  u_h^*(t_{nj})\big)
\notag\\
&+ \frac{\tau}{4}\sum_{j=1}^k w_j 
{\rm Re}\big( q_{nj} g(u_{nj})u_{nj} , \partial_t e^u_h(t_{nj})\big)
- \int_{I_n}  P_\tau^n d_r^n q_h  \d t ,   
\end{align}
\end{subequations}
which hold for all test functions $v_h\in \mathbb{P}^k\otimes S_h$ and $q_h\in \mathbb{P}^k$.

\begin{remark}\label{Remark:existence}
{\upshape 
If \eqref{GL-uh-rh-err-0} has a solution $(e^u_h, e_h^r)\in X_{\tau,h}\times Y_{\tau,h}$ with $u_h=u_h^*+e^u_h$ and $r_h=r_h^*+e^r_h$, then $(u_h,r_h)$ is a solution of the numerical scheme \eqref{GL}. In the following, we prove existence of a solution $(e^u_h,e_h^r)$ to \eqref{GL-uh-rh-err-0} with $u_h=u_h^*+e^u_h$ and $r_h=r_h^*+e^r_h$. 
}
\end{remark}

In this section, we prove existence and uniqueness of solutions to \eqref{GL-uh-err-0}--\eqref{GL-rh-err-0} 
by using Schaefer's  Fixed Point Theorem, which is quoted below. 

\begin{theorem}[Schaefer's  Fixed Point Theorem {\cite[Chapter 9.2, Theorem 4]{Evans-PDE}}] \label{THMSchaefer}
Let $B$ be a Banach space and let $M:B\rightarrow B$ be a continuous and compact mapping $($possibly nonlinear$)$. If the set 
	\begin{equation}
	\bigl\{\phi\in B:\; \exists\, \theta\in [0,1] \,\,\,\mbox{such that}\,\,\, \phi=\theta M(\phi) \bigr\}
	\end{equation}
	is bounded in $B$, then the mapping $M$ has at least one fixed point.
\end{theorem}

We define 
\begin{align}\label{def-X-star}
X_{\tau,h}^*&=\Bigl\{v_h\in X_{\tau,h}:  
\max_{1\le n\le N}\max_{1\le j\le k} \|v_h(t_{nj})-u_h^*(t_{nj})\|_{L^\infty\cap H^1}\le \frac12 \Bigr\} ,\\
Y_{\tau,h}^*&=\Bigl\{q_h\in Y_{\tau,h}:  
\max_{1\le n\le N}\max_{1\le j\le k} |q_h(t_{nj})-r_h^*(t_{nj})|\le \frac12 \Bigr\} ,
\label{def-Y-star}
\end{align}
where the norm $\|\cdot\|_{L^{\infty}\cap H^1}$ is defined as 
$$
\|\phi_h\|_{L^{\infty}\cap H^1}
:=\max \big( \|\phi_h\|_{L^{\infty}},\|\phi_h\|_{H^1} \big) . 
$$
For any element $(\phi_h,\varphi_h)\in X_{\tau,h}\times Y_{\tau,h}$, 
we define two associated numbers 
\begin{subequations}\label{def-rho}
\begin{align}\label{def-rho_a}
&\rho[\phi_h]:= \min\bigg(\frac{1}{\displaystyle
\max_{1\le n\le N}\max_{1\le j\le k}\|\phi_h(t_{nj})\|_{L^{\infty}\cap H^1}},1\bigg) ,\\
&\rho[\varphi_h]:= \min\bigg(\frac{1}{\displaystyle
\max_{1\le n\le N}\max_{1\le j\le k}
|\varphi_h(t_{nj})|},1\bigg) ,  \label{def-rho_b}
\end{align}
\end{subequations}
which are continuous with respect to $(\phi_h,\varphi_h)$ (because all norms are equivalent in the finite-dimensional space $X_{\tau,h}\times Y_{\tau,h}$). Furthermore, the two numbers defined above satisfy the following estimates:
\begin{align}
\max_{1\le n\le N} \max_{1\le j\le k}
\|\rho[\phi_h]\phi_h(t_{nj})\|_{L^{\infty}\cap H^1} &\le 1 , \label{phi-varphi-Linfty-1}\\
\max_{1\le n\le N} \max_{1\le j\le k} |\rho[\varphi_h]\varphi_h(t_{nj})| &\le 1 .
\label{phi-varphi-Linfty-2}
\end{align}
Then we define 
\begin{align}\label{def-e-phi}
&u^\phi:= 
u_h^*+\rho[\phi_h]\phi_h
\quad\mbox{and}\quad
r^\varphi:= 
r_h^*+\rho[\varphi_h]\varphi_h,
\end{align}
with the following abbreviations: 
\begin{align*}
\begin{aligned}
&u_{nj}^{\phi}=u_h^{\phi}(t_{nj})
&&\mbox{and}&&
\varphi_{nj}=\varphi_h(t_{nj}),
\end{aligned}
\end{align*}
and define $(e^u_h,e^r_h)\in X_{\tau,h}\times Y_{\tau,h}$ to be the solution of the following linear equations: 
\begin{align} \label{GL-uh-map}
&\i\int_{I_n}\big(\partial_t e^u_h, v_h\big)\, \d t 
+\int_{I_n} \big(\nabla P^n_\tau e^u_h, \nabla v_h\big)\, \d t 
=\frac{\tau}{2}\sum_{j=1}^k w_j 
\Big(\varphi_{nj} g(u_{nj}^\phi)u_{nj}^\phi, v_{nj}\Big) \\
&\qquad  +\frac{\tau}{2}\sum_{j=1}^k w_j 
\Big(r_{nj}^* \big[g(u_{nj}^\phi)u_{nj}^\phi - g(u_{nj}^*) u_{nj}^*\big], v_{nj}\Big)   
- \int_{I_n} (P_\tau^n d^u,v_h) \d t  \nonumber
\end{align}
and
\begin{align} \label{GL-rh-map}
\int_{I_n} \partial_t e^r_h q_h \, \d t 
&=\frac{\tau}{4}\sum_{j=1}^k w_j 
{\rm Re}\big( q_{nj} \big( g(u_{nj}^\phi) u_{nj}^\phi -g(u_{nj}^*) u_{nj}^* \big) , \partial_t u_h^*(t_{nj}) \big)\\
&\qquad + \frac{\tau}{4}\sum_{j=1}^k w_j 
{\rm Re}\big( q_{nj}g(u_{nj}^\phi) u_{nj}^\phi , \partial_t\phi_h(t_{nj}) \big)
-\int_{I_n} P_\tau^n d^r q_h \d t \nonumber
\end{align}
for all $v_h\in \mathbb{P}^k\otimes S_h$ and $q_h\in \mathbb{P}^k$, $n=1,\dots,N$. We denote by $M: X_{\tau,h}\times Y_{\tau,h}\rightarrow X_{\tau,h}\times Y_{\tau,h}$ the mapping from $(\phi_h,\varphi_h)$ to $(e^u_h,e^r_h)$, and define the set 
\begin{equation}
\mathfrak{B}=\bigl\{(\phi_h,\varphi_h)\in X_{\tau,h}\times Y_{\tau,h}:\; \exists\, \theta\in [0,1] \,\,\,\mbox{such that}\,\,\, (\phi_h,\varphi_h)=\theta M(\phi_h,\varphi_h) \bigr\},
\end{equation}
and the following norm on $X_{\tau,h}\times Y_{\tau,h}$: for any $ (\phi_h,\varphi_h) \in {X_{\tau,h}\times Y_{\tau,h}}$
\begin{equation}
\|(\phi_h,\varphi_h)\|_{X_{\tau,h}\times Y_{\tau,h}} :=\|\phi_h\|_{L^\infty(0,T;H^1)}
+\|\varphi_h\|_{L^\infty(0,T)}.
\end{equation}

\begin{lemma}
The mapping $M: X_{\tau,h}\times Y_{\tau,h}\rightarrow X_{\tau,h}\times Y_{\tau,h}$ is well defined, continuous and compact. 
\end{lemma}
\begin{proof}
Since the right-hand sides of \eqref{GL-uh-map}--\eqref{GL-rh-map} are given, the linear equations \eqref{GL-uh-map}--\eqref{GL-rh-map} have a unique solution $(e^u_h,e^r_h)\in X_{\tau,h}\times Y_{\tau,h}$ for any given $(\phi_h,\varphi_h)\in X_{\tau,h}\times Y_{\tau,h}$. Thus the mapping is well defined. 
Let $\ell^u(\phi_h,\varphi_h;v_h)$ and $\ell^r(\phi_h,\varphi_h;q_h)$ denote the right-hand sides of \eqref{GL-uh-map} and \eqref{GL-rh-map}, respectively. Since all norms are equivalent in the finite-dimensional space $X_{\tau,h}\times Y_{\tau,h}$, it follows that 
\begin{align*}
|\ell^u(\phi_h,\varphi_h;v_h)-\ell^u(\hat\phi_h,\hat\varphi_h;v_h)|
&\le o(1) \|v_h\|_{L^2(0,T;L^2)}, \qquad \forall v_h\in L^2(0,T; L^2), \\
|\ell^r(\phi_h,\varphi_h;q_h)-\ell^r(\hat\phi_h,\hat\varphi_h;q_h)|
&\le o(1) \|q_h\|_{L^2(0,T)} \qquad \forall q_h\in L^2(0,T),
\end{align*}  
as $(\phi_h,\varphi_h)\rightarrow (\hat\phi_h,\hat\varphi_h)
\,\,\,\mbox{in}\,\,\, X_{\tau,h}\times Y_{\tau,h}$. Where $o(1)$ represents some quantity tending to zero. Using this property, it is easy to verify that $(e^u_h,e^r_h)$ is continuous with respect to $(\phi_h,\varphi_h)$. 

Since $X_{\tau,h}\times Y_{\tau,h}$ is a finite-dimensional space, a continuous mapping is automatically compact. The proof is complete.
\end{proof}

We are now ready to state and prove the following key lemma.

\begin{lemma}\label{Lemma:M-theta}
Let $1\leq d\leq 3$ and assume that the solution of the NLS equation \eqref{pde} is sufficiently smooth. Then there exist positive constants $\tau_0$ and $h_0$ such that when $\tau\le \tau_0$ and $h\le h_0$, 
the following statement holds: 
If $(\phi_h,\varphi_h)\in \mathfrak{B}$ and $(e^u_h,e^r_h)=M(\phi_h,\varphi_h)$, then 
\begin{align}\label{H1-stability}
&\|e^u_h\|_{L^\infty(0,T;H^1)} + \|e^r_h\|_{L^\infty(0,T)} \nonumber \\
&\le \Big[ \|e^u_h(0)\|_{H^1} +|e^r_h(0)|
+\max_{1\le n\le N}\max_{t \in I_n}\bigl( \|d_u^n\|_{H^1} + |P_\tau^n  d_r^n| \bigr) 
\Big] ,\\ 
&\max_{1\le n\le N}\max_{1\le j\le k} 
\|e^u_h(t_{nj})\|_{L^\infty\cap H^1} \le \frac12
\quad\mbox{and}\quad
\max_{1\le n\le N}\max_{1\le j\le k} |e^r_h(t_{nj})| \le \frac12 , \label{Linfty-e} \\[5pt]
&\rho[\phi_h]= 1, \quad\, \rho[\varphi_h]=1 .  \label{Linfty-phi}
\end{align}  
\end{lemma}

\begin{proof} 
If $(\phi_h,\varphi_h)\in \mathfrak{B}$ and $(e^u_h,e^r_h)=M(\phi_h,\varphi_h)$, then 
$$
(\phi_h,\varphi_h)=\theta M(\phi_h,\varphi_h)=(\theta e^u_h,\theta e^r_h) ,
$$
which implies $\phi_h=\theta e^u_h$ and $\varphi_h=\theta e^r_h$. In this case, \eqref{GL-uh-map}--\eqref{GL-rh-map} can be rewritten as  
\begin{align}\label{GL-uh-err}
\i\int_{I_n}
\big(\partial_t e^u_h, v_h\big)\, \d t 
&= 
-\int_{I_n} \big(\nabla P^n_\tau e^u_h, \nabla v_h\big)\, \d t 
+\frac{\theta \tau}{2}\sum_{j=1}^k w_j 
\Big(e^r_{nj} g(u_{nj}^\phi) u_{nj}^\phi, v_{nj}\Big) \\
&\qquad\, +\frac{\tau}{2}\sum_{j=1}^k w_j 
\Big(r_{nj}^* \big[g(u_{nj}^\phi)u_{nj}^\phi - g(u_{nj}^*) u_{nj}^*\big], v_{nj}\Big) \nonumber \\
&\qquad - \int_{I_n} (P_\tau^n d_u^n,v_h) \d t, \nonumber \\
\label{GL-rh-err}
\int_{I_n} \partial_t e^r_h q_h \d t 
&=\frac{\tau}{4}\sum_{j=1}^k w_j 
{\rm Re}\big( q_{nj} \big( g(u_{nj}^\phi) u_{nj}^\phi -g(u_{nj}^*) u_{nj}^* \big) , \partial_t u_h^*(t_{nj}) \big)
\\
&\qquad\, + \frac{\theta\tau}{4}\sum_{j=1}^k w_j 
{\rm Re}\big( q_{nj} g(u_{nj}^\phi)u_{nj}^\phi , \partial_t e^u_h(t_{nj}) \big)  \nonumber \\
&\qquad - \int_{I_n} P_\tau^n d_r^n \, q_h \d t  , \nonumber
\end{align}
which hold for all $v_h\in \mathbb{P}^k\otimes S_h$ and $q_h\in \mathbb{P}^k$, $n=1,\dots,N$. In the following, 
we derive estimates for $e^u_h$ and $e^r_h$ based on equations \eqref{GL-uh-err}--\eqref{GL-rh-err}. 

From \eqref{phi-varphi-Linfty-1}--\eqref{phi-varphi-Linfty-2} and definition \eqref{def-e-phi} we get
\begin{align}\label{Linfty-uh-rh}
& \max_{1\le n\le N} \max_{1\le j\le k} \|u^\phi(t_{nj})\|_{L^\infty\cap H^1}
+ \max_{1\le n\le N} \max_{1\le j\le k}|r^\varphi(t_{nj})| \\
&\qquad \le \max_{1\le n\le N} \max_{1\le j\le k}
\|u_h^*(t_{nj})\|_{L^\infty\cap H^1} 
+ \max_{1\le n\le N} \max_{1\le j\le k} |r_h^*(t_{nj})| \nonumber \\
&\qquad \qquad + 
\max_{1\le n\le N} \max_{1\le j\le k} \|\rho[\phi_h]\phi_h(t_{nj})\|_{L^\infty\cap H^1}
+ \max_{1\le n\le N} \max_{1\le j\le k}
|\rho[\varphi_h]\varphi_h(t_{nj})| \nonumber \\
&\qquad \le \|u_h^*\|_{L^\infty(0,T;L^\infty\cap H^1)} 
+\|r_h^*\|_{L^\infty(0,T)} + 2 . \nonumber
\end{align}
Thus $\|u^\phi(t_{nj})\|_{L^\infty\cap H^1}$ and $|r^\varphi(t_{nj})|$ are bounded uniformly with respect to $\tau$ and $h$. 

Since the remainder of the proof is long, we divide it into four steps. 

\medskip
{\em Step 1: Estimation of $\|e^u_h\|_{L^2(I_n;H^1)}$.}
Note that 
\begin{align}\label{err-Tu2-pr}
&\int_{I_n} \|\nabla P_\tau^n e^u_h(t)\|^2 \d t 
=\int_{I_n}\big(\nabla e^u_h(t), \nabla P_\tau^n  e^u_h(t)\big) \d t \\
&\quad =\int_{I_n}
\bigg[ \big(\nabla e^u_h(t_{n-1}), \nabla P_\tau^n  e^u_h(t_{n-1}) \big) 
+ \int_{t_{n-1}}^t\partial_s\big(\nabla e^u_h(s), \nabla P_\tau^n  e^u_h(s)\big) \d s \bigg]\d t \nonumber\\ 
&\quad =\tau \big( \nabla e^u_h(t_{n-1}), \nabla P_\tau^n  e^u_h(t_{n-1})\big) +\int_{I_n} 
\partial_s\big( \nabla e^u_h(s), \nabla P_\tau^n e^u_h(s)\big) (t_n-s) \d s \nonumber\\
&\quad =\tau \big(\nabla e^u_h(t_{n-1}), \nabla P_\tau^n  e^u_h(t_{n-1})\big) \nonumber \\
&\qquad\, 
+{\rm Re}\int_{I_n}  \Big(\partial_t \nabla e^u_h(t), \nabla P_\tau^n  e^u_h(t)(t_n-t)\Big)\d t \nonumber\\
&\qquad\,
+{\rm Re}\int_{I_n} \Big(\nabla e^u_h(t), (t_n-t)\partial_t \nabla P_\tau^n  e^u_h(t)\Big) \d t \nonumber \\ 
&\quad =\tau \big(\nabla e^u_h(t_{n-1}), \nabla P_\tau^n e^u_h(t_{n-1})\big) \nonumber\\
&\qquad\, 
+{\rm Re}\int_{I_n} 
\Big(\partial_t \nabla e^u_h(t), P_\tau^n \big[\nabla P_\tau^n e^u_h(t)(t_n-t)\big]\Big)\d t \nonumber \\
&\qquad\,
+{\rm Re}\int_{I_n} \Big(\nabla P_\tau^n  e^u_h(t), (t_n-t)\partial_t \nabla P_\tau^n  e^u_h(t)\Big) \d t \nonumber \\
&\quad =: \tau \big(\nabla e^u_h(t_{n-1}), \nabla P_\tau^n  e^u_h(t_{n-1})\big) + J_{u1} + J_{u2} . \nonumber
\end{align}
In the second to last equality, we have used the identity
$$
\int_{I_n} \Big( \nabla e^u_h(t), (t_n-t)\partial_t \nabla P_\tau^n  e^u_h(t)\Big) \d t
=
\int_{I_n} \Big(\nabla P_\tau^n e^u_h(t), (t_n-t)\partial_t \nabla P_\tau^n  e^u_h(t)\Big) \d t ,
$$
which holds because $(t_n-t)\partial_t \nabla P_\tau^n e^u_h(t)$ is a polynomial of degree $k-1$ 
in time. 

Using integration by parts, we have 
\begin{align}\label{err-Tu2}
J_{u2} &=\int_{I_n}  \frac12 \frac{\d}{\d t} \|\nabla P_\tau^n e^u_h(t)\|^2(t_n-t) \d t \\
&=-\frac12 \|\nabla P_\tau^n e^u_h(t_{n-1})\|^2\tau
+ \int_{I_n} \frac12 \|\nabla P_\tau^n e^u_h(t)\|^2 \d t . \nonumber
\end{align}
Then substituting this into \eqref{err-Tu2-pr} yields 
\begin{align}\label{pn-err-u}
&\int_{I_n} \|\nabla P_\tau^n e^u_h(t)\|^2 \d t 
\le 
2\tau \big(\nabla e^u_h(t_{n-1}), \nabla P_\tau^n e^u_h(t_{n-1})\big)
+ 2 J_{u1} . 
\end{align} 
Setting $v_h = (-\Delta_h) P_\tau^n \big[P_\tau^n e^u_h(t)(t_n-t)\big]$ in \eqref{GL-uh-err} and taking the imaginary part yield
\begin{align} \label{err-Ju1-1}
&J_{u1}= {\rm Im}\int_{I_n} 
\Big({\rm i} \partial_t e^u_h(t), (-\Delta_h)  P_\tau^n \big[P_\tau^n  e^u_h(t)(t_n-t)\big]\Big)\d t\\
&\, =-{\rm Im} \int_{I_n} \big(\nabla P^n_\tau e^u_h(t), 
\nabla (-\Delta_h) P_\tau^n \big[P_\tau^n  e^u_h(t)(t_n-t)\big] \big)\, \d t 
\nonumber\\ 
&\quad +\frac{\theta\tau}{2}\sum_{j=1}^k w_j 
{\rm Im}\Big(e^r_{nj} g(u_{nj}^\phi)u_{nj}^\phi, 
(-\Delta_h)P_\tau^n \big[P_\tau^n e^u_{nj}(t_n-t_{nj}) \big]\Big)
\nonumber\\
&\quad +\frac{\tau}{2}\sum_{j=1}^k w_j {\rm Im}
\Big(r_{nj}^* \big(g(u_{nj}^\phi)u_{nj}^\phi - g(u_{nj}^*)u_{nj}^*\big), 
(-\Delta_h)P_\tau^n \big[P_\tau^n e^u_{nj}(t_n-t_{nj}) \big]\Big)
\nonumber\\
&\quad 
-\int_{I_n}{\rm Im} \big(P_\tau^n d_u^n, (-\Delta_h)P_\tau^n [P_\tau^n  e^u(t)(t_n-t)] \big) \d t
 =: \sum_{m = 1}^4 J_{u1}^{m} , \nonumber
\end{align}
where 
\begin{align}\label{err-Tu1-2}
J_{u1}^{1} &=-{\rm Im} \int_{I_n} \big(\Delta_hP^n_\tau e^u_h(t), 
\Delta_hP_\tau^n \big[P_\tau^n  e^u_h(t)(t_n-t)\big] \big)\, \d t  \\ 
&= -{\rm Im} \int_{I_n} 
\|\Delta_hP_\tau^n e^u_h(t)\|^2(t_n-t) \, \d t  = 0 . \nonumber 
\end{align}
From identity \eqref{Gauss-integral} and inequality \eqref{sum-inequality}, we have 
\begin{align*}  
J_{u1}^{2}
&=\frac{\tau \theta }{2} \sum_{j=1}^k w_j 
{\rm Im}\big( \nabla P_h[e^r_{nj} 
g(u_{nj}^\phi)u_{nj}^\phi], 
P_\tau^n [\nabla P_\tau^n e^u_{nj}(t_n-t_{nj}) ]\big)  \\
&\le 
\frac{\tau}{2}\sum_{j=1}^k w_j |e^r_{nj}|\| \nabla P_h[g( u_{nj}^\phi)u_{nj}^\phi]\|
\|P_\tau^n [\nabla P_\tau^n e^u_{nj}(t_n-t_{nj}) ]\| \nonumber\\
&\le \frac{\tau}{2}\sum_{j=1}^k w_j |e^r_{nj}|
\|\nabla [g(u_{nj}^\phi) u_{nj}^\phi]\| 
\|P_\tau^n [\nabla P_\tau^n e^u_{nj}(t_n-t_{nj}) ]\| 
\quad\mbox{(here \eqref{Ph-H1-stability} is used)} \nonumber\\
&\le
 C\tau \sum_{j=1}^k w_j |e^r_{nj}|
\|P_\tau^n [\nabla P_\tau^n e^u_{nj}(t_n-t_{nj}) ]\|  
\quad\mbox{(here \eqref{Linfty-uh-rh} is used)}\nonumber \\
&\le  \sum_{j=1}^k \frac{\tau}{8\tau^2} w_j\|P_\tau^n [\nabla P_\tau^n e^u_{nj}(t_n-t_{nj}) ]\|^2
+ C\tau^3 \sum_{j=1}^k w_j |e^r_{nj}|^2 
\nonumber\\ 
&\le \frac{1}{4\tau^2} \int_{I_n} \|P_\tau^n [\nabla P_\tau^n e^u_h(t)(t_n-t) ]\|^2 \d t
+ 2C\tau^2 \int_{I_n}|e^r_h|^2 \, \d t 
\nonumber\\ 
&\le \frac{1}{4}  \int_{I_n}\|\nabla P_\tau^n e^u_h\|^2 \, \d t
+ 2C\tau^2 \int_{I_n}|e^r_h|^2 \, \d t , \nonumber  
\end{align*}
\begin{align*}
J_{u1}^{3} &= \frac{\tau}{2}\sum_{j=1}^k w_j 
{\rm Im}\big(r_{nj}^* \nabla P_h\big[(g(u_{nj}^\phi)u_{nj}^\phi - g(u_{nj}^*)u_{nj}^*) \big], 
P_\tau^n [\nabla P_\tau^n e^u_{nj}(t_n-t_{nj}) ]\big) \\
&\le \frac{\tau}{2}\sum_{j=1}^k w_j
|r_{nj}^*| \|\nabla P_h\big[(g(P_\tau^n u_{nj}^\phi)P_\tau^n u_{nj}^\phi - g(P_\tau^n u_{nj}^*)P_\tau^n u_{nj}^*) \big]\| \\
&\hskip 1in \times \|P_\tau^n [\nabla P_\tau^n e^u_{nj}(t_n-t_{nj}) ]\| 
 \qquad\mbox{(here $u_{nj}^\phi=P_\tau^n u_{nj}^\phi$ is used)} \nonumber\\
&\le C\tau\sum_{j=1}^k w_j \|\nabla P_\tau^n e^u_{nj}\| \|P_\tau^n [\nabla P_\tau^n e^u_{nj}(t_n-t_{nj}) ]\| 
 \qquad\mbox{(here \eqref{Linfty-uh-rh} is used)} \nonumber\\
&\le C\tau \int_{I_n}\|\nabla P_\tau^n  e^u_h\|^2 \, \d t + C\tau^{-1}
\int_{I_n}\|P_\tau^n [\nabla P_\tau^n e^u_{nj}(t_n-t_{nj}) ]\|^2 \, \d t \nonumber\\
&\le C\tau  \int_{I_n}\|\nabla P_\tau^n  e^u_h\|^2 \, \d t,  \nonumber \\
J_{u1}^{4} &=-\int_{I_n} {\rm Im} \big(\nabla P_h P_\tau^n d_u^n, P_\tau^n [\nabla P_\tau^n  e^u(t)(t_n-t)] \big) \d t \\
&\le C\tau^2 \int_{I_n} \|\nabla P_hP_\tau^n d_u^n\|^2 \d t
+ \frac{1}{4\tau^2} \int_{I_n} \|P_\tau^n [\nabla P_\tau^n  e^u_{nj}(t_n-t_{nj})] \|^2 \d t \nonumber\\
&\le C\tau^3 \max_{t\in I_n} \|d_u^n\|_{H^1}^2
+ \frac{1}{4} \int_{I_n} \|\nabla P_\tau^n e^u_h\|^2 \d t .  \nonumber
\end{align*}
Substituting \eqref{err-Ju1-1} and the estimates of $J_{u1}^m$, $m=1,2,3,4$, into \eqref{pn-err-u}, for sufficiently small step size $\tau$ we obtain 
\begin{align*}
&\int_{I_n} \|\nabla P_\tau^n e^u_h\|^2 \d t \\
&\le C\tau  \big(\nabla e^u_h(t_{n-1}), \nabla P_\tau^n e^u_h(t_{n-1})\big)
+C\tau^2 \int_{I_n}  |e^r_h|^2 \d t   + C\tau^3 \max_{t\in I_n} \|d_u^n\|_{H^1}^2  \nonumber\\
&\le C\tau  \|\nabla e^u_h(t_{n-1})\| \|\nabla P_\tau^n e^u_h(t_{n-1})\|
+C\tau^2 \int_{I_n}  |e^r_h|^2 \d t  + C\tau^3 \max_{t\in I_n} \|d_u^n\|_{H^1}^2\nonumber\\
&\le C\tau \|\nabla e^u_h(t_{n-1})\| 
\bigg(\frac{1}{\tau}\int_{I_n} \|\nabla P_\tau^n e^u_h\|^2 \d t\bigg)^{\frac12}
+C\tau^2 \int_{I_n} |e^r_h|^2 \d t  + C\tau^3 \max_{t\in I_n} \|d_u^n\|_{H^1}^2 \nonumber\\
&\le C\tau \|\nabla e^u_h(t_{n-1})\|^2 
+\frac12\int_{I_n} \|\nabla P_\tau^n e^u_h\|^2 \d t 
+C\tau^2 \int_{I_n} |e^r_h|^2 \d t  
+C\tau^3 \max_{t\in I_n} \|d_u^n\|_{H^1}^2,
\end{align*}
which then implies 
\begin{align*}
\int_{I_n} \|\nabla P_\tau^n e^u_h\|^2 \d t 
&\le C\tau \|\nabla e^u_h(t_{n-1})\|^2 
+C\tau^2 \int_{I_n} |e^r_h|^2 \d t 
+ C\tau^3 \max_{t\in I_n} \|d_u^n\|_{H^1}^2 .
\end{align*}
By using inequality \eqref{L2-uh-Phuh}, we can remove the operator $P_\tau^n$ in the above inequality (meanwhile replace the $H^1$ seminorm by the full norm), i.e.,  
\begin{align}\label{pn-err-u-com}
\int_{I_n} \|e^u_h\|_{H^1}^2 \d t 
&\le C\tau \|e^u_h(t_{n-1})\|_{H^1}^2 
+C\tau^2 \int_{I_n} |e^r_h|^2 \d t 
+ C\tau^3 \max_{t\in I_n} \|d_u^n\|_{H^1}^2 .
\end{align} 

{\em Step 2: Estimation of $\|e^r_h\|_{L^2(I_n)}$.}
To estimate the second term on the right-hand side of \eqref{pn-err-u-com}, we proceed similarly as \eqref{err-Tu2-pr}, that is,   
\begin{align}\label{pn-err-r-pr}
\int_{I_n} |P_\tau^n e^r_h|^2 \d t 
&=\int_{I_n} e^r_h(t) \, P_\tau^n e^r_h(t) \, \d t  \\
&=\int_{I_n} 
\bigg[e^r_h(t_{n-1}) P_\tau^n e^r_h(t_{n-1}) 
+\int_{t_{n-1}}^t \partial_s\big(e^r_h(s) P_\tau^n e^r_h(s)\big) \d s \bigg] \d t \nonumber\\
&=\tau e^r_h(t_{n-1}) P_\tau^n e^r_h(t_{n-1}) 
+\int_{I_n}  \partial_t e^r_h(t) P_\tau^n \big[P_\tau^n e^r_h(t)(t_n-t)\big] \d t \nonumber\\
&\qquad 
+\int_{I_n} e^r_h(t) (t_n-t)\partial_t P_\tau^n e^r_h(t) \d t \nonumber \\
&=\tau  e^r_h(t_{n-1}) P_\tau^n e^r_h(t_{n-1}) 
+\int_{I_n} \partial_t e^r_h(t) P_\tau^n \big[P_\tau^n e^r_h(t)(t_n-t)\big] \d t \nonumber\\
&\qquad +\int_{I_n} P_\tau^n e^r_h(t) (t_n-t)\partial_t P_\tau^n e^r_h(t) \d t \nonumber \\
&=:\tau e^r_h(t_{n-1}) P_\tau^n e^r_h(t_{n-1}) + J_{r1} + J_{r2} , \nonumber
\end{align}
where we have changed the order of integration in the third equality, and used the the following identity in the second to last equality:
$$
\int_{I_n}  e^r_h(t) (t_n-t)\partial_t P_\tau^n e^r_h(t) \d t
=\int_{I_n}  P_\tau^n e^r_h(t) (t_n-t)\partial_t P_\tau^n e^r_h(t)\d t,
$$
which holds because $(t_n-t)\partial_t P_\tau^n e^r_h(t)$ is a polynomial of degree $k-1$ in time. 

Using integration by parts, we have  
\begin{align}\label{err-Tr2}
J_{r2} &=\int_{I_n}  \frac12 \frac{\d}{\d t} |P_\tau^n e^r_h(t)|^2(t_n-t) \d t  \\
&= -\frac12 |P_\tau^n e^r_h(t_{n-1})|^2\tau 
+ \int_{I_n}  \frac12 |P_\tau^n e^r_h(t)|^2 \d t 
\le \frac12 \int_{I_n}   |P_\tau^n e^r_h(t)|^2 \d t.  \nonumber
\end{align}
Then substituting \eqref{err-Tr2} into \eqref{pn-err-r-pr} yields 
\begin{align*}
\int_{I_n} |P_\tau^n e^r_h|^2 \d t 
&\le 2\tau \big(e^r_h(t_{n-1}), P_\tau^n e^r_h(t_{n-1})\big) 
+ 2J_{r1} \nonumber\\
&\le C\tau|e^r_h(t_{n-1})| \|P_\tau^n e^r_h\|_{L^\infty(I_n)}
+ 2J_{r1} \nonumber\\
&\le C\tau|e^r_h(t_{n-1})| \bigg(\frac{1}{\tau}\int_{I_n} |P_\tau^n e^r_h|^2 \d t\bigg)^{\frac12}
+ 2J_{r1}  \\
&\le C\tau|e^r_h(t_{n-1})|^2
+ \frac{1}{2} \int_{I_n} |P_\tau^n e^r_h|^2 \d t 
+ 2J_{r1} ,
\end{align*}
which then implies 
\begin{align}\label{pn-err-r}
\int_{I_n} |P_\tau^n e^r_h|^2 \d t 
&\le C\tau|e^r_h(t_{n-1})|^2 + 4 J_{r1} ,
\end{align}

In order to estimate $J_{r1}$, we choose $q_h = P_\tau^n \big[P_\tau^n e^r_h(t)(t_n-t)\big]$ in \eqref{GL-rh-err}, which yields the following identity:  
\begin{align} \label{err-Jr1}
&J_{r1} 
=\int_{I_n} 
\partial_te^r_h(t) P_\tau^n \big[P_\tau^n e^r_h(t) (t_n-t)\big]\d t \\
&\quad =-\frac{\tau}{4}\sum_{j=1}^k w_j 
{\rm Re}\Big( 
P_\tau^n [P_\tau^n e^r_{nj}(t_n-t_{nj}) ] \big(g(u_{nj}^\phi) u_{nj}^\phi -g(u_{nj}^*)u_{nj}^*\big) , \partial_t u_h^*(t_{nj}) \Big)
\nonumber\\
&\qquad + \frac{\tau\theta}{4} \sum_{j=1}^k w_j 
{\rm Re}\Big(  
P_\tau^n [P_\tau^n e^r_{nj}(t_n-t_{nj}) ] g(u_{nj}^\phi)u_{nj}^\phi  ,  \partial_t e^u_h(t_{nj})  \Big)
\nonumber\\
&\qquad  - \int_{I_n} P_\tau^nd_r^n P_\tau^n [P_\tau^n e^r(t)(t_n-t) ] \d t  
=\!: \sum_{m = 1}^3 J_{r1}^{m} .  \nonumber
\end{align}
By using \eqref{sum-inequality}, we have 
\begin{align} \label{err-Jr1-1}
J_{r1}^1 
&\le \frac{\tau}{4}\sum_{j=1}^k w_j 
|P_\tau^n [P_\tau^n e^r_{nj}(t_n-t_{nj}) ]| 
\|g(u_{nj}^\phi) u_{nj}^\phi -g(u_{nj}^*) u_{nj}^* \|
\|\partial_t u_h^*(t_{nj})\| \\
&\le C\tau \sum_{j=1}^k w_j 
|P_\tau^n [P_\tau^n e^r_{nj}(t_n-t_{nj}) ]|  
\|e^u_{nj} \| \nonumber\\
&\le \frac{1}{8\tau^2}  \frac{\tau}{2}\sum_{j=1}^k w_j 
|P_\tau^n [P_\tau^n e^r_{nj}(t_n-t_{nj}) ]| ^2 
+C\tau^2 \frac{\tau}{2}\sum_{j=1}^k w_j \|e^u_{nj} \|^2 \nonumber\\
&\le \frac{1}{8\tau^2} 
\int_{I_n}|P_\tau^n [P_\tau^n e^r_h(t)(t_n-t) ]|^2 \, \d t
+ C\tau^2 \int_{I_n}\|P_\tau^n  e^u_h \|^2 \, \d t
\nonumber\\
&\le  \frac{1}{8} \int_{I_n}|P_\tau^ne^r_h|^2 \, \d t
+ C\tau^2 \int_{I_n}\|e^u_h\|^2 \, \d t , \nonumber 
\end{align}
\begin{align}
\label{err-Jr1-2}
J_{r1}^{2} &\le 
\frac{\tau}{4}\sum_{j=1}^k w_j 
|P_\tau^n [P_\tau^n e^r_{nj}(t_n-t_{nj}) ]|
\|g(u_{nj}^\phi) u_{nj}^\phi\|_{L^{\infty}}
\| \partial_t e^u_h(t_{nj}) \| \\ 
&\le \frac{1}{8\tau^2}  \frac{\tau}{2}\sum_{j=1}^k w_j 
|P_\tau^n [P_\tau^n e^r_{nj}(t_n-t_{nj}) ]|^2
+C\tau^2 \frac{\tau}{2}\sum_{j=1}^k w_j \| \partial_t e^u_h(t_{nj}) \| ^2 \nonumber\\ 
&\le \frac{1}{8\tau^2} \int_{I_n} |P_\tau^n [P_\tau^n e^r_h(t)(t_n-t) ]|^2 \, \d t
+ C\tau^2 \int_{I_n} \|\partial_te^u_h\|^2 \, \d t \nonumber\\
&\le \frac{1}{8} \int_{I_n} |P_\tau^ne^r_h|^2 \, \d t
+ C\int_{I_n} \|e^u_h\|^2 \, \d t ,  \nonumber \\ 
\label{err-Jr1-3} J_{r1}^{3} 
&= - \int_{I_n} P_\tau^nd_r^n P_\tau^n [P_\tau^n e^r(t)(t_n-t) ] \d t \\
&\le C\tau^2 \int_{I_n}  |P_\tau^nd_r^n|^2 \d t
+ \frac{1}{8\tau^2}\int_{I_n}|P_\tau^n [P_\tau^n e^r_h(t)(t_n-t) ]|^2\d t \nonumber\\
&\le  C\tau^3 \max_{t\in I_n} |P_\tau^nd_r^n|^2
+ \frac{1}{8} \int_{I_n}|P_\tau^ne^r_h|^2\d t.  \nonumber
\end{align}
Substituting  \eqref{err-Jr1-1}--\eqref{err-Jr1-3} into \eqref{pn-err-r}--\eqref{err-Jr1}, and using \eqref{L2-uh-Phuh}, we get
\begin{align}\label{pn-err-r-com}
\int_{I_n} |e^r_h|^2 \d t 
\le C\tau |e^r_h(t_{n-1})|^2  
 +C\int_{I_n} \|e^u_h\|^2 \d t 
+ C\tau^3 \max_{t\in I_n} |P_\tau^n d_r^n|^2 . 
\end{align}
Then, combining \eqref{pn-err-u-com} and \eqref{pn-err-r-com}, we obtain (for sufficiently small $\tau$)
\begin{align}\label{eu-L2H1-1}
&\int_{I_n} \|e^u_h\|_{H^1}^2 \,\d t 
\le C\tau \bigl[\|e^u_h(t_{n-1})\|_{H^1}^2 + |e^r_h(t_{n-1})|^2 \\
&\hskip 1.2in +\tau^2 \max_{t\in I_n} \bigl(\|d_u^n\|_{H^1}^2 + |P_\tau^n d_r^n|^2 \bigr) \bigr], 
\nonumber \\ 
&\int_{I_n} |e^r_h|^2 \d t 
\le C\tau \bigl[ \|e^u_h(t_{n-1})\|_{H^1}^2 + |e^r_h(t_{n-1})|^2  
+ \tau^2 \max_{t\in I_n} \bigl(\|d_u^n\|_{H^1}^2 + |P_\tau^n d_r^n|^2 \bigr) \bigr] .
\label{er-L2-1}
\end{align}
{\em Step 3: Estimation of $\|\nabla e^u_h(t_n)\|$ and $|e^r_h(t_n)|$.}
Setting $v_h = \partial_t e^u_h$ in \eqref{GL-uh-err} and taking the real part, we get 
\begin{align} \label{err-Eu}
&\frac 12\|\nabla e^u_h(t_{n})\|^2 - \frac 12\|\nabla e^u_h(t_{n-1})\|^2  
=\frac{\theta\tau}{2}\sum_{j=1}^k w_j 
{\rm Re}\,\big(e^r_{nj} g(u_{nj}^\phi)u_{nj}^\phi ,
 \partial_t e^u_h(t_{nj}) \big) \\
&\,  +\frac{\tau}{2}\sum_{j=1}^k w_j \, {\rm Re}
\Big(r_{nj}^* \big[g(u_{nj}^\phi)u_{nj}^\phi - g(u_{nj}^*)u_{nj}^*\big] , 
 \partial_t e^u_h(t_{nj}) \Big)  
- \int_{I_n}{\rm Re}\, (d_u^n, \partial_t e^u_h) \d t \nonumber\\
&\le C\tau \sum_{j=1}^k w_j \bigl(|e^r_{nj}|
+ \|\nabla e^u_{nj}\|\bigr)\| \partial_t e^u_h(t_{nj}) \|_{H^{-1}}
+ C\int_{I_n}\|d_u^n\|_{H^1} \|\partial_t e^u_h \|_{H^{-1}} \d t \nonumber\\
&\le C\int_{I_n}(\|e^u_h \|_{H^1}^2 +|e^r_{nj}|^2 + \|d_u^n\|_{H^1}^2) \d t 
+ \frac12 \int_{I_n}\|\partial_t e^u_h\|_{H^{-1}}^2\d t . \nonumber
\end{align} 
In order to estimate the last term $\int_{I_n}\|\partial_t e^u_h\|_{H^{-1}}^2\d t$ above, we consider \eqref{GL-uh-err}, from which we can derive the following estimate for any test function $v\in L^2(I_n;H^1_0)$: 
\begin{align*}
\Bigl|\int_{I_n}(\partial_t e^u_h,v)\d t\Bigr| 
&= \Bigl| \i\int_{I_n} \big(\nabla e^u_h, \nabla P_hP_\tau^nv\big)\, \d t 
- \i \frac{\theta \tau}{2} \sum_{j=1}^k w_j 
\Big(e^r_{nj} g(u_{nj}^\phi) u_{nj}^\phi, P_hP_\tau^n v(t_{nj}) \Big)\\
&\quad\, 
 -\i \frac{\tau}{2} \sum_{j=1}^k w_j 
\Big(r_{nj}^* \big[g(u_{nj}^\phi)u_{nj}^\phi - g(u_{nj}^*) u_{nj}^*\big], P_hP_\tau^n v(t_{nj}) \Big)\\
&\quad\, 
+ \i \int_{I_n} (d_u^n, P_hP_\tau^nv) \d t \Bigr|  \\
&\le C \bigl( \|e^u_h\|_{L^2(I_n;H^1)}+|e^r_h|_{L^2(I_n)} + \|d_u^n\|_{L^2(I_n;H^{-1})} \bigr)
\|P_hP_\tau^nv\|_{L^2(I_n;H^1_0)} \\
&\le C \bigl( \|e^u_h\|_{L^2(I_n;H^1)}+|e^r_h|_{L^2(I_n)} + \|d_u^n\|_{L^2(I_n;H^{-1})} \bigr)
\|v\|_{L^2(I_n;H^1_0)} .
\end{align*}
By the duality between $L^2(I_n;H^{-1})$ and $L^2(I_n;H^1_0)$, we obtain
\begin{align}\label{dteu-L2H-1}
\int_{I_n}\|\partial_t e^u_h\|_{H^{-1}}^2\d t
\le 
C\int_{I_n}(\|e^u_h \|_{H^1}^2+|e^r_h|^2+\|d_u^n\|_{H^{-1}}^2)\d t .
\end{align}
Then, summing up \eqref{err-Eu} and \eqref{dteu-L2H-1}, we have 
\begin{align}  \label{H1-euh-tn}
\|\nabla e^u_h(t_{n})\|^2 - \|\nabla e^u_h(t_{n-1})\|^2 
&+\int_{I_n}\|\partial_t e^u_h\|_{H^{-1}}^2\d t  \\
&\le C\int_{I_n} \bigl(\|e^u_h \|_{H^1}^2+|e^r_h|^2 + \|d_u^n\|_{H^1}^2\bigr)\d t . \nonumber
\end{align} 

Setting $q_h=2e^r_h$ in \eqref{GL-rh-err} yields 
\begin{align}\label{erh-tn}
|e^r_h(t_{n})|^2 -|e^r_h(t_{n-1})|^2  
&=\frac{\tau}{2}\sum_{j=1}^k w_j 
{\rm Re}\big( e^r_{nj} \big( g(u_{nj}^\phi) u_{nj}^\phi -g(u_{nj}^*) u_{nj}^* \big) , \partial_t u_h^*(t_{nj}) \big)
\nonumber \\
&\quad\, + \frac{\theta\tau}{2}\sum_{j=1}^k w_j 
{\rm Re}\big( e^r_{nj} g(u_{nj}^\phi)u_{nj}^\phi , \partial_t e^u_h(t_{nj}) \big) \nonumber \\
&\quad\,
- \int_{I_n} P_\tau^n d_r^n 2e^r \d t    \\ 
&\le C\tau \sum_{j=1}^k w_j |e^r_{nj}| \|e^u_{nj}\|  
+ C\tau \sum_{j=1}^k w_j |e^r_{nj}| \|\partial_t e^u_{nj}\|_{H^{-1}} \nonumber \\
&\quad\,
+ C\int_{I_n} |P_\tau^n d_r^n| |e^r| \d t \nonumber \\
&\le C\int_{I_n} \bigl( \|e^u_h\|^2+|e^r_h|^2 + \|\partial_t e^u_h\|_{H^{-1}}^2 \bigr) \d t \nonumber 
+C\int_{I_n} |P_\tau^n d_r^n |^2 \d t \\
&\le C\int_{I_n} \bigl( \|e^u_h\|_{H^1}^2+|e^r_h|^2 \bigr) \d t
+C \int_{I_n} (\|d_u^n\|_{H^1}^2 + |P_\tau^n d_r^n |^2) \d t ,  \nonumber
\end{align} 
where we have used \eqref{dteu-L2H-1} to obtain the last inequality. 

{\em Step 4: Completion of the proof.} 
Summing up \eqref{H1-euh-tn} and \eqref{erh-tn} yields 
\begin{align}
\|\nabla e^u_h(t_{n})\|^2 +|e^r_h(t_{n})|^2
- \|\nabla e^u_h(t_{n-1})\|^2 -|e^r_h(t_{n-1})|^2 
+\int_{I_n}\|\partial_t e^u_h\|_{H^{-1}}^2\d t  \\
 \le
C\int_{I_n}  \bigl( \|e^u_h\|_{H^1}^2+|e^r_h|^2 \bigr) \d t 
+C \int_{I_n}(\|d_u^n\|_{H^1}^2 + |P_\tau^n d_r^n |^2) \d t . \nonumber
\end{align} 
Then, substituting \eqref{eu-L2H1-1}--\eqref{er-L2-1} into the inequality above, we obtain 
\begin{align}
\bigl(\|\nabla e^u_h(t_{n})\|^2 +|e^r_h(t_{n})|^2 \bigr) 
- \bigl(\|\nabla e^u_h(t_{n-1})\|^2 +|e^r_h(t_{n-1})|^2 \bigr)
+\int_{I_n}\|\partial_t e^u_h\|_{H^{-1}}^2\d t \\
\le C\tau \bigl(\|\nabla e^u_h(t_{n-1})\|^2 + |e^r_h(t_{n-1})|^2 \bigr) 
+C\int_{I_n}(\|d_u^n\|_{H^1}^2 + |P_\tau^n d_r^n |^2) \d t. \nonumber
\end{align} 
It follows from Gronwall's inequality that 
\begin{align}\label{L2H1-dt-euh}
\max_{1\le n\le N}
&\bigl(\|\nabla e^u_h(t_{n})\|^2 +|e^r_h(t_{n})|^2 \bigr) 
+C\int_0^T\|\partial_t e^u_h\|_{H^{-1}}^2\, \d t  \\
&\le 
C\big( \|e^u_h(0)\|_{H^1}^2 +|e^r_h(0)|^2 \big) 
+ C\sum_{n=1}^N \int_{I_n} (\|d_u^n\|_{H^1}^2 + |P_\tau^n d_r^n |^2) \d t. \nonumber
\end{align}
Then, substituting this inequality into \eqref{eu-L2H1-1}--\eqref{er-L2-1} and using temporal inverse inequlaity, we obtain 
\begin{align}\label{H1-L2-euh-erh}
&\max_{t\in [0,T]}
\bigl( \|e^u_h(t)\|_{H^1}^2 +|e^r_h(t)|^2 \bigr)  \nonumber \\
&\le
C\Big[ \|e^u_h(0)\|_{H^1}^2 +|e^r_h(0)|^2
+\max_{1\le n\le N}\max_{t \in I_n}\bigl( \|d_u^n\|_{H^1}^2 + |P_\tau^n  d_r^n|^2 \bigr) 
\Big] .
\end{align}
Hence, \eqref{H1-stability} holds. 

When $\tau$ and $h$ are sufficiently small, inequality \eqref{H1-L2-euh-erh} implies that 
\begin{align}\label{H1-euh-er-small}
\max_{t\in [0,T]}\|e^u_h(t)\|_{H^1}\le \frac12 
\qquad\mbox{and}\qquad 
\max_{t\in [0,T]}|e^r_h(t)|  \le \frac12 . 
\end{align}
On the one hand, by the inverse inequality, we have 
\begin{align}\label{eu-Linfty-1}
\max_{t\in [0,T]}\|e^u_h(t)\|_{L^\infty}
&\le C\ell_h \max_{t\in [0,T]}\|e^u_h(t)\|_{H^1} \nonumber \\
&\le C\ell_h 
\Big[ \|e^u_h(0)\|_{H^1} +|e^r_h(0)|
+\max_{1\le n\le N}\max_{t \in I_n}\bigl( \|d_u^n\|_{H^1} + |P_\tau^n  d_r^n| \bigr) 
\Big]
, 
\end{align}
where
$$
\ell_h = \left\{
\begin{aligned}
&1 &&\mbox{if}\,\,\, d=1,\\
&\ln(2+1/h) &&\mbox{if}\,\,\, d=2,\\
&h^{-\frac{1}{2}} &&\mbox{if}\,\,\, d= 3 .
\end{aligned}
\right.
$$
On the other hand, by choosing a test function $v$ in \eqref{GL-uh-err} satisfying the  properties $v(t_{nj})=1$ and $v(t_{ni})=0$ for $i\neq j$, and using property \eqref{Ptau-tnj}, we obtain 
\begin{align}\label{discrete-H2-euh}
\|\Delta_h e^u_{nj}\| & = \Bigl\| \i\partial_t e^u_{nj}-
\theta P_h \big[e^r_{nj} g(u_{nj}^\phi)u_{nj}^\phi \big]
  + P_hd_{nj}^u  \\
  &\qquad\qquad - P_h  \big[r_{nj}^* (g(u_{nj}^\phi)u_{nj}^\phi - g(u_{nj}^*) u_{nj}^*) \big] \Bigr\| \nonumber \\
&\le C\tau^{-1} 
\Big[ \|e^u_h(0)\|_{H^1} +|e^r_h(0)|
+\max_{1\le n\le N}\max_{t \in I_n}\bigl( \|d_u^n\|_{H^1} + |P_\tau^n  d_r^n| \bigr) 
\Big]
, \nonumber
\end{align}
where we have used \eqref{L2H1-dt-euh}--\eqref{H1-L2-euh-erh} and an inverse inequality in time in estimating $\partial_t e^u_{nj}$. By the discrete Sobolev embedding inequality, for $1\leq d\leq 3$ we have  
\begin{align}\label{eu-Linfty-2}
\|e^u_{nj}\|_{L^\infty}
&\le
C\|e^u_{nj}\|_{H^1}^{\frac12}
\|\Delta_h e^u_{nj}\|^{\frac12} \nonumber \\
&\le
C\tau^{-\frac12} \max_{1\le n\le N}
\Big[ \|e^u_h(0)\|_{H^1} +|e^r_h(0)|
+\max_{1\le n\le N}\max_{t \in I_n}\bigl( \|d_u^n\|_{H^1} + |P_\tau^n  d_r^n| \bigr) 
\Big] ,
\end{align}
where we have used \eqref{H1-L2-euh-erh} and \eqref{discrete-H2-euh} in the last inequality. 
Then, combining \eqref{eu-Linfty-1} and \eqref{eu-Linfty-2} yields 
\begin{align*}
&\max_{1\le n\le N} \max_{1\le j\le k} 
\|e^u(t_{nj})\|_{L^\infty} \\
& \le
C\min(\ell_h,\tau^{-\frac12}) \Big[ \|e^u_h(0)\|_{H^1} +|e^r_h(0)|
+\max_{1\le n\le N}\max_{t \in I_n}\bigl( \|d_u^n\|_{H^1} + |P_\tau^n  d_r^n| \bigr) 
\Big]
\\
& \le C \bigl(h^{p-\frac12} + \tau^{k+\frac12} \bigr) ,
\end{align*}
where we have used the consistency estimate from Theorem \ref{THM:consistency}.
When $\tau$ and $h$ are sufficiently small, the inequality above implies 
\begin{align}\label{Linfty-euh-small}
\max_{1\le n\le N} \max_{1\le j\le k} \|e^u(t_{nj})\|_{L^\infty} \le \frac12 \, . 
\end{align}
This together with \eqref{H1-euh-er-small} gives \eqref{Linfty-e}. 

Furthermore, since $\phi_h=\theta e^u_h$ and $\varphi_h=\theta e^r_h$, it follows that 
$$
\max_{1\le n\le N} \max_{1\le j\le k} \|\phi_h(t_{nj})\|_{L^{\infty}\cap H^1}\le \frac12
\quad\mbox{and}\quad
\max_{1\le n\le N} \max_{1\le j\le k}|\varphi_h(t_{nj})| \le \frac12 ,
$$
which imply $\rho[\phi_h]=\rho[\varphi_h]=1$ in view of the definition in \eqref{def-rho}. This proves \eqref{Linfty-phi}. 
\end{proof} 

We now are ready to state and prove existence, uniqueness and convergence of numerical solutions, which comprise of the second main theorem of this paper. 

\begin{theorem}\label{THM:main}
	Let $1\leq d\leq 3$ and assume that the solution of the NLS equation \eqref{pde} is sufficiently smooth. Then there exist positive constants $\tau_0$ and $h_0$ such that when $\tau\le \tau_0$ and $h\le h_0$, the numerical method \eqref{GL} has a unique solution $(u_h,r_h) \in X_{\tau,h}^*\times Y_{\tau,h}^*$. Moreover, this solution satisfies the following error estimate:
	\begin{equation}\label{error_bound}
	\max_{t\in[0,T]}
	\Bigl( \|u_h(t)-u_h^*(t)\|_{H^1} + |r_h(t)-r_h^*(t)| \Bigr) 	\le C(h^{p}+\tau^{k+1}) . 
\end{equation}
\end{theorem}

\begin{proof}
{\em Step 1: Existence.} 
By the definition of $\mathfrak{B}$, if $(\phi_h,\varphi_h)\in \mathfrak{B}$ and $(e^u_h,e^r_h)=M(\phi_h,\varphi_h)$ then $\phi_h=\theta e^u_h$ and $\varphi_h=\theta e^r_h$. Thus \eqref{H1-stability} implies
\begin{align}
&\|(\phi_h,\varphi_h)\|_{X_{\tau,h}\times Y_{\tau,h}} 
=\|\phi_h\|_{L^\infty(0,T;H^1)}
+\|\varphi_h\|_{L^\infty(0,T)} 
\le C , 
\end{align} 
which together with Schaefer's fixed point theorem imply the existence of a fixed point 
$(\phi_h,\varphi_h)$ for the mapping $M$ (corresponding to $\theta=1$), with 
$$
(e^u_h,e^r_h)=(\phi_h,\varphi_h) , 
\quad 
u^\phi =u_h^*+\phi_h 
\quad\mbox{and}\quad 
r^\phi=r_h^*+\varphi_h , 
$$
satisfying \eqref{GL-uh-map}--\eqref{GL-rh-map}, where we have used \eqref{Linfty-phi} in the expression \eqref{def-e-phi}. Consequently, $(e^u_h,e^r_h)$ is a solution of \eqref{GL-uh-rh-err-0} with $(u_h,r_h)=(u^\phi_h,r^\varphi_h)=(u_h^*+e^u_h,r_h^*+e^r_h)$. Hence, in view of the discussions in Remark \ref{Remark:existence}, $(u_h,r_h)$ is a solution of the numerical scheme \eqref{GL}, and \eqref{Linfty-e} implies $(u_h,r_h)$ is in the set $X_{\tau,h}^*\times Y_{\tau,h}^*$  defined in \eqref{def-X-star}--\eqref{def-Y-star}. This proves existence of a numerical solution in $X_{\tau,h}^*\times Y_{\tau,h}^*$. 

\medskip
{\em Step 2:  Uniqueness.} 
Suppose that $(u_h,r_h)$ and $(\widetilde u_h,\widetilde r_h)$ in $X_{\tau,h}^*\times Y_{\tau,h}^*$ are two pairs of numerical solutions, and set $e^u_h=u_h-\widetilde u_h$ and $e^r_h=r_h-\widetilde r_h$ (abusing the notation). 
Subtracting the corresponding equations satisfied by $(u_h,r_h)$ and $(\widetilde u_h,\widetilde r_h)$ shows that $(e^u_h,e^r_h)$ satisfies equations \eqref{GL-uh-rh-err-0} with $d_u^n=d_r^n=0$. In the meantime, the definition in \eqref{def-X-star}--\eqref{def-Y-star} implies 
\begin{align}\label{error-condition-fixed-point}
\|e^u_h(t_{nj})\|_{L^{\infty}\cap H^1} \le 1
\quad\mbox{and}\quad
|e^r_h(t_{nj})|\le 1 .
\end{align}
Accordingly, $(e^u_h,e^r_h)$ is a fixed point of the mapping $M$ (corresponding to $\theta=1$ in $\mathfrak{B}$) in the case $e^u_h(0)=e^r_h(0)=0$ and $d_u^n=d_r^n=0$. 
Hence, an application of \eqref{H1-stability} yields
\begin{align*}
&\|e^u_h\|_{L^\infty(0,T;H^1)} + \|e^r_h\|_{L^\infty(0,T)} \\ 
&\le C \Big[ \|e^u_h(0)\|_{H^1} +|e^r_h(0)|
+\max_{1\le n\le N}\max_{t \in I_n}\bigl( \|d_u^n\|_{H^1} + |P_\tau^n  d_r^n| \bigr) 
\Big]
=0. 
\end{align*}  
Thus, $(u_h,r_h)=(\widetilde u_h,\widetilde r_h)$ and the uniqueness of the numerical solution is proved. 

\medskip
{\em Step 3: Error estimate.} 
Since the error functions $e^u_h=u_h-u_h^*$ and $e^r_h=r_h-r_h^*$ satisfy \eqref{GL-uh-rh-err-0} and \eqref{error-condition-fixed-point}, it follows that $(e^u_h,e^r_h)$ is a fixed point of the mapping $M$ (corresponding to $\theta=1$ in $\mathfrak{B}$). Hence, an application of \eqref{H1-stability} yields 
\begin{align*}
&\|e^u_h\|_{L^\infty(0,T;H^1)} + \|e^r_h\|_{L^\infty(0,T)} \\
&\le C
\Big[ \|e^u_h(0)\|_{H^1} +|e^r_h(0)|
+\max_{1\le n\le N}\max_{t \in I_n}\bigl( \|d_u^n\|_{H^1} + |P_\tau^n  d_r^n| \bigr) 
\Big]
 .
\end{align*} 
Substituting the consistency error estimates from Theorem \ref{THM:consistency} into the above 
inequality yields the desired estimate \eqref{error_bound}. The proof is complete.
\end{proof}

\begin{remark}
{\upshape
For the periodic and Neumann boundary conditions, the mass and energy conservations in Theorem \ref{THM:conservation} and the error estimate in Theorem \ref{THM:main} can be proved similarly. 
}
\end{remark}

\section{Numerical experiments}\label{sec-5}
In this section, we present some one-dimensional numerical tests to  validate the theoretical 
results proved in Theorems \ref{THM:conservation} and \ref{THM:main} about the mass and energy conservations, and the convergence rates of the proposed method. All the computations are 
performed using the software package FEniCS (\url{https://fenicsproject.org}). 

We consider the cubic nonlinear Schr\"odinger equation 
\begin{align}\label{nls_sol}
\begin{aligned}
\i \partial_t u - \partial_{xx} u -2|u|^2 u 
&=0     &&\qquad\mbox{in}\,\,\, (-L,L)\times(0,T],\\
u|_{t=0} &=u_0 &&\qquad \mbox{in}\,\,\,(-L,L) ,\quad \mbox{with $L=20$},
\end{aligned}
\end{align}
subject to the periodic boundary condition. We choose $u_0 = \sech(x)\exp(2\i x)$ 
  so that the exact solution is given by 
\begin{align}\label{exact_sol}
u(x,t) = \sech(x+4t)\exp(\i(2x+3t)).
\end{align}
This example contains a soliton wave and is often used as a benchmark for meansuring the effectiveness of numerical methods for the NLS equation; see \cite{Taghizadeh2011Exact, Xu2005Local,lu2015mass}. 

\subsection{Convergence rates}
We solve problem \eqref{nls_sol} by the proposed method \eqref{GL} and compare the numerical solutions with the exact solution \eqref{exact_sol}. Newton's method is used to solve the nonlinear system. The iteration is stopped when the error is below $10^{-10}$. 

The time discretization errors are presented in Table \ref{table_time_errors_p3}, where we have used finite elements of degree $3$ with a sufficiently spatial mesh $h=2L/5000$ so that the error from spatial discretization is negligibly small in observing the temporal convergence rates.  
From Table \ref{table_time_errors_p3} we see that the error of time discretization 
is $O(\tau^{k+1})$, which is consistent with the result proved in Theorem \ref{THM:main}. 

The spatial discretization errors are presented in Table \ref{table_space_errors_gauss3}, where we have chosen $k=3$ with a sufficiently small time stepsize $\tau=1/1000$ so that the time discretization 
error is negligibly small compared to the spatial error.  
From Table \ref{table_space_errors_gauss3} we see that the spatial discretization errors are $O(h^{p})$ in the $H^1$ norm. This is also consistent with the result proved in Theorem \ref{THM:main}. 

\begin{table}[htp]\centering\small\footnotesize
\caption{Time discretization errors of the proposed method, with $h=\frac{2L}{5000}$ and $T = 1$. }
\setlength{\tabcolsep}{7mm}{
\begin{tabular}{crcc}
\toprule
$k$ &$\tau$&\multicolumn{2}{c}{$p = 3$}\\
\cmidrule(lr){3-4}
&&
$\|u(x, t) - u_h(x, t)\|_{L^\infty(0,T;H^1)}$ &order\\
\midrule
\multirow{5}{*}{$2$}    
&  1/60&     3.7964E--05&     --\\
&  1/70&     2.3429E--05& 3.1312\\
&  1/80&     1.5460E--05& 3.1132\\
&  1/90&     1.0733E--05& 3.0985\\
& 1/100&     7.7542E--06& 3.0853\\
\\
\multirow{5}{*}{$3$}  
&  1/20&     3.4019E--05&     --\\
&  1/25&     1.3821E--05& 4.0364\\
&  1/30&     6.6322E--06& 4.0275\\
&  1/35&     3.5689E--06& 4.0200\\
&  1/40&     2.0886E--06& 4.0123\\
\\
\multirow{5}{*}{$4$}  
&   1/8&     1.2291E--04&     --\\
&  1/12&     1.5120E--05& 5.1681\\
&  1/14&     6.8492E--06& 5.1369\\
&  1/16&     3.4634E--06& 5.1067\\
&  1/20&     1.1555E--06& 4.9192\\
\bottomrule
\end{tabular}}
\label{table_time_errors_p3}
\vspace{20pt}

\caption{Spatial discretization errors of the proposed method, with $\tau=\frac{1}{1000}$ and $T = 1$. }
\setlength{\tabcolsep}{7mm}{
\begin{tabular}{crccc}
\toprule
$p$&$M$
&\multicolumn{2}{c}{$k = 3$}\\
\cmidrule(lr){3-4}
&&$\|u(x, t) - u_h(x, t)\|_{L^\infty(0,T;H^1)}$ &order\\
\midrule
\multirow{5}{*}{$1$}  
&1400&       5.8670E--02&     --\\
&1600&       5.1134E--02& 1.0295\\
&1800&       4.5330E--02& 1.0229\\
&2000&       4.0719E--02& 1.0183\\
&2200&       3.6964E--02& 1.0149\\
\\
\multirow{5}{*}{$2$}  
& 240&      1.9306E--02&     --\\
& 260&      1.6438E--02& 2.0094\\
& 280&      1.4167E--02& 2.0062\\
& 300&      1.2338E--02& 2.0041\\
& 320&      1.0842E--02& 2.0027\\
\\
\multirow{5}{*}{$3$} 
&  90&       1.6147E--02&     --\\
& 100&       1.1661E--02& 3.0894\\
& 110&       8.7112E--03& 3.0599\\
& 120&       6.6844E--03& 3.0436\\
& 130&       5.2435E--03& 3.0334\\
\bottomrule
\end{tabular}}
\label{table_space_errors_gauss3}
\end{table}


\subsection{Mass and energy conservations}


We denote the mass and SAV energy of a numerical solution by 
\begin{align}\label{dis-SAV-mass}
M_h(t)=\int_\Omega |u_h(t)|^2\d x
\quad\mbox{and}\quad 
E_h(t)=\frac12\int_\Omega |\nabla u_h(t)|^2 \d x - r_h(t)^2 ,
\end{align}
respectively. 
The evolution of mass and SAV energy of the numerical solutions is presented in Figure \ref{fig_mass_energy} with $\tau=0.2$ and $h=0.2$. It is shown that 
$$
{\rm mass} = 2 + O(10^{-12})
\quad\mbox{and}\quad 
{\rm SAV\,\, energy} =  -7.33358048516+  O(10^{-12}) ,
$$ 
which are much smaller than the error of the numerical solutions, as shown in Figure \ref{fig_mass_energy_error}. This shows the effectiveness of the proposed method in preserving mass and energy (independent of the error of numerical solutions). The number of iterations at each time level is presented in Figure \ref{fig_iter_num} to show the effectiveness of the Newton's method.

\begin{figure}[htp]
\centerline{
\includegraphics[width=2.6in]{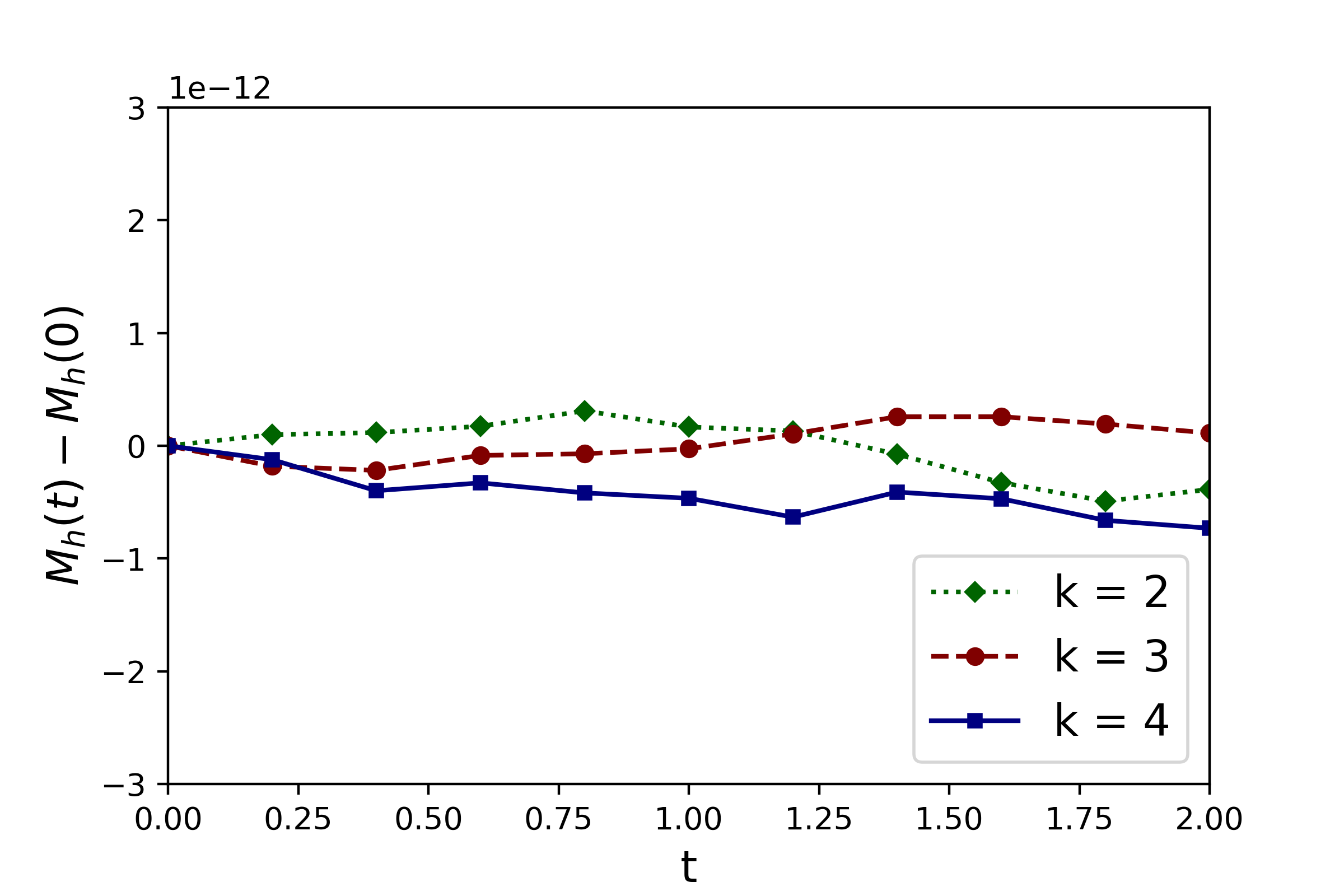}\,
\includegraphics[width=2.6in]{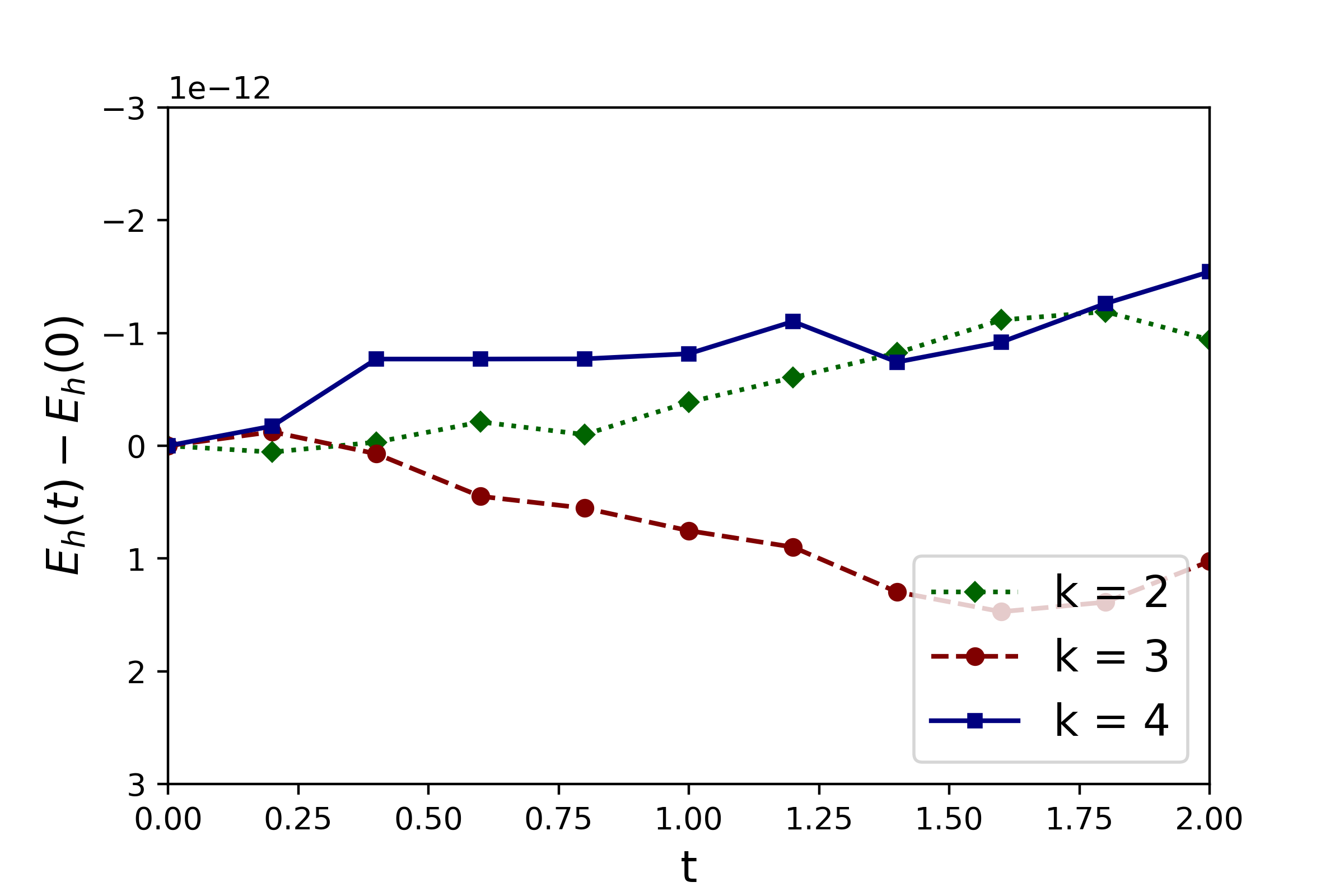}
}
\vspace{-8pt}
\caption{Evolution of mass $M_h(t)-M_h(0)$ and SAV energy $E_h(t)-E_h(0)$, with $p=3$ and $\tau=h=0.2$.}
\label{fig_mass_energy}
\smallskip 
\centerline{
\includegraphics[width=2.6in]{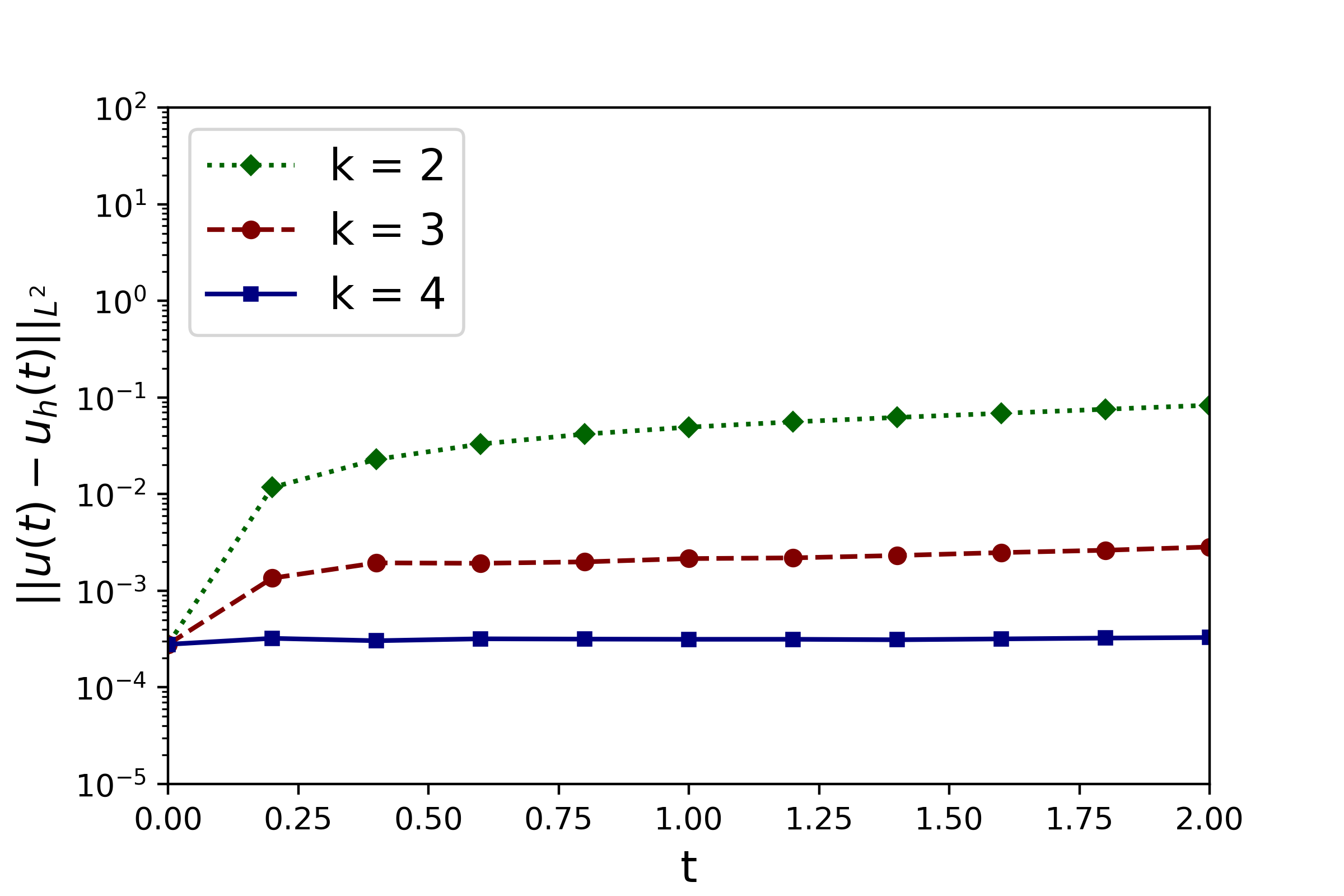}\,
\includegraphics[width=2.6in]{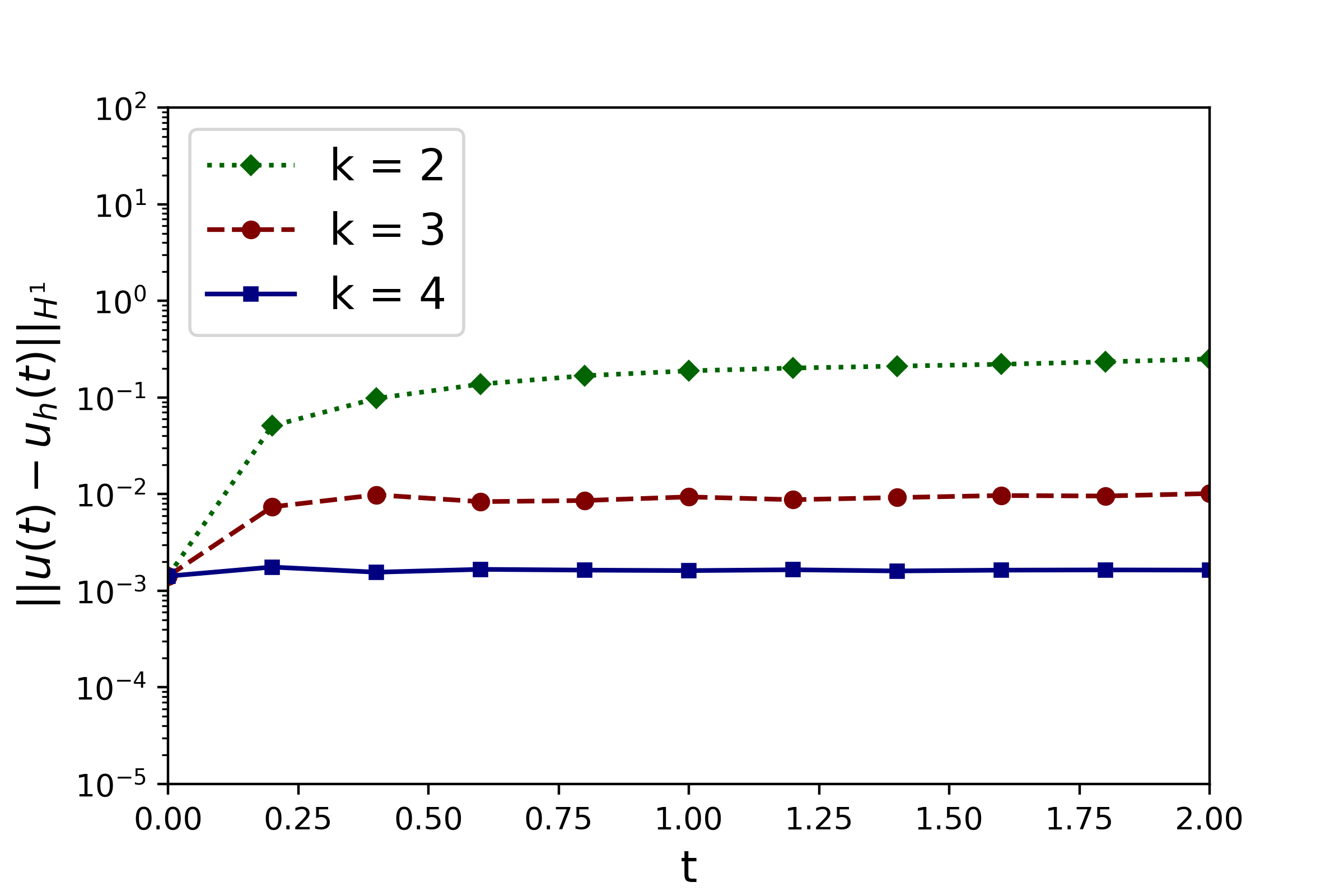}
}
\vspace{-8pt}
\caption{Evolution of error of the numerical solution, with $p=3$ and $\tau=h=0.2$.}
\label{fig_mass_energy_error}
\centering
\includegraphics[width=2.6in]{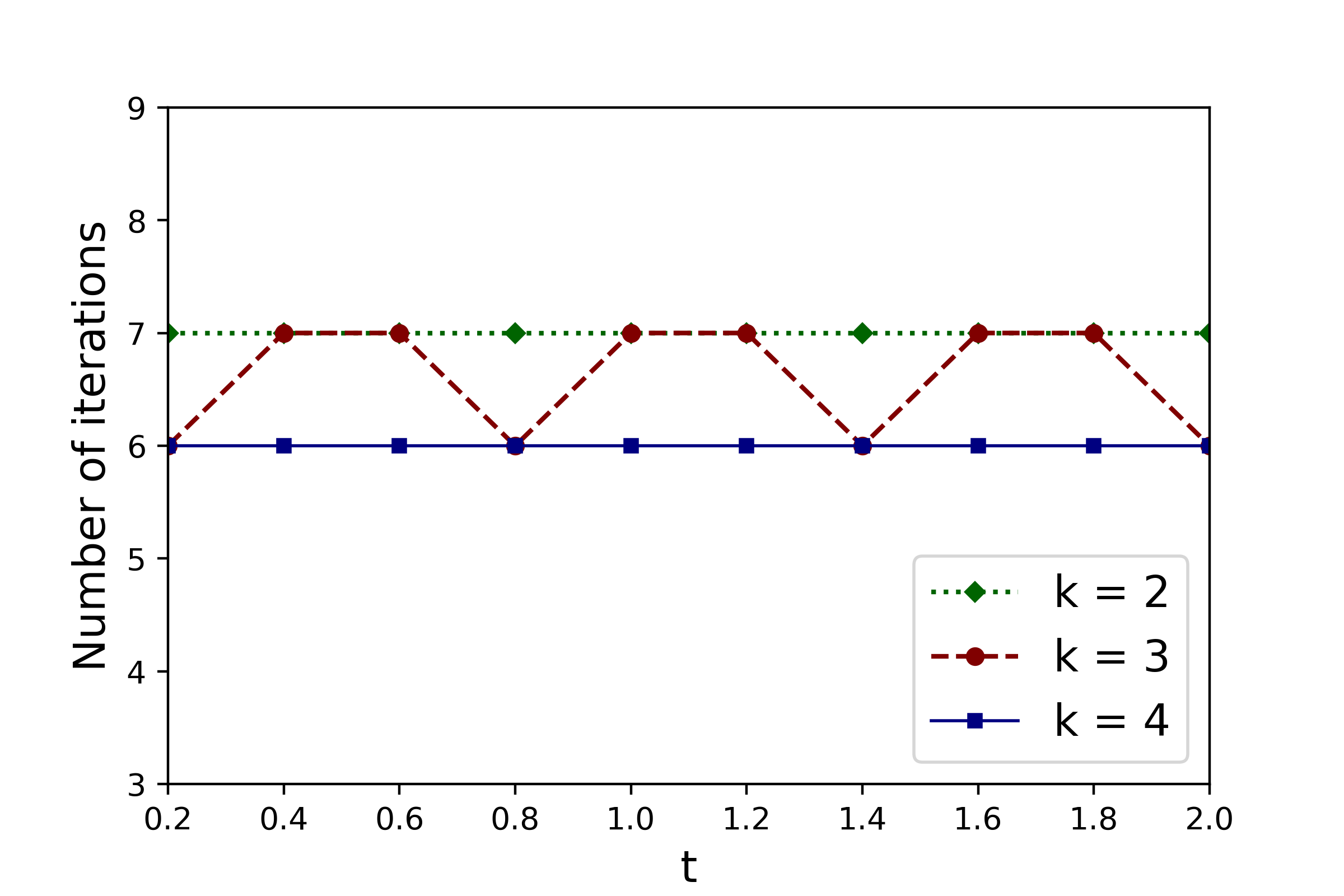}
\vspace{-8pt}
\caption{ Number of iterations at each time level, with $p=3$ and $\tau=h=0.2$.}
\label{fig_iter_num}
%
\centering
\includegraphics[width=2.6in]{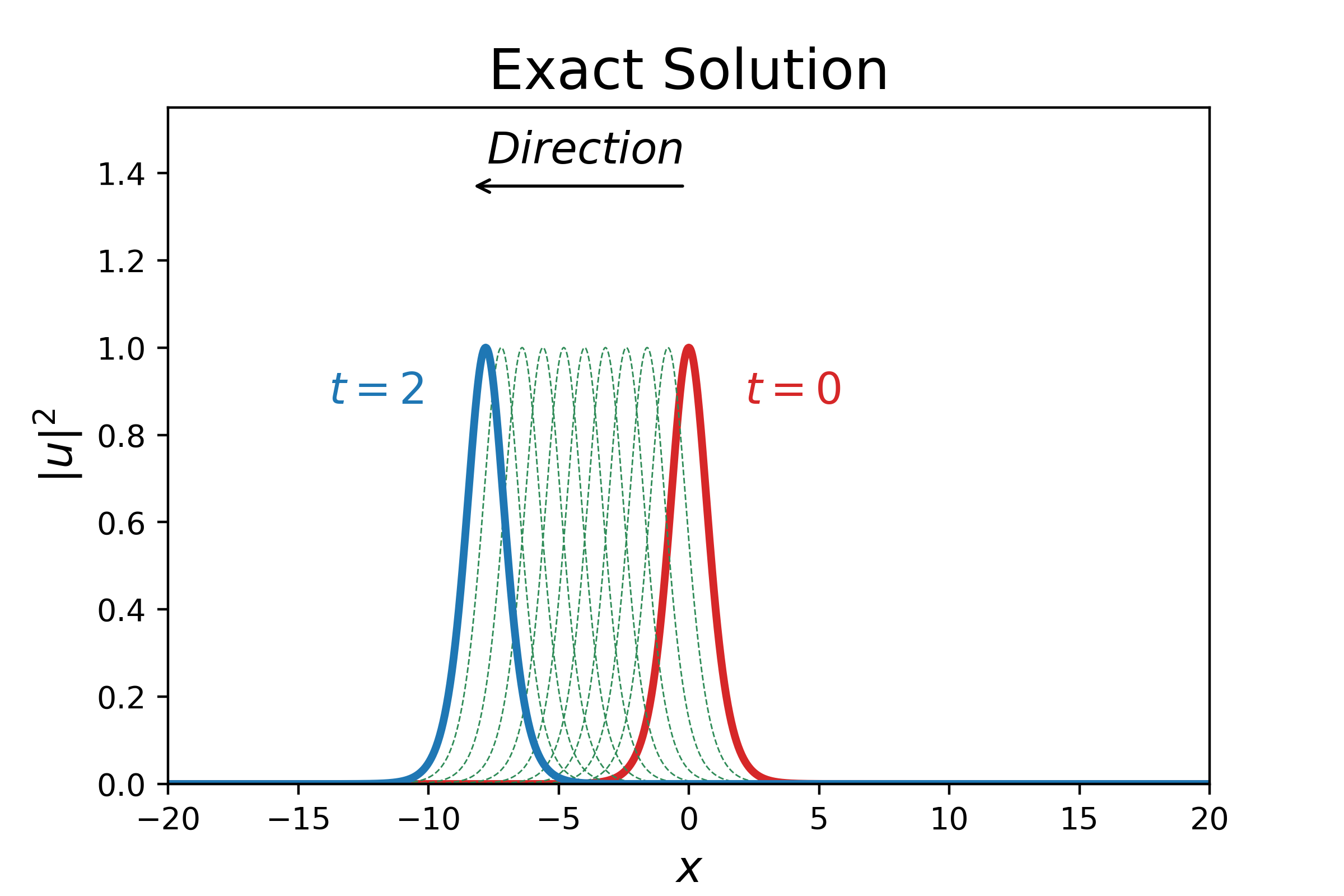}
\caption{Soliton propagation when $t\in [0, 2]$: graph of the exact solution $|u(\cdot,t)|$.} 
\label{fig_solotion_exact}
\end{figure}

	

\subsection{Comparison of different methods in preserving the shape of a soliton}

The graph of $|u(x,t)|$ is a soliton propagating towards left. Its shape remains unchanged for all
$t\geq 0$ as shown in Figure~\ref{fig_solotion_exact}. 
The graphs of numerical solutions given by several different numerical methods using the same mesh sizes are presented in Figures \ref{fig_solotion_methods2} and \ref{fig_solotion_methods}. All the methods preserve mass and energy conservations. The numerical results show the effectiveness of the proposed method in preserving the shape of the soliton. 
\begin{figure}[htp]
\centerline{
\subfigure[]{
\label{fig_solotion_methods2:a}
\includegraphics[width=2.6in]{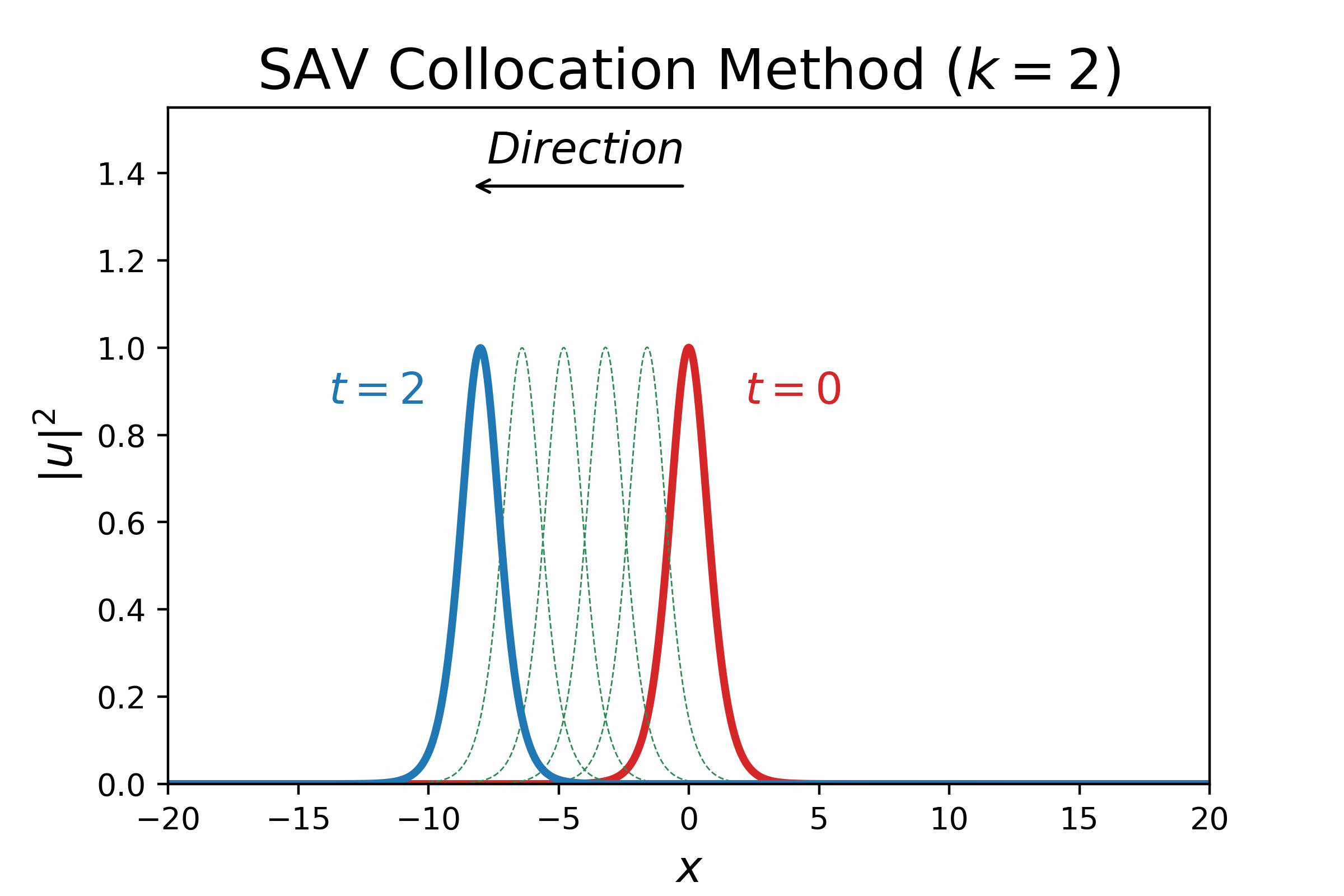}}
\subfigure[]{
\label{fig_solotion_methods2:b}
\includegraphics[width=2.6in]{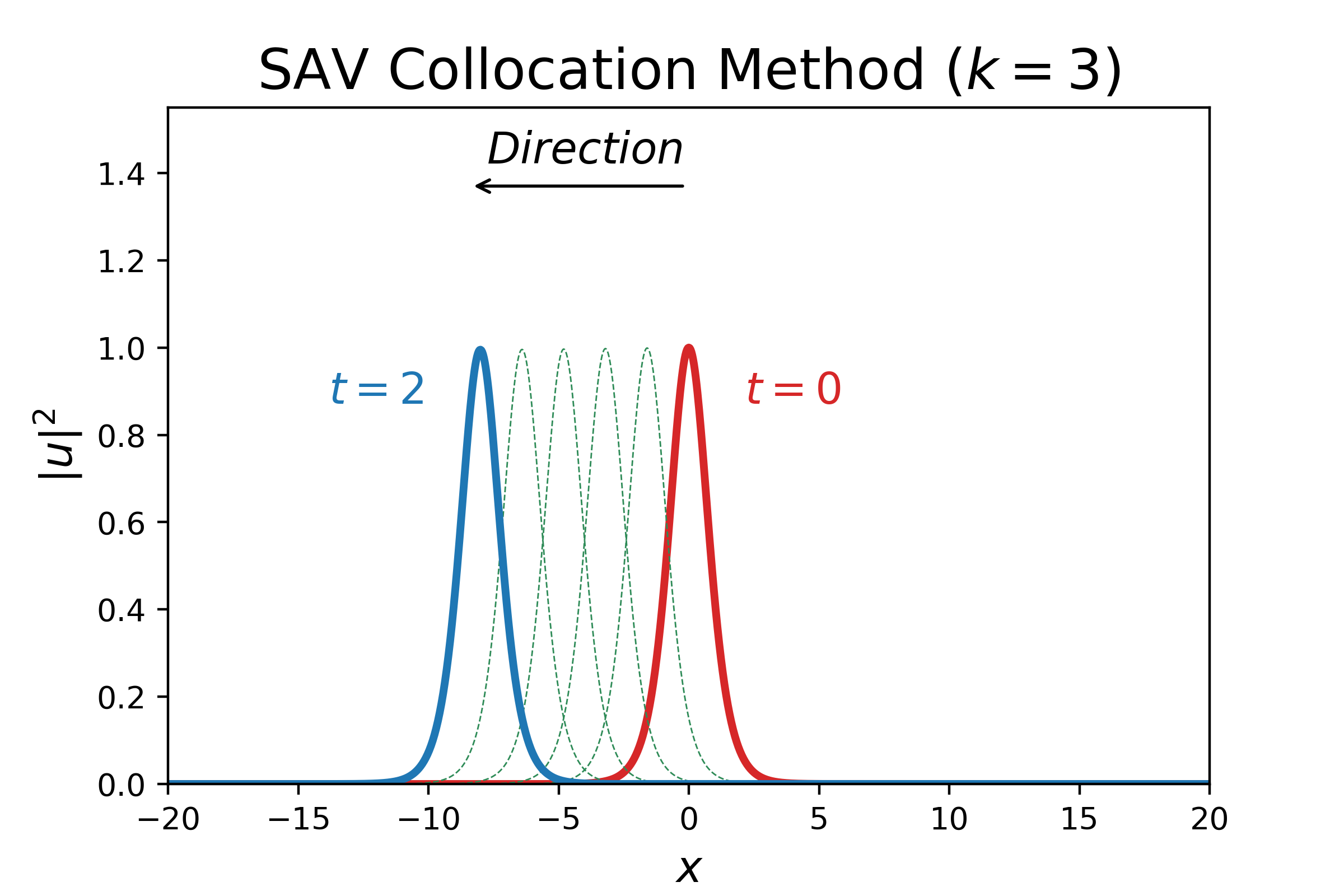}}
}
\centerline{
\subfigure[]{
\label{fig_solotion_methods2:c}
\includegraphics[width=2.6in]{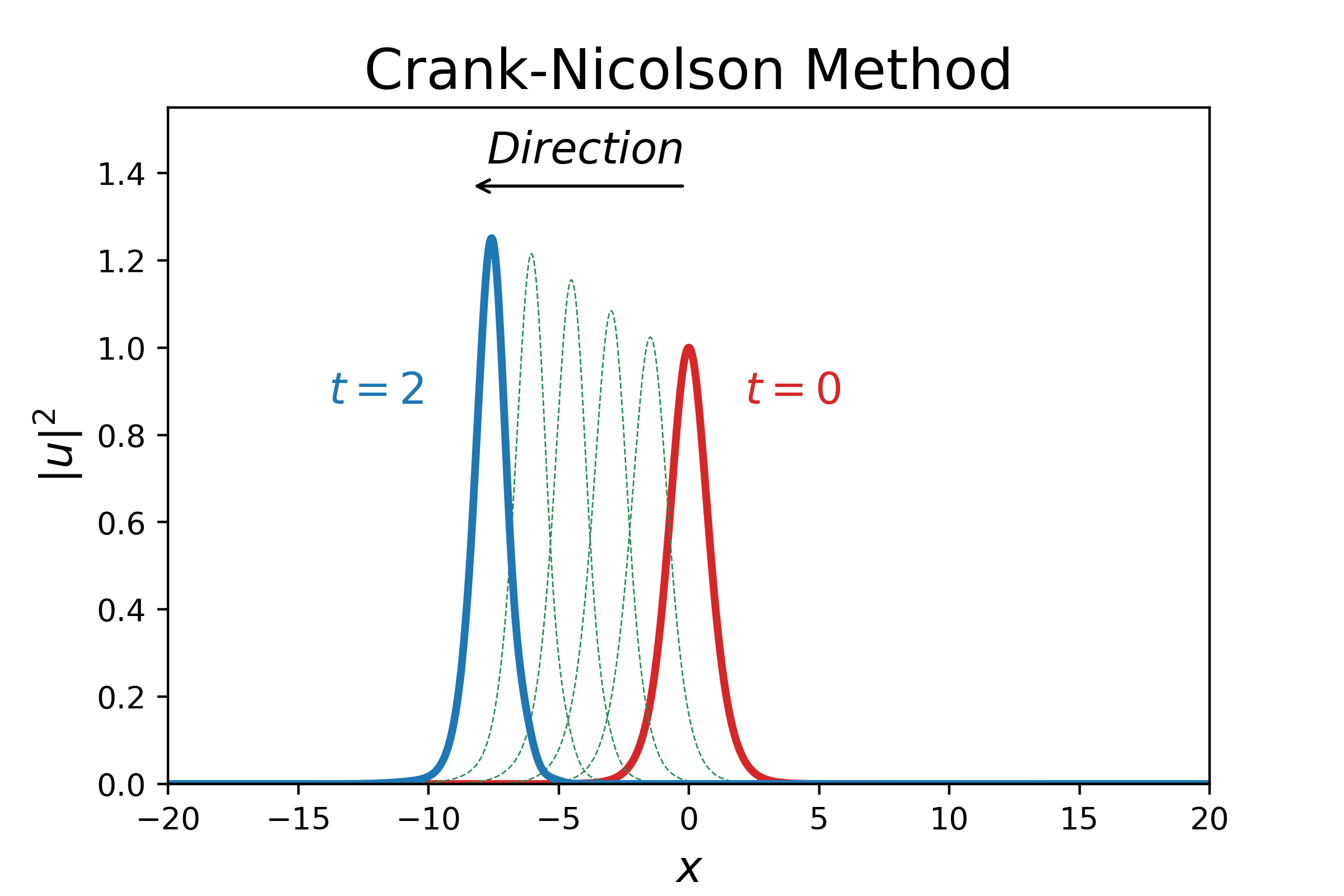}}
\subfigure[]{
\label{fig_solotion_methods2:d}
\includegraphics[width=2.6in]{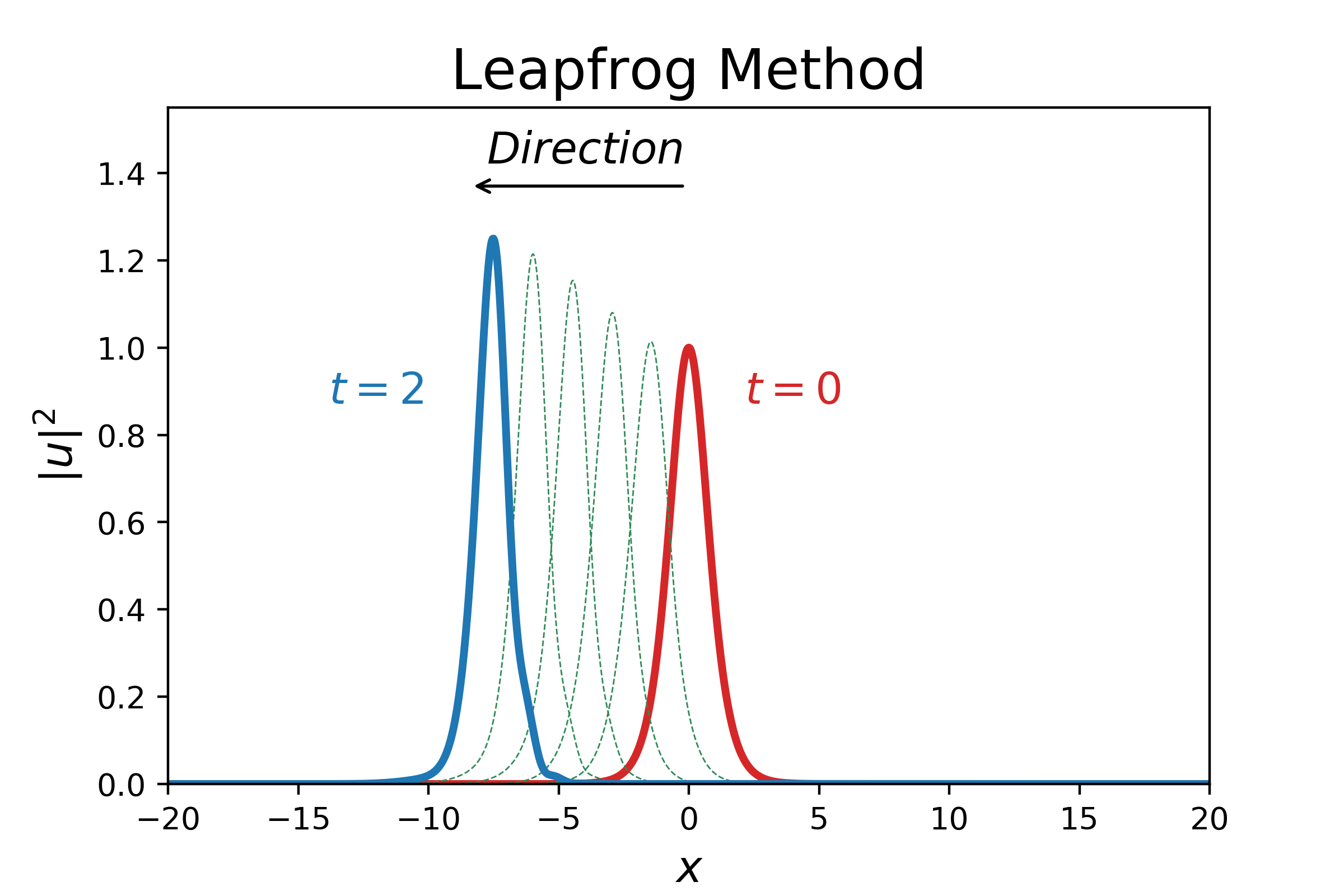}}
}
\vspace{-10pt}
\caption{Soliton propagation when $t\in [0, 2]$: numerical solutions with $p=1$, $M = 1200$ and $\Delta t = 0.1$. }
\label{fig_solotion_methods2}
%
\centerline{
\subfigure[]{
\label{fig_solotion_methods:a}
\includegraphics[width=2.6in]{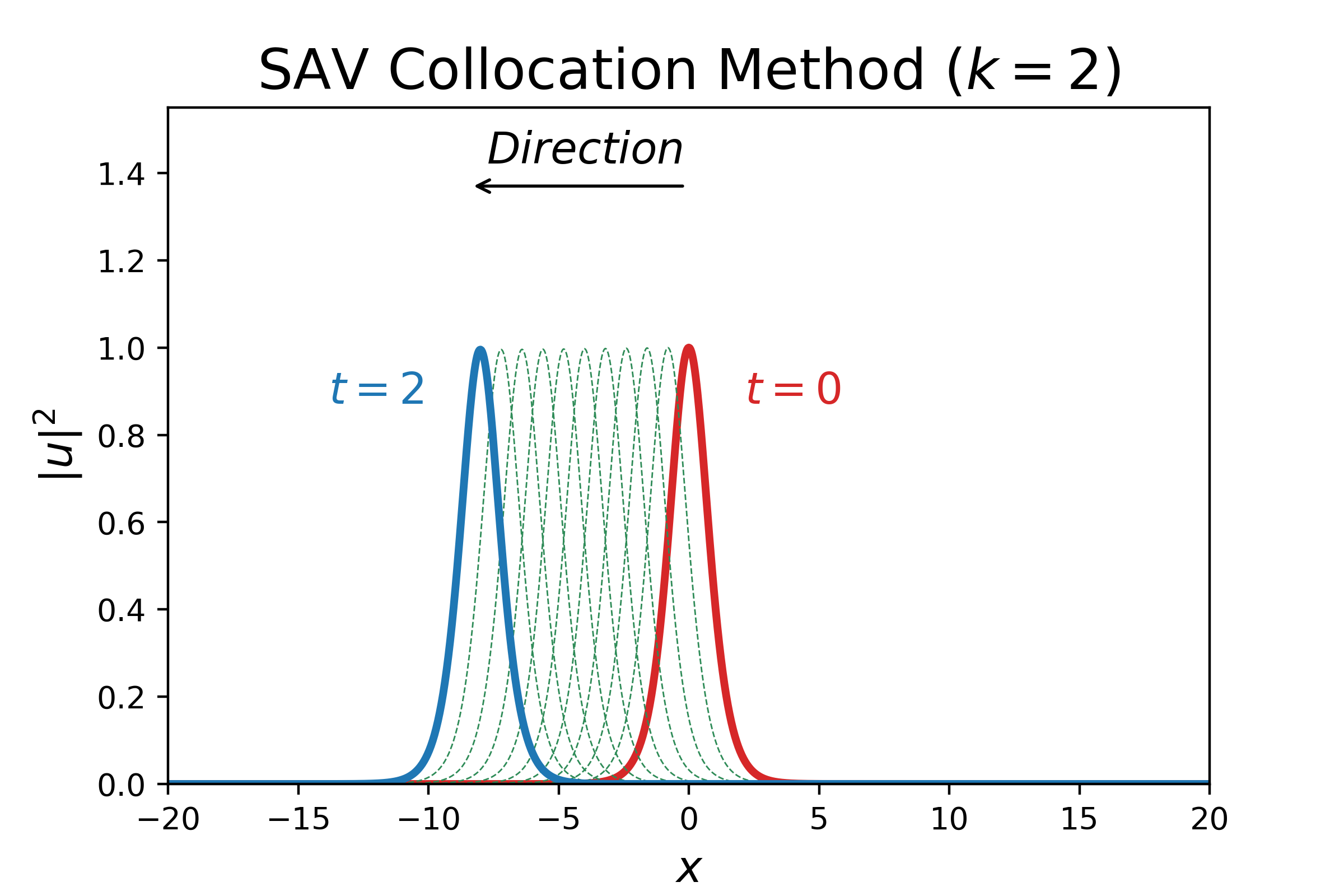}}
\subfigure[]{
\label{fig_solotion_methods:b}
\includegraphics[width=2.6in]{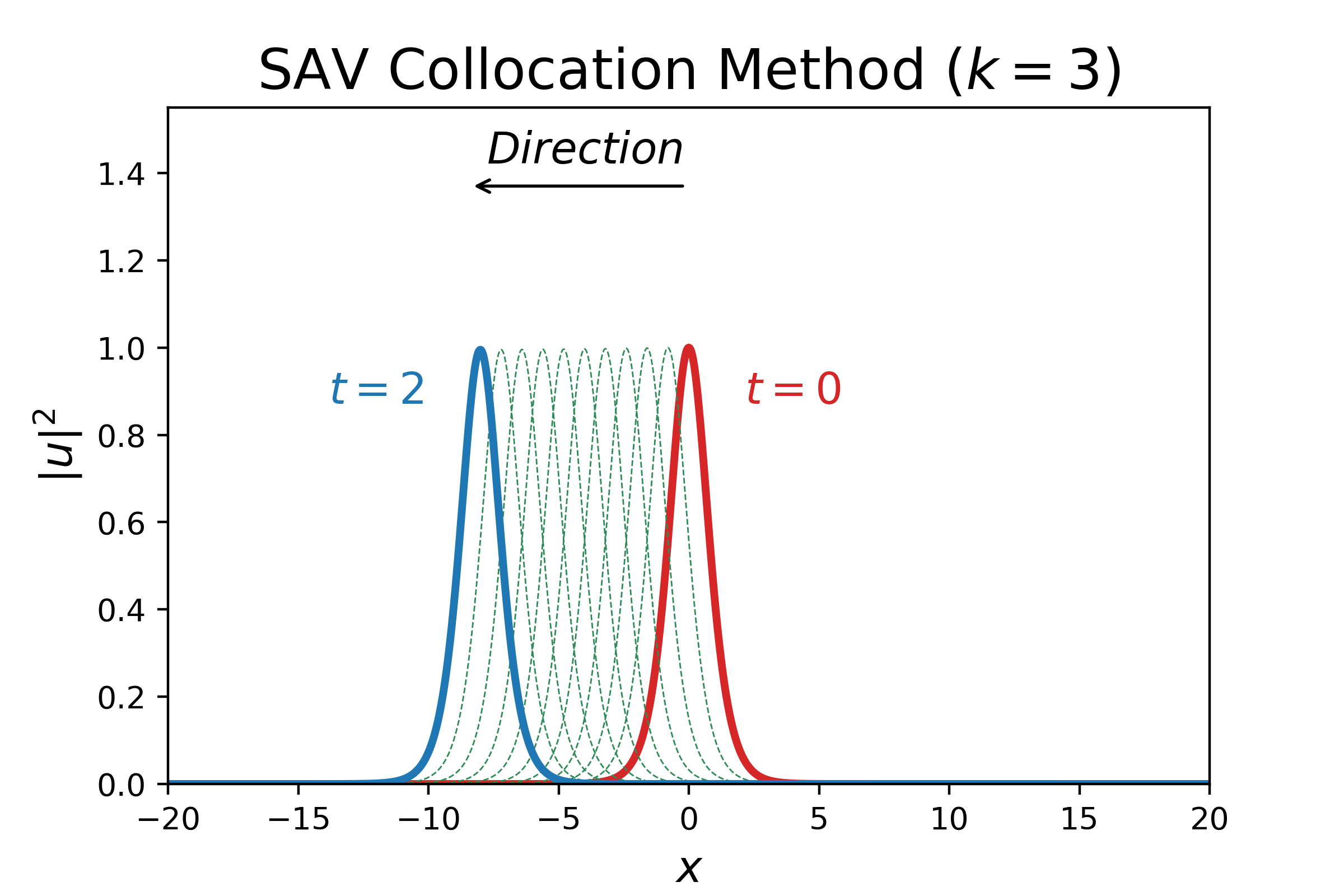}}
}
\centerline{
\subfigure[]{
\label{fig_solotion_methods:c}
\includegraphics[width=2.6in]{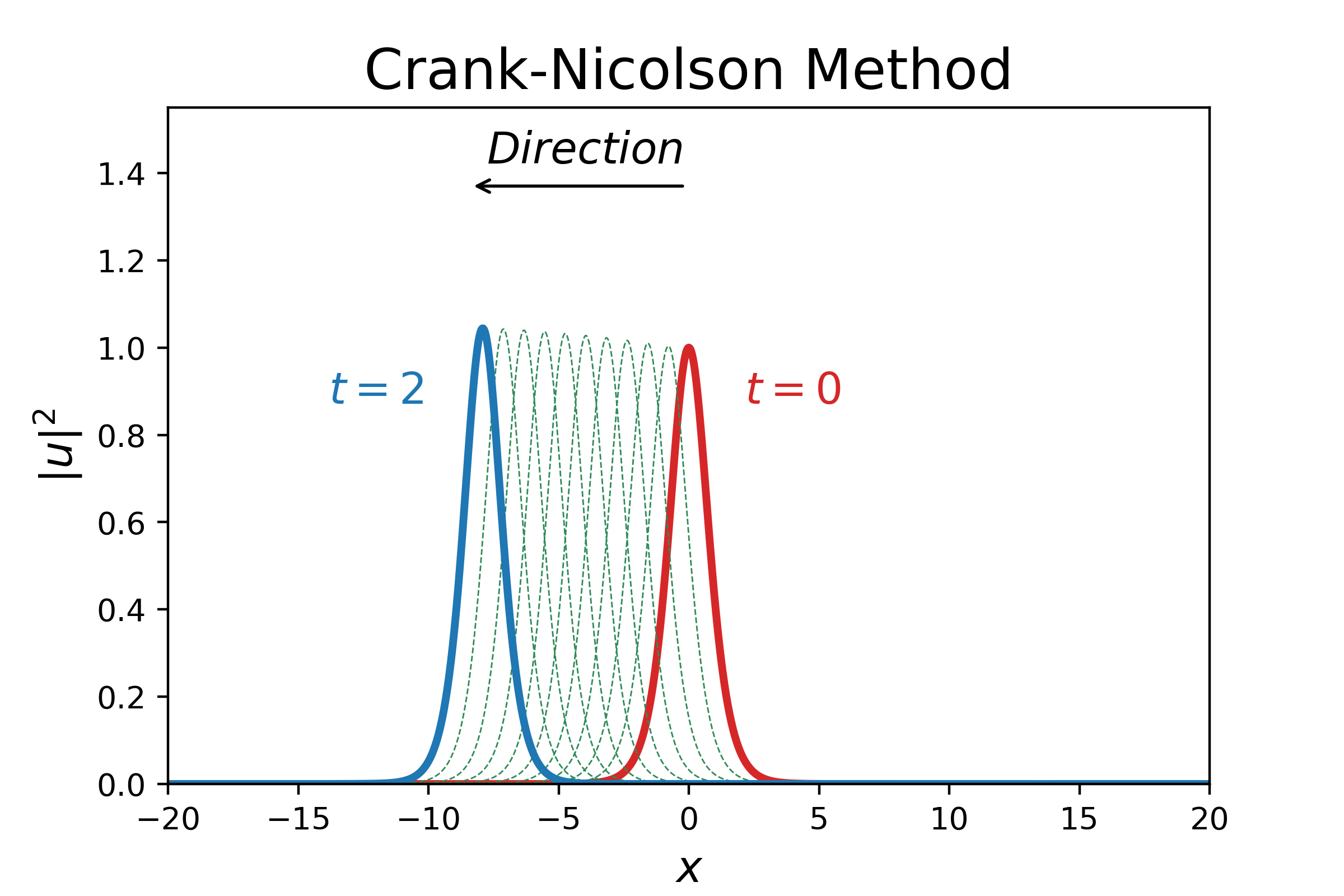}}
\subfigure[]{
\label{fig_solotion_methods:d}
\includegraphics[width=2.6in]{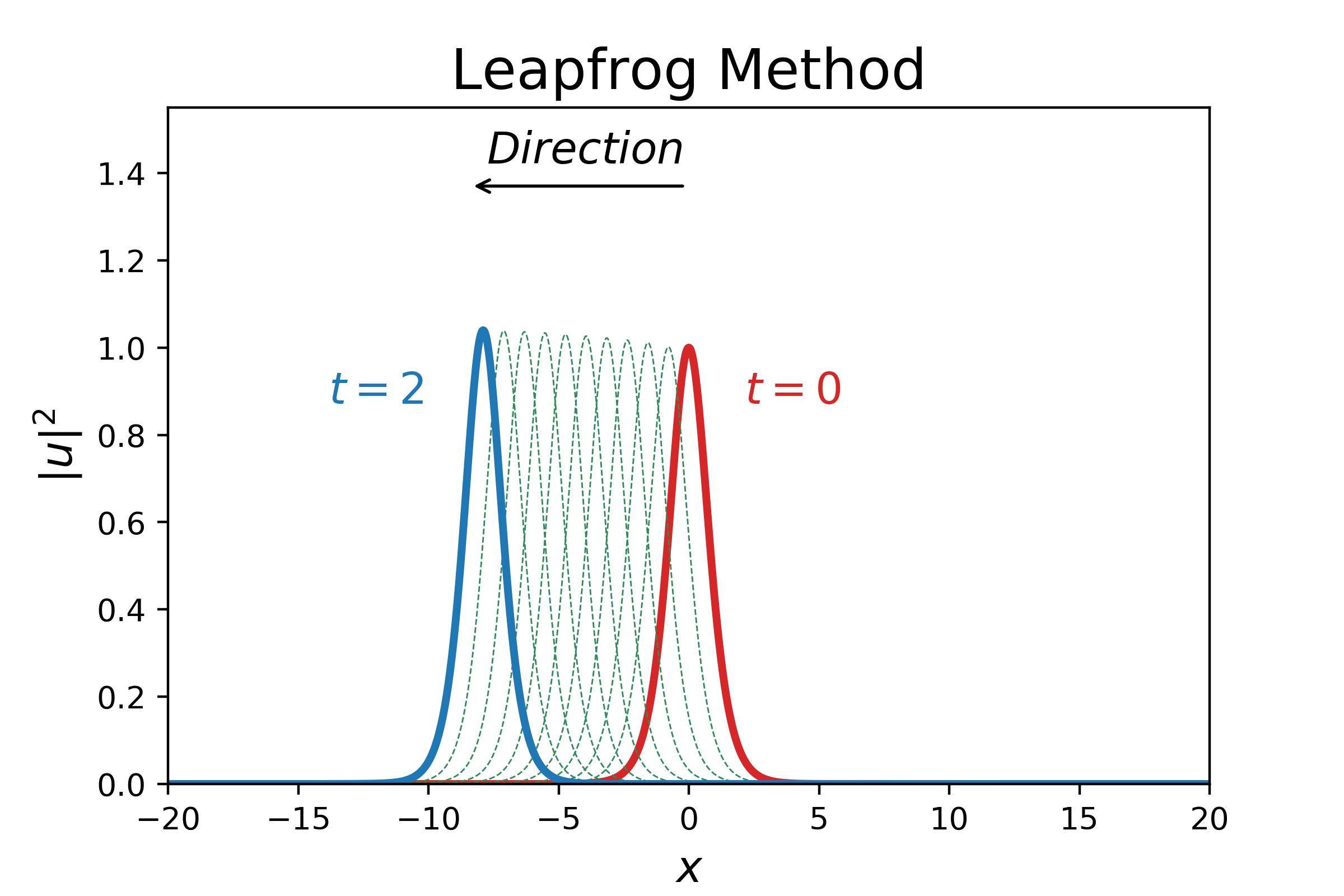}}
}
\vspace{-10pt}
\caption{Soliton propagation when $t\in [0, 2]$: numerical solutions with $p=1$, $M = 1200$ and $\Delta t = 0.05$.}
\label{fig_solotion_methods}
\end{figure}

\subsection{Capability of solving focusing nonlinearity}
We consider the cubic nonlinear Schr\"odinger equation 
\begin{align}\label{nls_sol_2d}
\begin{aligned}
\i \partial_t u - \partial_{xx} u - \partial_{yy} u + 2|u|^2 u 
&=0     &&\qquad\mbox{in}\,\,\, \Omega\times(0,T],\\
u|_{t=0} &=u_0 &&\qquad \mbox{in}\,\,\,\Omega,
\end{aligned}
\end{align}
in two-dimensional space $\Omega =[0, 1] \times [0, 1]$ subject to the periodic boundary condition. We choose $u_0 = \exp(2 \pi\i (x+y))$ 
so that the exact solution is given by 
\begin{align}\label{exact_sol_2d}
u(x,t) = \exp(\i (2 \pi  x  + 2 \pi  y + (2 + 8\pi^2)t)),
\end{align}
which admits a progressive plane wave solution; see \cite{Xu2005Local}. 

We solve problem \eqref{nls_sol_2d} by the proposed method \eqref{GL} and compare the numerical solutions with the exact solution \eqref{exact_sol_2d}. Newton's method is used to solve the nonlinear system. The iteration is stopped when the error is below $10^{-10}$. 

The time discretization errors are presented in Table \ref{table_time_errors_p3_2d}, where we have used finite elements of degree $3$ with a sufficiently spatial mesh $h=1/80$ so that the error from spatial discretization is negligibly small in observing the temporal convergence rates.  
From Table \ref{table_time_errors_p3_2d} we see that the error of time discretization 
is $O(\tau^{k+1})$, which is consistent with the result proved in Theorem \ref{THM:main}. 

The spatial discretization errors are presented in Table \ref{table_space_errors_gauss3_2d}, where we have chosen $k=3$ with a sufficiently small time stepsize $\tau=1/1000$ so that the time discretization 
error is negligibly small compared to the spatial error.  
From Table \ref{table_space_errors_gauss3_2d} we see that the spatial discretization errors are $O(h^{p})$ in the $H^1$ norm. This is also consistent with the result proved in Theorem \ref{THM:main}. 

\begin{table}[htp]\centering\small\footnotesize
\caption{Time discretization errors of the proposed method, with $h=\frac{1}{80}$ and $T = 0.1$. }
\setlength{\tabcolsep}{7mm}{
\begin{tabular}{crcc}
\toprule
$k$ &$\tau$&\multicolumn{2}{c}{$p = 3$}\\
\cmidrule(lr){3-4}
&&
$\|u(x, t) - u_h(x, t)\|_{L^\infty(0,T;H^1)}$ &order\\
\midrule
\multirow{5}{*}{$2$}    
&  1/460&     5.0023E--04&     --\\
&  1/480&     4.3780E--04& 3.1321\\
&  1/500&     3.8572E--04& 3.1027\\
&  1/520&     3.4198E--04& 3.0686\\
&  1/540&     3.0504E--04& 3.0290\\
\\
\multirow{5}{*}{$3$}  
&   1/60&    1.6206E--02&     --\\
&   1/80&    4.9792E--03& 4.1022\\
&  1/100&    2.0173E--03& 4.0490\\
&  1/120&    9.6960E--04& 4.0183\\
&  1/140&    5.2530E--04& 3.9761\\
\\
\multirow{5}{*}{$4$} 
&   1/30&    3.6941E--02&     --\\
&   1/40&    8.0993E--03& 5.2750\\
&   1/50&    2.5534E--03& 5.1731\\
&   1/60&    1.0078E--03& 5.0989\\
&   1/70&    4.6554E--04& 5.0104\\
\bottomrule
\end{tabular}}
\label{table_time_errors_p3_2d}
\vspace{20pt}

\caption{Spatial discretization errors of the proposed method, with $\tau=\frac{1}{1000}$ and $T = 0.1$. }
\setlength{\tabcolsep}{7mm}{
\begin{tabular}{crccc}
\toprule
$p$&$h$
&\multicolumn{2}{c}{$k = 3$}\\
\cmidrule(lr){3-4}
&&$\|u(x, t) - u_h(x, t)\|_{L^\infty(0,T;H^1)}$ &order\\
\midrule
\multirow{5}{*}{$1$}  
&  1/70&       5.6297E--01&     --\\
&  1/80&       4.8304E--01& 1.1466\\
&  1/90&       4.2346E--01& 1.1178\\
& 1/100&       3.7726E--01& 1.0964\\
& 1/110&       3.4035E--01& 1.0803\\
\\
\multirow{5}{*}{$2$}  
&  1/10&       4.9467E--01&     --\\
&  1/15&       2.0992E--01& 2.1141\\
&  1/20&       1.1748E--01& 2.0178\\
&  1/25&       7.5177E--02& 2.0005\\
&  1/30&       5.2233E--02& 1.9972\\
\\
\multirow{5}{*}{$3$} 
&  1/12&       2.1955E--02&     --\\
&  1/14&       1.3738E--02& 3.0412\\
&  1/16&       9.1747E--03& 3.0236\\
&  1/18&       6.4327E--03& 3.0144\\
&  1/20&       4.6849E--03& 3.0092\\
\bottomrule
\end{tabular}}
\label{table_space_errors_gauss3_2d}
\end{table}

The evolution of mass and SAV energy of the numerical solutions is presented in Figure \ref{fig_mass_energy_2d} with $\tau=0.2$ and $h=0.2$. It is shown that 
$$
{\rm mass} = 1.004814962453 + O(10^{-12})
\quad\mbox{and}\quad 
{\rm SAV\,\, energy} =  80.45628698537 +  O(10^{-11}) ,
$$ 
which are much smaller than the error of the numerical solutions, as shown in Figure \ref{fig_mass_energy_error_2d}. This shows the effectiveness of the proposed method in preserving mass and energy (independent of the error of numerical solutions). The number of iterations at each time level is presented in Figure \ref{fig_iter_num_2d} to show the effectiveness of the Newton's method.

\begin{figure}[htp]
\centerline{
\includegraphics[width=2.6in]{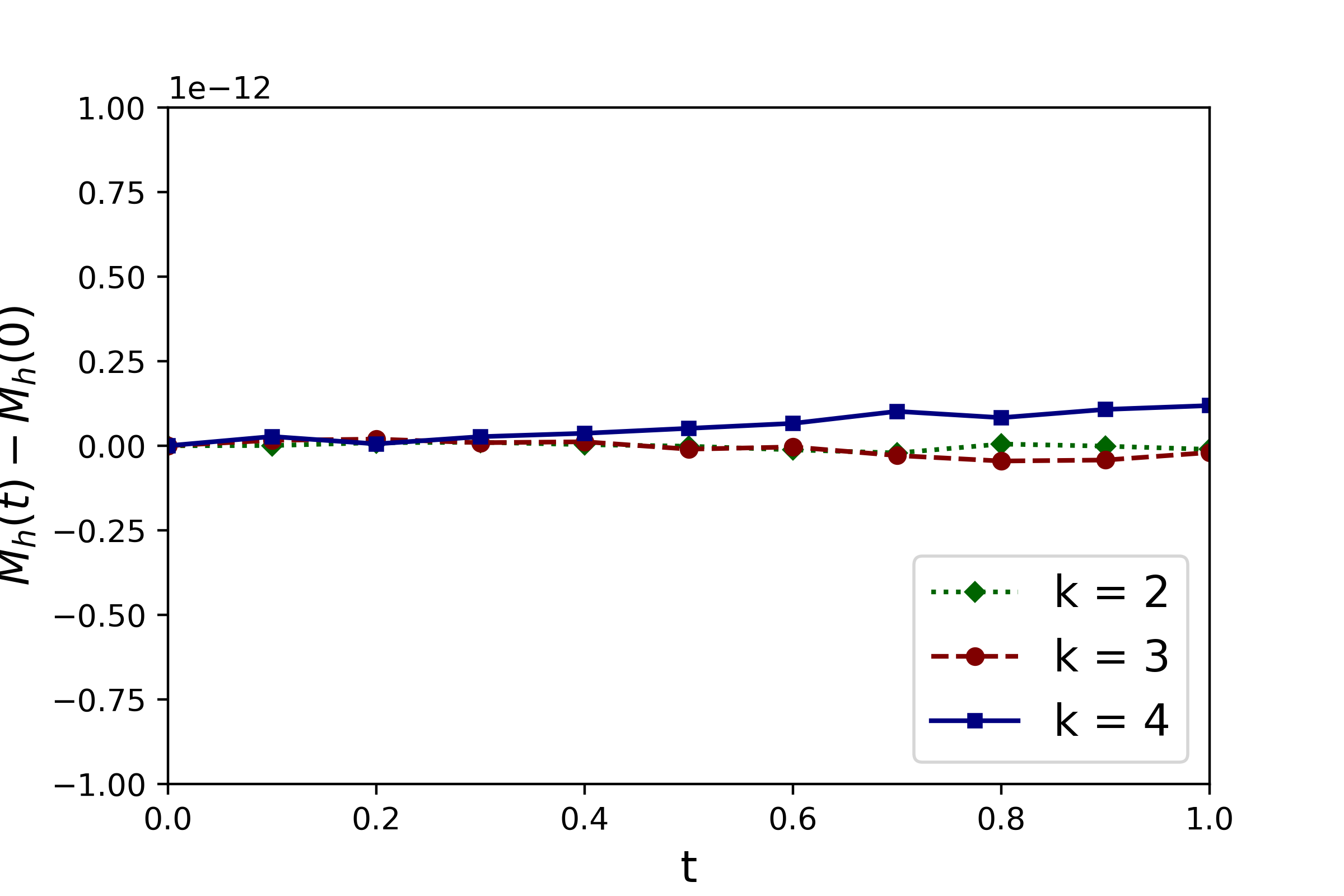}\,
\includegraphics[width=2.6in]{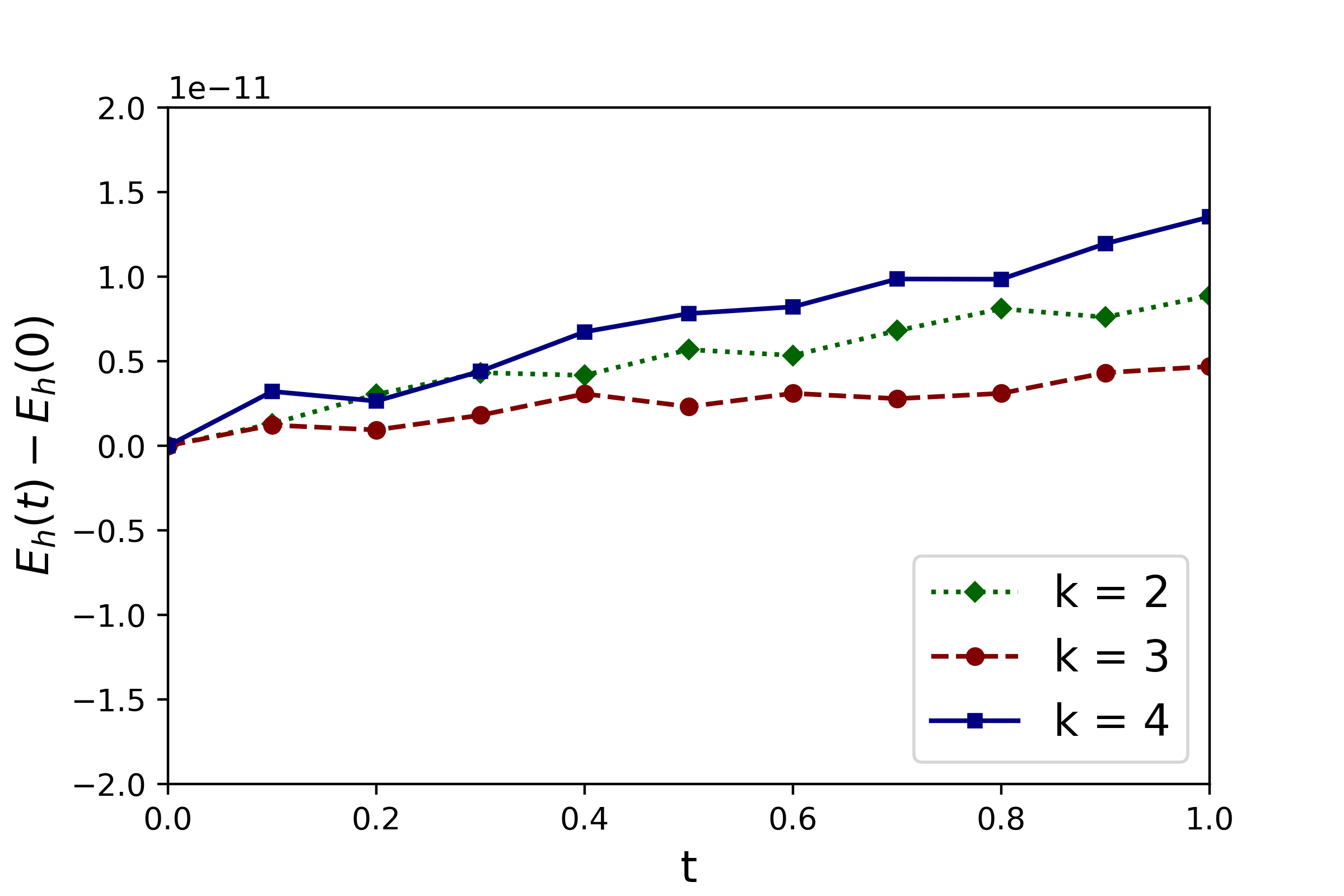}
}
\vspace{-8pt}
\caption{Evolution of mass $M_h(t)-M_h(0)$ and SAV energy $E_h(t)-E_h(0)$, with $p=3$ and $\tau=h=0.2$.}
\label{fig_mass_energy_2d}
\smallskip 
\centerline{
\includegraphics[width=2.6in]{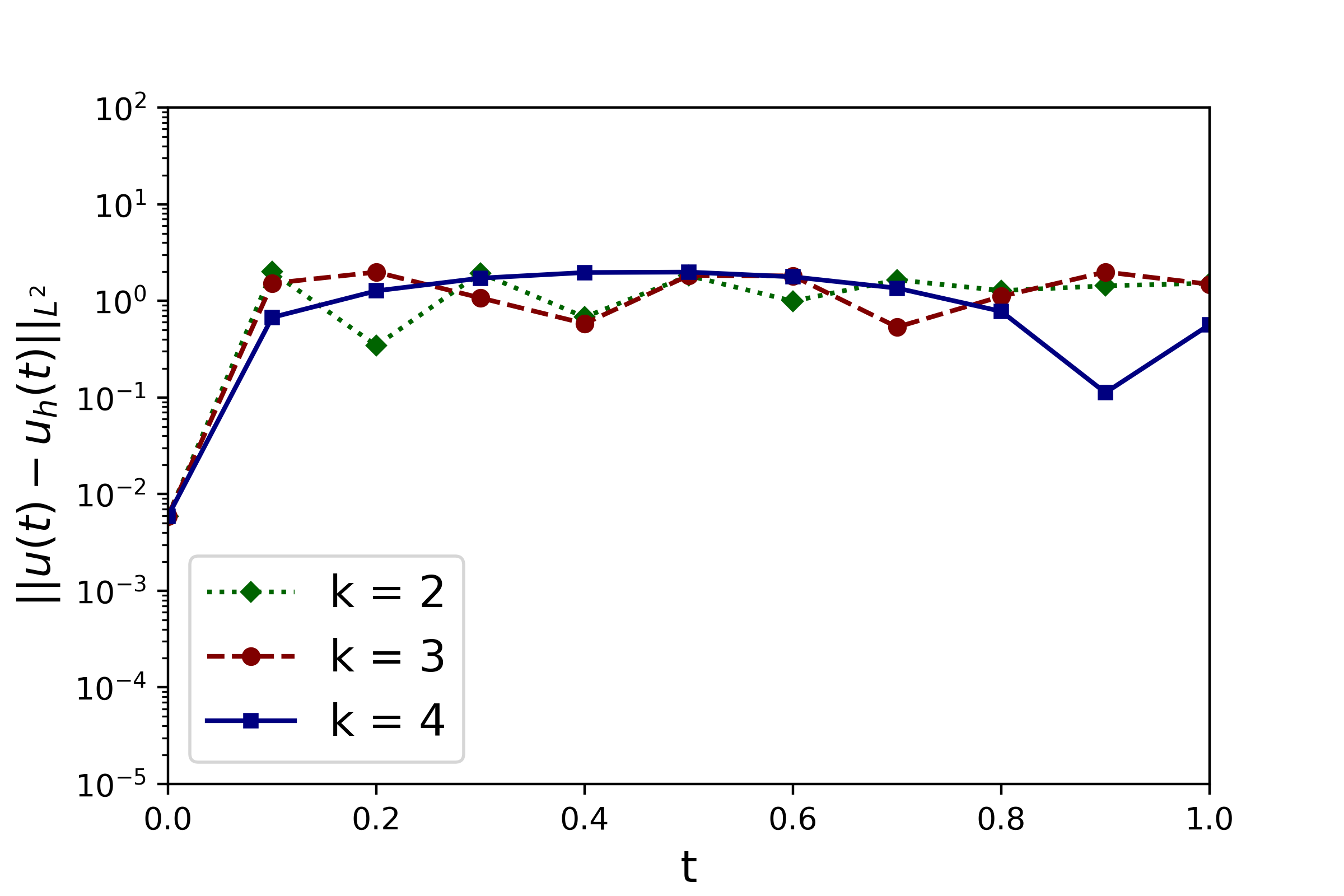}\,
\includegraphics[width=2.6in]{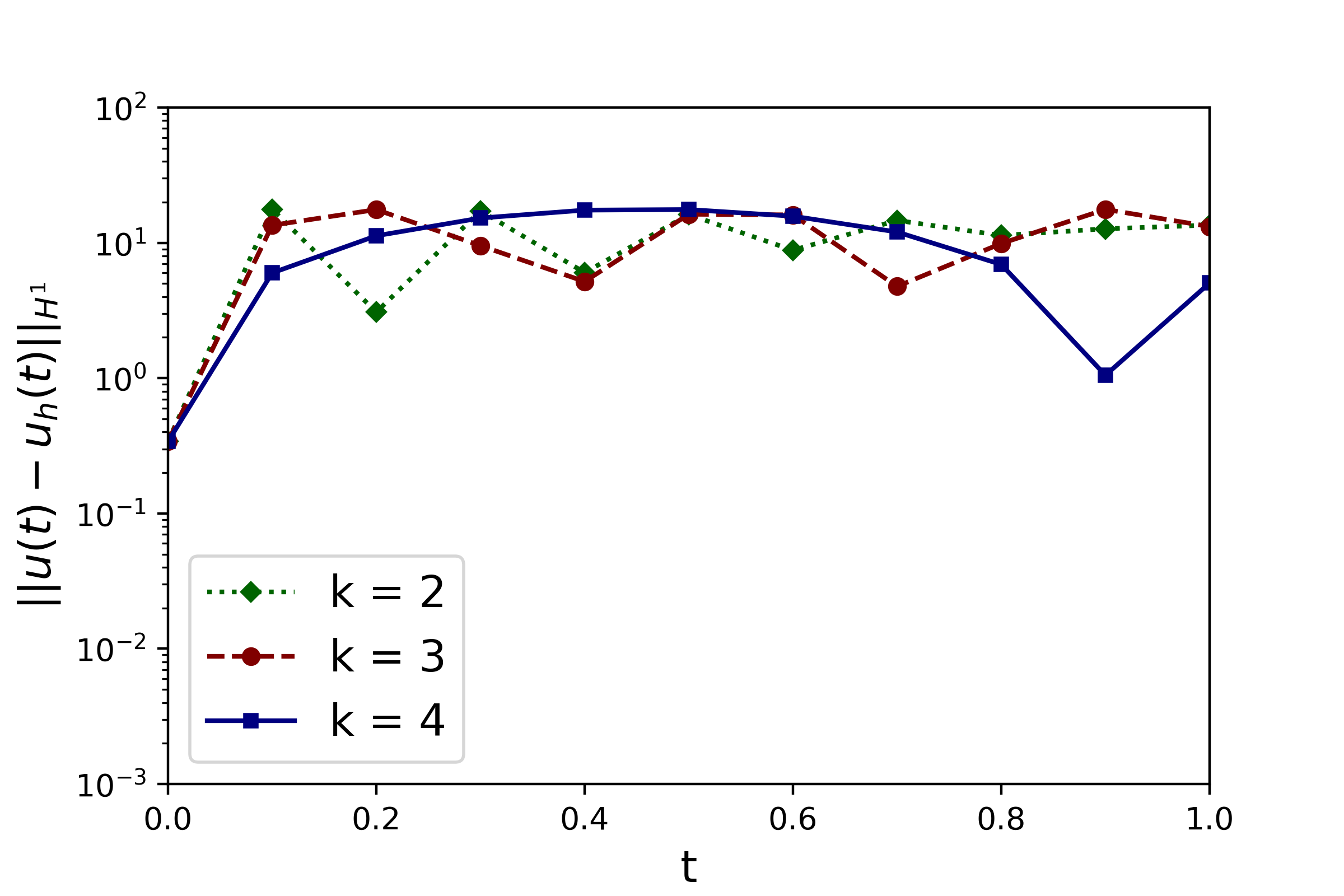}
}
\vspace{-8pt}
\caption{Evolution of error of the numerical solution, with $p=3$ and $\tau=h=0.2$.}
\label{fig_mass_energy_error_2d}
\centering
\includegraphics[width=2.6in]{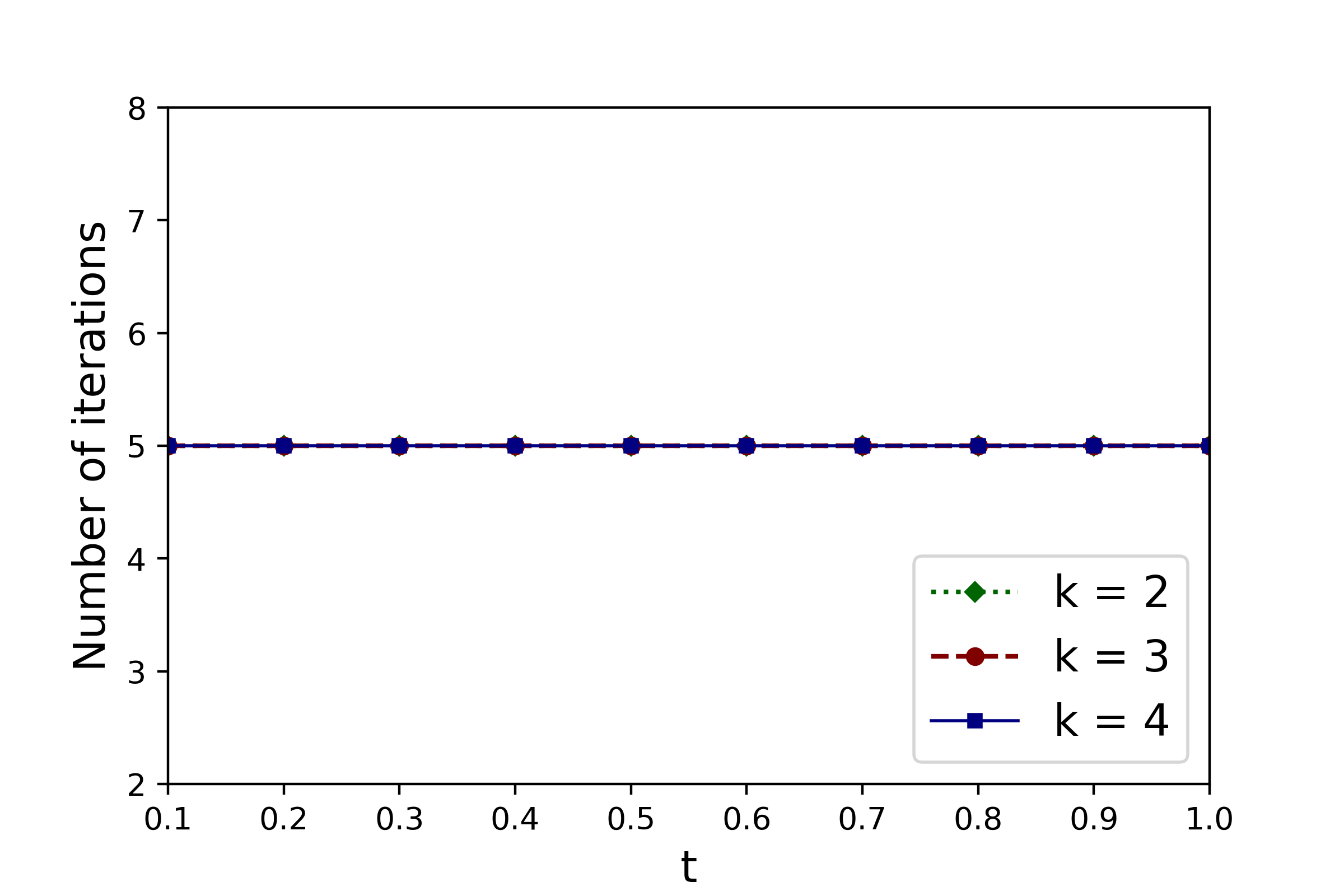}
\vspace{-8pt}
\caption{Number of iterations at each time level, with $p=3$ and $\tau=h=0.2$.}
\label{fig_iter_num_2d}
\end{figure}




\bibliographystyle{abrv}
\bibliographystyle{plain}

\end{document}